\title{Classification of Module Categories for $SO(3)_{2m}$}
\author{
{\sc David E.\ Evans and Mathew Pugh}\\
 {\footnotesize School of Mathematics, Cardiff University,}\\  {\footnotesize Senghennydd Road, Cardiff CF24 4AG, Wales, U.K.}
}
\date{\today}

\documentclass[12pt]{article}

\usepackage{amsmath,amssymb,amsthm}
\usepackage{graphicx}
\usepackage{framed}
\usepackage{color}
\usepackage[usenames,dvipsnames]{xcolor}
\usepackage{multirow}
\usepackage{cancel}
\usepackage{array}
\usepackage{arydshln}
\usepackage[all]{xy}
\usepackage{url}
\usepackage[normalem]{ulem}

\theoremstyle{definition}
\newtheorem{Def}{Definition}[section]
\newtheorem{Prop}[Def]{Proposition}
\newtheorem{Lemma}[Def]{Lemma}

\newtheorem{Thm}[Def]{Theorem}
\newtheorem{Rem}[Def]{Remark}

\newcommand\bbR{\mathbb{R}}
\newcommand\bbC{\mathbb{C}}
\newcommand\bbZ{\mathbb{Z}}
\newcommand\bbN{\mathbb{N}}
\newcommand\bbT{\mathbb{T}}
\newcommand\ben{\begin{enumerate}}
\newcommand\een{\end{enumerate}}
\newcommand\disp{\displaystyle}

\begin{document}
\maketitle

\begin{abstract}
The main goal of this paper is to classify $\ast$-module categories for the $SO(3)_{2m}$ modular tensor category. This is done by classifying $SO(3)_{2m}$ nimrep graphs and cell systems, and in the process we also classify the $SO(3)$ modular invariants. There are module categories of type $\mathcal{A}$, $\mathcal{E}$ and their conjugates, but there are no orbifold (or type $\mathcal{D}$) module categories.
We present a construction of a subfactor with principal graph given by the fusion rules of the fundamental generator of the $SO(3)_{2m}$ modular category. We also introduce a Frobenius algebra $A$ which is an $SO(3)$ generalisation of (higher) preprojective algebras, and derive a finite resolution of $A$ as a left $A$-module along with its Hilbert series.
\end{abstract}

\bigskip

{\footnotesize
\tableofcontents
}

\section{Introduction} \label{sect:intro}

The Verlinde ring in conformal field theory can be described by a modular tensor category either in the language of the representation theory of a conformal net or of a vertex operator algebra, particularly in the case of Wess-Zumini-Witten models associated to the positive energy representations of the loop group of a compact Lie group. The full conformal field theory is described by a gluing of a left and right Verlinde ring, the most basic part of which is encoded in a modular invariant partition function, a matrix of non-negative integers which is invariant under the action of the modular group.
Not all such modular invariants arise from a conformal field theory -- the physical ones can be described in the conformal net picture by braided subfactors or module categories over the Verlinde ring. It is therefore of interest to understand module categories over a modular tensor category, and which modular invariants they correspond to.

Here we look at the question of classifying module categories for the modular tensor category corresponding to the non-simply connected, compact Lie group $SO(3)$ at level $2m$. The question could also be phrased in terms of quantum subgroups, since the module categories over the tensor category of the representation of a compact group is described by subgroups, together with an element of two-cohomology of the subgroup (see e.g. \cite[$\S$7.12]{etingof/gelaki/nikshych/ostrik:2015}). The corresponding questions have been addressed for $SU(2)$ and $SU(3)$ (see references below), and the doubles of finite groups \cite{ostrik:2003ii}.

By $\alpha$-induction, any module category for the modular tensor category $\mathcal{C}^m$ yields a modular invariant partition function $Z_{\lambda,\mu} = \langle \alpha_{\lambda}^+, \alpha_{\mu}^- \rangle$ \cite{bockenhauer/evans/kawahigashi:1999, bockenhauer/evans:2000, evans:2003, fuchs/runkel/schweigert:2002}.
A nimrep (non-negative matrix integer representation of the fusion rules) is an assignment of a matrix $G_{\lambda}$ for each simple object $\lambda$ in the modular tensor category for $SO(3)_{2m}$, with non-negative integer entries, which satisfies the fusion rules of $SO(3)_{2m}$, i.e. $G_{\lambda} G_{\mu} = \sum_{\nu} N_{\lambda \mu}^{\nu} G_{\nu}$.
By a standard argument (see e.g. \cite[p.425]{evans/kawahigashi:1998}), the $G_{\lambda}$'s can be simultaneously diagonalised. The eigenvalues of $G_{\lambda}$ given by $S_{\mu \lambda}/S_{\mu 0}$ for $\mu$ running in some multi-set (possibly with multiplicities) which is described by the diagonal part $Z_{\mu,\mu}$ of the modular invariant associated with the module category, where the $S$-matrix is one of the generators of the modular group and $0$ here denotes the object corresponding to the vacuum.
Our first step is therefore to classify modular invariants, which we require to be normalised, i.e. $Z_{00}=1$, and compatible nimrep graphs. This was done for $SU(2)$ in \cite{cappelli/itzykson/zuber:1987ii}, and for $SU(3)$ in \cite{gannon:1994, di_francesco/zuber:1990, ocneanu:2002}.

Just as the fusion rules alone are not enough to determine a fusion category but there are additional structural constants, the nimrep is not enough to determine a module category, or even to show its existence. The other ingredient is a system of Ocneanu cells, which define $6j$-symbols, (self-)connections or Boltzmann weights without spectral parameter. A cell system consists of a system of complex numbers satisfying certain local equations of cohomological nature \cite{ocneanu:2002}. These equations can be derived from diagrammatic axiomatisations of the representation theory of the Lie group, which for $SU(2)$ is due to Kauffman \cite{kauffman:1987}, and the cells are simply the entries of the Perron-Frobenius eigenvector. Thus for $SU(2)$ the module categories are classified by their nimreps, the $ADE$ Dynkin diagrams (see also \cite{kirillov/ostrik:2002, ostrik:2003, etingof/ostrik:2004}). For $SU(3)$ the equations follow from Kuperberg's $A_2$ spider \cite{kuperberg:1996}, the existence of cell systems for the $SU(3)$ nimreps was claimed by Ocneanu \cite{ocneanu:2000ii}, and they were explicitly computed in \cite{evans/pugh:2009i}.

A related question is whether any modular invariant $Z_{\mathcal{G}}$ with associated nimrep $\mathcal{G}$ can be realised by a braided subfactor $N \subset M$, where $N$, $M$ are hyperfinite type $\mathrm{III}$ factors and the Verlinde algebra is realised by systems of endomorphisms ${}_N \mathcal{X}_N$ of $N$, and the action of the system ${}_N \mathcal{X}_N$ on the $N$-$M$ sectors ${}_N \mathcal{X}_M$ gives the nimrep $\mathcal{G}$. For $SU(2)$ this was answered in \cite{ocneanu:2000ii, ocneanu:2002, xu:1998, bockenhauer/evans:1999i, bockenhauer/evans:1999ii, bockenhauer/evans/kawahigashi:1999, bockenhauer/evans/kawahigashi:2000}, for $SU(3)$ Ocneanu claimed \cite{ocneanu:2000ii, ocneanu:2002} that all $SU(3)$ modular invariants were realised by subfactors and this was shown in \cite{xu:1998, bockenhauer/evans:1999i, bockenhauer/evans:1999ii, bockenhauer/evans/kawahigashi:1999, bockenhauer/evans:2001, bockenhauer/evans:2002, evans/pugh:2009i, evans/pugh:2009ii}.

By the work of Toledano-Laredo \cite{toledano_laredo:1999} (see also \cite{carey/wang:2008, freed/hopkins/teleman:2008, braun/schafer-nameki:2008}) the loop group of $SO(3)=SU(2)/\bbZ_2$ is defined only at even levels of $SU(2)$ and is additionally characterised by a character $\chi \in \mathrm{Hom}(\bbZ_2,\bbT) \cong \bbZ_2$. For $LSU(2)$ level $2k$, we denote the corresponding level of $LSO(3)$ by $k$.
For $(k,\chi) = (2m,+1)$, the irreducible positive energy projective representations $\lambda_j$ of $LSO(3)$ are given as $LSU(2)$-modules by
\begin{equation} \label{eqn:branching_coeffs}
\lambda_0 := \lambda^{(0)} \oplus \lambda^{(4m)}, \;\; \lambda_1 := \lambda^{(2)} \oplus \lambda^{(4m-2)}, \ldots \;\; \lambda_{m-1} := \lambda^{(2m-2)} \oplus \lambda^{(2m+2)}, \;\; \lambda_{m}^{\pm} := \lambda^{(2m)},
\end{equation}
where $\lambda^{(l)}$, $l=0,1,2,\dots,2k$, denote all the irreducible positive energy representations of $LSU(2)$ at level $2k$, and $\lambda_{m}^{\pm}$ are both equal to $\lambda_{2m}$ as $LSU(2)$-modules with the group of discontinuous loops corresponding to $\bbZ_2$ acting via the the character $\pm 1$.

The Verlinde ring of positive energy representations of $LSO(3)$ admits a fusion product only for even levels $2m$ ($LSU(2)$ level $4m$) and $\chi = +1$ \cite{carey/wang:2008, freed/hopkins/teleman:2008}.
The positive energy representations in $LSO(3)$ at level $2m$ have fusion rules $\lambda \mu = \bigoplus_{\nu} N_{\lambda \mu}^{\nu} \nu$, where for $\lambda = \rho$ corresponding to the fundamental generator of $LSO(3)$ the fusion matrix $N_{\rho} = [N_{\rho \mu}^{\nu}]_{\mu,\nu}$ is the adjacency matrix for the even part of $D_{2m+2}$ (see the diagram on the left in Figure \ref{fig-Dk+2_fusion_graph-f2}). These figures are also in \cite{baver/gepner:1996}, along with the $S$-matrix of the corresponding conformal field theory.

Existence of a modular category with the $SO(3)$ fusion rules is known by $\alpha$-induction \cite[$\S$3]{bockenhauer/evans:1999i}. Equation \eqref{eqn:branching_coeffs} gives the branching coefficients $b_{\lambda_j,\lambda^{(l)}}$ of the $SO(3)_{2m}$ representation $\lambda_j$ into $SU(2)_{4m}$ representations $\lambda^{(l)}$, i.e. $\lambda_j = \bigoplus_l b_{\lambda_j,\lambda^{(l)}} \lambda^{(l)}$. The branching coefficient matrix intertwines the modular data for the two theories \cite{bockenhauer/evans:2000}:
\begin{equation} \label{eqn:B-ST}
bS^{SO(3)} = S^{SU(2)}b, \qquad bT^{SO(3)} = T^{SU(2)}b,
\end{equation}
where $S^G$, $T^G$ denote the modular $S$, $T$-matrices for $G=SU(2),SO(3)$. Equations \eqref{eqn:B-ST} 
and the requirement that $S^G$, $T^G$ are unitary are sufficient to uniquely determine $T^{SO(3)}$, and to uniquely determine $S^{SO(3)}$ apart from $S^{SO(3)}_{\lambda,\mu}$ for $\lambda,\mu \in \{ \lambda_m^{\pm} \}$.
The identity
$$S^{SO(3)}_{\lambda\mu} = \overline{T^{SO(3)}_{0,0}} T^{SO(3)}_{\lambda,\lambda} T^{SO(3)}_{\mu,\mu} \sum_{\nu} N_{\lambda \mu}^{\nu} T^{SO(3)}_{\nu,\nu} \overline{S^{SO(3)}_{\nu,0}},$$
which is derived from the Verlinde formula and $(S^{SO(3)} T^{SO(3)})^3 = (S^{SO(3)})^2$, then uniquely determines the remaining entries of $S^{SO(3)}$.
The above procedure is known as a fixed point resolution for simple currents \cite{schellekens/yankielowicz:1989, schellekens/yankielowicz:1990, fuchs/schellekens/schweigert:1996}.
Note that $S^2 = I$ for $m$ even, but $S^2=C \neq I$ for $m$ odd, c.f. \cite[$\S$3.5]{izumi:1991}.

This paper is organised as follows. We present a construction of the modular tensor category for $SO(3)_{2m}$ in Section \ref{sect:construct-cat}, which was first presented in \cite{pugh:2017}.
This category has also recently appeared in \cite{edie-michell:2017} without the modular structure, where it is called the adjoint subcategory $\text{Ad}(D_{2m+2})$ (note that in \cite{edie-michell:2017} there is a scaling of the trivalent vertex by $\sqrt{[4]_q}/[2]_q$ compared to that used here, where $[m]_q$ is the quantum integer $[m]_q := (q^m - q^{-m})/(q - q^{-1})$). 
Note that our $SO(3)$ modular tensor category is not the same as the category called $SO(3)_q$ in \cite{morrison/peters/snyder:2017}, which is the even subcategory of the Temperley-Lieb category, and which for $q$ is a primitive $2N+2$-th root of unity is the adjoint subcategory $\text{Ad}(A_{N})$ of \cite{edie-michell:2017}.

In Section \ref{sect:module_cat} we discuss module categories for the $SO(3)_{2m}$ modular tensor category, which are classified by a nimrep graph and cell system. The possible nimrep graphs are classified in Section \ref{sect:mod-inv-nimreps}, along with the $SO(3)$ modular invariants. Cell systems are classified in Section \ref{sect:cell-systems}. The classification of $\ast$-module categories is given in Theorem \ref{Thm:classification} at the end of Section \ref{sect:T-cell-system}. 
In \cite{edie-michell:2017} the Brauer-Picard group, that is, its group of Morita auto-equivalences, of the $SO(3)_{2m}$ (or $\text{Ad}(D_{2m+2})$) modular tensor category was shown to be $\bbZ_2^2$, except for $m=4$ where it is $S_3^2$ . Our classification of module categories agrees with the classification in \cite{edie-michell:2017} of module categories which extend to invertible bimodules, except at level 14, where we find two additional module categories which not extend to invertible bimodules.

n the final two sections we apply the theory and results from the main part of this paper to other mathematical constructions, including the $SO(3)$-Temperley-Lieb algebra and preprojective algebras. More specifically, in Section \ref{sect:realisation_MI_subfactors} we relate the $SO(3)$ diagrammatic algebra with the $SO(3)$-analogue of the Temperley-Lieb algebra from \cite{lehrer/zhang:2010, lehrer/zhang:2008}. We correct a claim in \cite{fendley/krushkal:2010} about the surjectivity of a certain homomorphism from the $BMW$ algebra to the $SO(3)$-Temperley-Lieb algebra and disprove a conjecture about the injectivitiy of this homomorphism.
We also present a construction of a subfactor with principal graph the bipartite unfolding of $\mathcal{A}_{2m}$, recovering a subfactor presented in \cite{izumi:1991}.

Finally in Section \ref{sect:SO3_preprojective} we introduce a Frobenius algebra $A$ which is an $SO(3)$ generalisation of 
preprojective algebras.
Preprojective algebras \cite{gelfand/ponomarev:1979} play an important role in representation theory, for example the moduli stack of representations of the preprojective algebra is the cotangent bundle of the moduli stack of representations of the quiver associated to the preprojective algebra. In \cite{evans/pugh:2010ii} it was argued that preprojective algebras are a construction related to $SU(2)$, and this construction was generalised to $SU(3)$. When these generalised preprojective algebras are built from braided subfactors associated to $SU(3)$ modular invariants, it was shown in this work that they are finite-dimensional Frobenius algebras. These algebras were the first examples of higher rank analogues of preprojective algebras, which have since been heavily studied following the work of Iyama and Oppermann \cite{iyama/oppermann:2013}.
We derive a finite resolution of $A$ as a left $A$-module, and its Hilbert series.

\paragraph{\it Acknowledgements.}

The authors thank the Isaac Newton Institute for Mathematical Sciences, Cambridge during the programme Operator Algebras: Subfactors and their applications, for generous hospitality while researching this
paper. The authors also thank the referees for their helpful comments on earlier drafts.
Both authors' research was supported in part by EPSRC grant nos EP/K032208/1 and EP/N022432/1.

\section{Diagrammatic calculus for $SO(3)_{2m}$} \label{sect:construct-cat}

\subsection{The Temperley-Lieb category} \label{sect:TL-cat}

We first describe the Verlinde algebra and fusion rules for $SU(2)$ in the diagrammatic and categorical language of the Temperley-Lieb algebra \cite{kauffman:1987, turaev:1994, yamagami:2003, cooper:2007}.

Let $q$ be real or a primitive root of unity, so that $\delta = [2]_q$ is real. Denote by $\mathcal{T}_{l,n}$ the set of all planar diagrams consisting of a rectangle with $l$, $n$ vertices along the top, bottom edge respectively, and with $(l+n)/2$ curves, called strings, inside the rectangle so that each vertex is the endpoint of exactly one string, and the strings do not cross each other.
Let $\mathcal{V}_{l,n}$ denote the free vector space over $\mathbb{C}$ with basis $\mathcal{T}_{l,n}$.
Composition $RS$ of diagrams $R \in \mathcal{T}_{l,n}$, $S \in \mathcal{T}_{l,p}$ is given by gluing $S$ vertically below $R$ such that the vertices at the bottom of $R$ and the top of $S$ coincide, removing these vertices, and isotoping the glued strings if necessary to make them smooth. Any closed loops which may appear are removed, contributing a factor of $\delta$. The resulting diagram is in $\mathcal{T}_{l,p}$. This composition is clearly associative, and the product in $\mathcal{V} = \bigcup_{l,n \geq 0} \mathcal{V}_{l,n}$ is defined as its linear extension. The adjoint $R^{\ast} \in \mathcal{T}_{n,l}$ of a diagram in $R \in \mathcal{T}_{l,n}$ is given by reflecting $R$ about a horizontal line halfway between the top and bottom vertices of the diagram. This action is extended conjugate linearly to $\mathcal{V}$, so that $\mathcal{V}$ is a $\ast$-algebra.

The Temperley-Lieb category $TL(\delta)$ is the matrix category $TL(\delta) = \mathrm{Mat}(\mathcal{C}_0)$, where $\mathcal{C}_0$ is the tensor category whose objects are (self adjoint) projections in $\mathcal{V}_n := \mathcal{V}_{n,n}$, and whose morphisms $\mathrm{Hom}(p_1,p_2)$ between projections $p_i \in \mathcal{V}_{n_i}$, $i=1,2$, are given by the space $p_2 \mathcal{V}_{n_2,n_1} p_1$. We will typically use fraktur script to denote morphisms.
The tensor product is defined on the objects and morphisms by horizontal juxtaposition. The trivial object $\mathrm{id}_0$ is the empty diagram which is a projection in $\mathcal{V}_0$. (The category $\mathcal{C}_0$ is the idempotent completion, or Karoubi envelope, of the category whose objects are non-negative integers, and whose morphisms are given by $\mathcal{V}_{l,n}$.)
Then the matrix category $\mathrm{Mat}(\mathcal{C}_0)$ is the category with objects are given by formal direct sums of objects in $\mathcal{C}_0$, and morphisms $\mathrm{Hom}(p_1 \oplus \cdots \oplus p_{n_1}, q_1 \oplus \cdots \oplus q_{n_2})$ given by $n_2 \times n_1$ matrices, where the $i,j$-th entry is in $\mathrm{Hom}(p_j,q_i)$. The tensor product on $TL(\delta)$ is given on objects by $(p_1 \oplus \cdots \oplus p_{n_1}) \otimes (q_1 \oplus \cdots \oplus q_{n_2}) = (p_1 \otimes q_1) \oplus (p_1 \otimes q_2) \oplus \cdots \oplus (p_{n_1} \otimes q_{n_2})$, and on morphisms by the usual tensor product on matrices with the tensor product for $\mathcal{C}_0$ on matrix entries.
We write $TL(\delta)_n := \mathcal{V}_n$.

We define a trace on $\mathcal{T}_{n,n}$, a map $\mathcal{T}_{n,n} \to \bbC$ such that $\text{trace}(ab) = \text{trace}(ba)$, by attaching a string joining the $j^{\mathrm{th}}$ vertex along the top with the $j^{\mathrm{th}}$ vertex along the bottom, for each $j\in\{1,2,\ldots,n\}$. This will yield a collection of closed loops, each of which yields a factor of $\delta$. The trace is normalised by multiplying by a factor of $\delta^{-n}$. This trace is extended linearly to $\mathcal{V}_n$.

We call a projection $p \in TL$ simple if $\textrm{Hom}(p,p) \cong \bbC$.
In the generic case, $\delta \geq 2$, the Temperley-Lieb category $TL$ is semisimple, that is, every projection is a direct sum of simple projections, and for any pair of non-isomorphic simple projections $p_1$, $p_2$ we have $\langle p_1,p_2 \rangle = 0$. The Jones-Wenzl projection $f^{(j)}$,
$j = 0,1,2,\ldots$, is the largest projection for which $f^{(j)} E_i = 0 = E_i f^{(j)}$ for all $i=1,2,\ldots,j$, where $E_i$ is the diagram
\begin{equation} \label{def:E_i}
E_i = \;\; \raisebox{-.45\height}{\includegraphics[width=40mm]{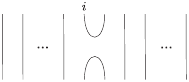}}
\end{equation}
The Jones-Wenzl projections satisfy the recursion relation
\begin{equation}
\raisebox{-.45\height}{\includegraphics[width=18mm]{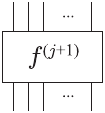}} \;\; = \;\; \raisebox{-.45\height}{\includegraphics[width=18mm]{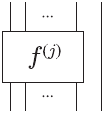}} \;\; - \frac{[j]_q}{[j+1]_q} \;\; \raisebox{-.45\height}{\includegraphics[width=18mm]{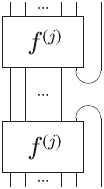}}
\end{equation}
where $[m]_q$ is the quantum integer $[m]_q := (q^m - q^{-m})/(q - q^{-1})$. From the recursion relation we may deduce the identities
\begin{equation} \label{eqn:JW_cap}
\raisebox{-.45\height}{\includegraphics[width=16mm]{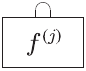}} \;\; = 0
\end{equation}
\begin{equation} \label{eqn:fj_partial_trace}
\raisebox{-.45\height}{\includegraphics[width=18mm]{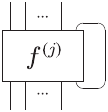}} \;\; = \frac{[j+1]_q}{[j]_q} \;\; \raisebox{-.45\height}{\includegraphics[width=12mm]{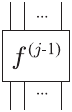}}
\end{equation}

The simple objects in $TL$ are $\rho^{(j)}$, $j = 0,1,2,\ldots$ where $\rho^{(j)}$ is the projection given by the diagram $f^{(j)}$.
Then from \eqref{eqn:fj_partial_trace}, $\mathrm{trace}(\mathfrak{f}^{(j)}) = [j+1]_q$, where $\mathfrak{f}^{(j)}$ is the morphism $\mathrm{id}_{\rho^{(j)}}$.

In the non-generic case, $\delta = [2]_q < 2$, for $q$ a primitive $2k+4^{\mathrm{th}}$ root of unity, we have $\mathrm{trace}(\mathfrak{f}^{(k+1)}) = [k+2]_q = 0$. Thus the negligible morphisms (morphisms $\mathfrak{p}$ for which $\langle \mathfrak{p}, \mathfrak{p} \rangle = 0$) are those in the unique proper tensor ideal in the Temperley-Lieb category generated by $\mathfrak{f}^{(k+1)}$ \cite{goodman/wenzl:2003}. The quotient $TL^{(k)} := TL(\delta)/ \langle \mathfrak{f}^{(k+1)} \rangle$ is semisimple with simple objects $\rho^{(j)}$, $j=0,1,2,\ldots,k$.

There is a notion of under- and over-crossings on diagrams in $TL(\delta)$, $TL^{(k)}$, which is unique up to interchanging $q^{1/2} \leftrightarrow q^{-1/2}$ and up to choice of primitive root $q^{1/2}$, given by
\begin{equation} \label{eqn:braiding-TL}
\raisebox{-.45\height}{\includegraphics[width=12mm]{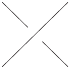}} \;\; = iq^{1/2} \;\; \raisebox{-.45\height}{\includegraphics[width=12mm]{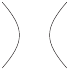}} \;\; -iq^{-1/2} \;\; \raisebox{-.45\height}{\includegraphics[width=12mm]{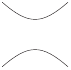}}
\end{equation}
These crossings satisfy the type II and III Reidemeister moves, that is, the inverse of a crossing is given by reflection about a horizontal axis and they satisfy the Yang-Baxter equation.
These crossing diagrams provide additional diagrams to work with in $TL(\delta)$, $TL^{(k)}$.

\subsection{The $D_{2m+2}$ category} \label{sect:D2m+2-cat}

We now restrict to the case $\delta = [2]_q$ for $q$ a primitive $4(2m+1)^{\mathrm{th}}$ root of unity, corresponding to $k=4m$ in the discussion above. We will add an additional generator to the diagrammatic algebra $\mathcal{V}$ to obtain the $D_{2m+2}$ category of \cite{morrison/peters/snyder:2010}.

Let $s$ be a diagram depicted by a box which has $2m$ vertices along both top and bottom.
Denote by $\mathcal{T}^s_{l,n}$ the set of all diagrams consisting of a rectangle with $l$, $n$ vertices along the top, bottom edge respectively, with a finite (possibly zero) number of copies of $s$, and with a finite number of strings inside the rectangle so that each vertex (on the boundary of the rectangle and on the boundaries of any $s$ boxes) is the endpoint of exactly one string, and the strings do not cross each other, i.e. $\mathcal{T}^s_{l,n}$ the set of all planar diagrams with $l$, $n$ vertices along the top, bottom respectively, generated by the $s$ box.
Let $\mathcal{V}^s_{l,n}$ denote the free vector space over $\mathbb{C}$ with basis $\mathcal{T}^s_{l,n}$.
Composition of diagrams in $\mathcal{T}^s_{l,n}$ is as for $\mathcal{T}_{l,n}$, but with the additional relations \cite{morrison/peters/snyder:2010}:
\begin{enumerate}
\item $\raisebox{-.45\height}{\includegraphics[width=15mm]{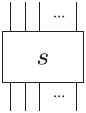}} \;\; = i \;\; \raisebox{-.45\height}{\includegraphics[width=17mm]{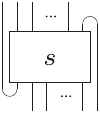}}$
\item $\raisebox{-.45\height}{\includegraphics[width=15mm]{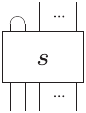}} \;\; =0$
\item $\raisebox{-.45\height}{\includegraphics[width=35mm]{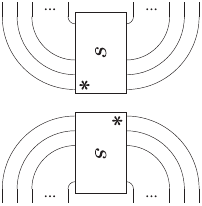}} \;\; = [2m+1]_q f^{(4m)}$
\end{enumerate}
where in the last relation we have inserted $\ast$ to indicate the orientation of the $s$-boxes. Joining the endpoints of the last $2m$ strands along the top in the last relation to the last $2m$ strands along the bottom yields the identity $s^2=f^{(2m)}$ \cite{morrison/peters/snyder:2010}.

Relations 1 and 2 above in fact follow from relation 3. We have the following identity for Jones-Wenzl projections (see e.g. \cite{morrison/peters/snyder:2010})
$$\raisebox{-.45\height}{\includegraphics[width=20mm]{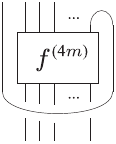}} \;\; = i \;\; \raisebox{-.45\height}{\includegraphics[width=15mm]{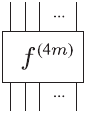}}$$
and composing the diagrams on both sides of this equality with the diagram in Figure \ref{fig-S-capping} we obtain relation 1. Relation 2 follows from capping the diagrams on both sides of relation 3 with a single cap along the top and then composing the diagrams on both sides of the equality with the diagram in Figure \ref{fig-S-capping}. The right hand of the resulting equality is zero due to the cap on $f^{(4m)}$. Thus the only relation needed to define the $s$-box is relation 3.

\begin{figure}
\begin{center}
  \includegraphics[width=35mm]{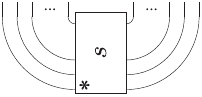}
  \caption{Capping with an $s$-box} \label{fig-S-capping}
\end{center}
\end{figure}

Under- and over-crossings for diagrams in $TL^{(4m)}$ were defined in (\ref{eqn:braiding-TL}). Care is needed in working with these crossing diagrams in $\mathcal{V}^s_{l,n}$, since isotoping a string under an $s$ introduces a factor of $-1$ \cite[Theorem 3.2]{morrison/peters/snyder:2010}:
\begin{equation*}
\raisebox{-.45\height}{\includegraphics[width=18mm]{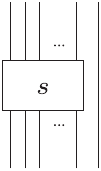}} \;\; = \;\; \raisebox{-.45\height}{\includegraphics[width=20mm]{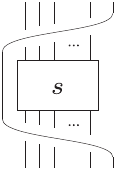}} \hspace{30mm}
\raisebox{-.45\height}{\includegraphics[width=18mm]{fig-SO3_partial_braiding-1}} \;\; = - \;\; \raisebox{-.45\height}{\includegraphics[width=20mm]{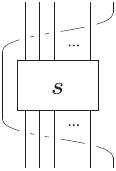}}
\end{equation*}
Thus any diagram is equal to one with at most one $s$-box, since if a diagram has at least two $s$-boxes, one may use the crossings to move any two $s$-boxes adjacent, and after applying relation 1 appropriately the diagram will look locally as in the L.H.S. of relation 3. Thus the two $s$-boxes may be replaced by $[2m+1]_q f^{(4m)}$. This process may be repeated for any remaining pairs of $s$-boxes in the diagram.

We define a trace on $\mathcal{V}^s_{n}:=\mathcal{V}^s_{n,n}$, by attaching strings joining the vertices along the top of diagrams to the vertices along the bottom as for $\mathcal{V}_n$. Applying relations 1-3 we can reduce the number of $s$-boxes to at most one, as described above. If there is a single $s$-box in the resulting diagram, the trace will be 0 by relation 2, whilst if there are no $s$-boxes, the diagram will be a collection of closed loops which each yield a factor of $\delta$.
Then as for $TL^{(4m)}$, $\text{trace}(\mathfrak{f}^{(4m+1)})=[4m+2]_q=0$, and by \cite{morrison/peters/snyder:2010}, the $D_{2m+2}$-category $\widetilde{\mathcal{C}}$ is the matrix category $\widetilde{\mathcal{C}} = \mathrm{Mat}(\mathcal{C}_s)/\langle \mathfrak{f}^{(2k+1)} \rangle$, where $\mathcal{C}_s$ is the tensor category whose objects are (self adjoint) projections in $\mathcal{V}^s_n$ and whose morphisms $\mathrm{Hom}(p_1,p_2)$ between projections $p_i \in \mathcal{V}^s_{n_i}$, $i=1,2$, are given by the space $p_2 \mathcal{V}^s_{n_2,n_1} p_1$.
The category $\widetilde{\mathcal{C}}$ is semisimple \cite{morrison/peters/snyder:2010} with simple objects $\rho^{(j)}$, $j=0,1,2,\ldots,2m-1$, and $P_{\pm}$ which is given by $\frac{1}{2}(f^{(2m)}\pm s)$.
It was shown in \cite{morrison/peters/snyder:2010} that the fusion graph for tensoring with the $SU(2)_{4m}$ generator $f^{(1)}$ is the Dynkin diagram $D_{2m+2}$. The fusion graphs for tensoring with any other simple object in $TL^{(4m)}$ can be deduced from this (see e.g. \cite[$\S$3.5]{izumi:1991}). In particular, the fusion graph for tensoring with $\rho^{(2)}$ is the graph with two connected components illustrated in Figure \ref{fig-Dk+2_fusion_graph-f2}. Izumi showed \cite{izumi:1991} that for $m$ even, $P_+$ and $P_-$ are both self-dual (that is, $P_{\pm} \otimes P_{\pm}$ contains an object isomorphic to the identity $\rho^{(0)}$), whilst for $m$ odd, $P_+$ is dual to $P_-$ (that is, $P_{\pm} \otimes P_{\mp}$ contains an object isomorphic to the identity $\rho^{(0)}$).
From relations 1 and 2 we have $\text{trace}(s)=0$, and therefore $\text{trace}(\mathfrak{P}_{\pm})=[2m+1]_q/2$, where $\mathfrak{P}_{\pm} := \mathrm{id}_{P_{\pm}}$.

\begin{figure}
\begin{center}
  \includegraphics[width=130mm]{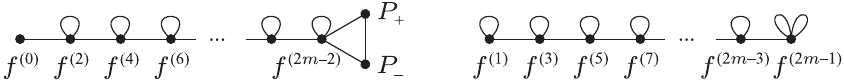}
 \caption{Fusion graph for tensoring with $\rho^{(2)}$ in $TL^{(4m)}$} \label{fig-Dk+2_fusion_graph-f2}
\end{center}
\end{figure}

In $\widetilde{\mathcal{C}}$ the objects $\rho^{(2m+j)}$ corresponding to the Jones-Wenzl projections $f^{(2m+j)}$, $j=1,2,\ldots,2m$ are identified with simple objects of $\widetilde{\mathcal{C}}$ by $\rho^{(2m+j)} \cong \rho^{(2m-j)}$, where the isomorphism  $\phi_j:\rho^{(2m+j)} \to \rho^{(2m-j)}$ is given (up to a scalar factor $[2m-j+1]^{1/2}[4m+1]^{-1/2}$) in Figure \ref{fig-isomorphism-fk-j_fk+j}, where a label next to a strand indicates the number of strands it represents.

\begin{figure}
\begin{center}
  \includegraphics[width=25mm]{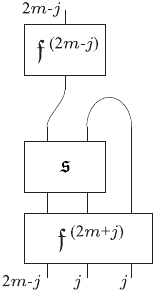}
 \caption{Isomorphism $\rho^{(2m+j)} \cong \rho^{(2m-j)}$, $j=1,2,\ldots,2m$} \label{fig-isomorphism-fk-j_fk+j}
\end{center}
\end{figure}

\subsection{The $SO(3)_{2m}$ category $\mathcal{C}^m$} \label{sect:SO3-cat}

Since the $s$-boxes have an even number $4m$ of vertices on their boundaries, and the diagrams are all planar, we can endow $\mathcal{V}^s_{l,n}$ with a checkerboard shading such that the region which has as its boundary the left-hand side of the outer rectangle is shaded white. The shaded $SO(3)_{2m}$ fundamental generator, which we will also denote by $f^{(2)}$, is the shaded diagram $f^{(2)}$ such that the left-most region of the diagram is shaded white.
Denote by $\mathcal{V}^{(2)}_{l,n}$ the even sub-space of $\mathcal{V}^s_{l,n}$ consisting of all (shaded) planar diagrams generated by the shaded diagram $f^{(2)}$ along with the $4m$-box $s$, that is, the space of tangles generated by $f^{(2)}$ which are consistent with the shading. An equivalent description of these even tangles is that they are given by $(\otimes^l f^{(2)}) \mathcal{V}^s_{2l,2n} (\otimes^n f^{(2)})$, that is, the tangles obtained by composing $\mathcal{V}^s_{2l,2n}$ with $l$ copies of $f^{(2)}$ along the bottom and $n$ copies along the top.
Then the sub-category $\widetilde{\mathcal{C}}^{(2)}$ of $\widetilde{\mathcal{C}}$ is the matrix category $\widetilde{\mathcal{C}}^{(2)} = \mathrm{Mat}(\mathcal{C}_s^{(2)})$, where $\mathcal{C}_s^{(2)}$ is the tensor category whose objects are projections in $\mathcal{V}^{(2)}_n:=\mathcal{V}^{(2)}_{n,n}$, and whose morphisms $\mathrm{Hom}(p_1,p_2)$ between projections $p_i \in \mathcal{V}^{(2)}_{n_i}$, $i=1,2$, are given by the space $p_2 \mathcal{V}^{(2)}_{n_2,n_1} p_1$.
In order to respect the shading, one can no longer isotope a single strand over/under an $s$ box, but rather we must consider pairs of strings.
Isotoping a pair of strings under an $s$ box therefore introduces a factor $(-1)^2 = 1$, and so the crossings (\ref{eqn:braiding-TL}) for diagrams in $\widetilde{\mathcal{C}}$ can be used to pull strings over or under $s$-boxes in $\widetilde{\mathcal{C}}^{(2)}$ with no change of sign.
The sub-category $\widetilde{\mathcal{C}}^{(2)}$ has simple objects $\rho^{(j)}$, $j=0,1,2,\ldots,2m-2$, and $P_{\pm}$.
The fusion graph $\mathcal{A}_{2m}$ for tensoring with $\rho^{(2)}$ is given by the graph on the left hand side in Figure \ref{fig-Dk+2_fusion_graph-f2}, the connected component of $\rho^{(0)}$.

Let $\mathcal{C}^m$ be the category which is isomorphic to $\widetilde{\mathcal{C}}^{(2)}$ obtained under the faithful functor $\varphi: \widetilde{\mathcal{C}}^{(2)} \to \mathcal{C}^m$ which sends $f^{(2)}$ to a single strand, an $s$-box (with $4m$ vertices) to a $t = \varphi(s)$ box with $2m$ vertices, and
$$\varphi: \;\; \raisebox{-.45\height}{\includegraphics[width=15mm]{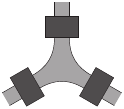}} \;\; \mapsto \frac{\sqrt{[4]_q}}{[2]_q} \;\; \raisebox{-.45\height}{\includegraphics[width=15mm]{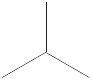}}$$
where the dark boxes denote $f^{(2)}$.
Then $\mathcal{C}^m$ is the category of trivalent graphs and $t$-boxes. Recall that the quantum integer $[m]_q$ is defined as $(q^m-q^{-m})/(q-q^{-1})$, and that here $q$ is a primitive $4(2m+1)^{\mathrm{th}}$ root of unity.
The $SU(2)$ Jones-Wenzl projections $f^{(2j)}$ are mapped to $f_j := \varphi(f^{(2j)})$, $j=0,1,\ldots,2m$, which we call $SO(3)$-Jones-Wenzl projections.
It is easy to verify that under $\varphi$ we have the following relations for diagrams in $\mathcal{C}^m$:
\begin{equation}\label{eqn:SO3-relation1}
\raisebox{-.45\height}{\includegraphics[width=8mm]{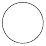}} \;\; = [3]_q, \hspace{30mm} \raisebox{-.45\height}{\includegraphics[width=5mm]{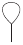}} \;\; = 0
\end{equation}
\begin{equation}\label{eqn:SO3-relation4}
\raisebox{-.45\height}{\includegraphics[width=12mm]{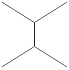}} \;\; - \;\; \raisebox{-.45\height}{\includegraphics[width=12mm]{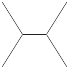}} \;\; = \frac{[2]_q}{[4]_q} \;\; \raisebox{-.45\height}{\includegraphics[width=12mm]{fig-SO3-diagram-1}} \;\; -\frac{[2]_q}{[4]_q} \;\; \raisebox{-.45\height}{\includegraphics[width=12mm]{fig-SO3-diagram-E}}
\end{equation}
and the relations on the $2m$-box $t$ derive from:
\begin{equation}\label{eqn:diagrammatic_relations-T}
\raisebox{-.45\height}{\includegraphics[width=35mm]{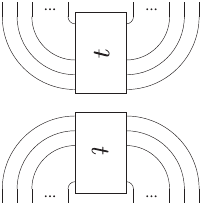}} \;\; = [2m+1]_q f_{2m}
\end{equation}

Just as for the $s$-box in the $D_{2m+2}$-category $\widetilde{\mathcal{C}}$, we can deduce that $t^2 = f_{m}$ and we have the following relations for $t$ (c.f. Section \ref{sect:D2m+2-cat}):
\begin{equation} \label{eqn:T-relations}
\raisebox{-.45\height}{\includegraphics[width=16mm]{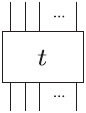}} \;\; = - \;\; \raisebox{-.45\height}{\includegraphics[width=18mm]{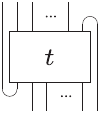}} \hspace{20mm}
\raisebox{-.45\height}{\includegraphics[width=16mm]{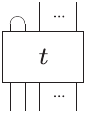}} \;\; = 0 = \;\; \raisebox{-.45\height}{\includegraphics[width=16mm]{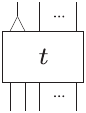}}
\end{equation}
From relations \eqref{eqn:SO3-relation1}-\eqref{eqn:SO3-relation4} we can obtain the relations
$$\raisebox{-.45\height}{\includegraphics[width=5mm]{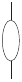}} \;\; = \;\; \raisebox{-.45\height}{\includegraphics[width=1mm]{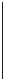}}$$
$$\raisebox{-.45\height}{\includegraphics[width=20mm]{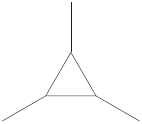}} \;\; = \frac{q^2-1+q^{-2}}{q^2+q^{-2}} \;\; \raisebox{-.45\height}{\includegraphics[width=15mm]{fig-phi-triple_point-2}} \;\; = \frac{[6]_q}{[3]_q[4]_q} \;\; \raisebox{-.45\height}{\includegraphics[width=15mm]{fig-phi-triple_point-2}}$$
\begin{align}
\raisebox{-.45\height}{\includegraphics[width=16mm]{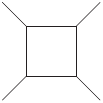}} \;\; &= \frac{q^2-2+q^{-2}}{q^2+q^{-2}} \;\; \raisebox{-.45\height}{\includegraphics[width=10mm]{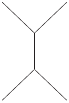}} \;\; + \frac{q^2-1+q^{-2}}{(q^2+q^{-2})^2} \;\; \raisebox{-.45\height}{\includegraphics[width=10mm]{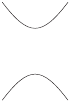}} \;\; + \frac{1}{(q^2+q^{-2})^2} \;\; \raisebox{-.45\height}{\includegraphics[width=10mm]{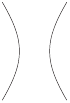}} \label{eqn:SO3-relation6} \\
&= \frac{[2]_q^2[6]_q}{[3]_q[4]_q[2m-2]_q[2m]_q} \;\; \raisebox{-.45\height}{\includegraphics[width=10mm]{fig-SO3-relation6-2}} \;\; + \frac{[2]_q[6]_q}{[3]_q[4]_q^2} \;\; \raisebox{-.45\height}{\includegraphics[width=10mm]{fig-SO3-relation6-3}} \;\; + \frac{[2]_q^2}{[4]_q^2} \;\; \raisebox{-.45\height}{\includegraphics[width=10mm]{fig-SO3-relation6-4}} \nonumber
\end{align}
Similarly, for any diagram which contains an elliptic face, that is, a region which is bounded on all sides by strings, applying relation \eqref{eqn:SO3-relation4} to one of these strings yields a linear combination of diagrams bounded by fewer strings. Iterating this procedure will result in a linear combination of diagrams without any elliptic faces, that is, the intersection of any region of the diagram with the boundary of the diagram is non-empty.

\begin{Lemma}
The $SO(3)$-Jones-Wenzl projections $f_j$, $j=0,1,2,\ldots,2m$ satisfy the following defining properties:
$$\raisebox{-.45\height}{\includegraphics[width=16mm]{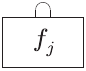}} \;\; = 0 = \;\; \raisebox{-.45\height}{\includegraphics[width=16mm]{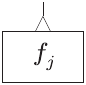}}, \hspace{30mm}
\raisebox{-.45\height}{\includegraphics[width=20mm]{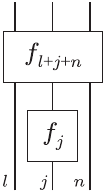}} \;\; = \;\; \raisebox{-.45\height}{\includegraphics[width=18mm]{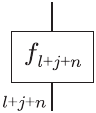}},$$
for any $l,n\geq0$, where the thick lines denote multiple strands, with the number of strands represented ($l,j,n$ respectively) written next to the lines.
\end{Lemma}

\begin{proof}
The first two identities follow easily from \eqref{eqn:JW_cap} in $TL^{(4m)}$:
\begin{align*}
\raisebox{-.45\height}{\includegraphics[width=16mm]{fig-fj_cap-1}} \;\; &\stackrel{\varphi^{-1}}{=} \;\; \raisebox{-.45\height}{\includegraphics[width=16mm]{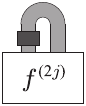}} \;\; = \;\; \raisebox{-.45\height}{\includegraphics[width=16mm]{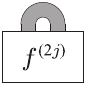}} \;\; = 0, \\
\raisebox{-.45\height}{\includegraphics[width=16mm]{fig-fj_trivalent_cap-1}} \;\; &\stackrel{\varphi^{-1}}{=} \frac{[2]_q}{\sqrt{[4]_q}} \;\; \raisebox{-.45\height}{\includegraphics[width=16mm]{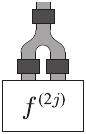}} \;\; = \frac{[2]_q}{\sqrt{[4]_q}} \;\; \raisebox{-.45\height}{\includegraphics[width=16mm]{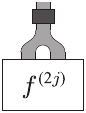}} \;\; = 0
\end{align*}
The projector identity follows from the analogous identity for Jones-Wenzl projections in $TL^{(4m)}$. Finally, that these relations uniquely define the $SO(3)$-Jones-Wenzl projections follows by a similar argument to the proof of \cite[Theorem 2.2.1(3)]{suciu:1997}.
\end{proof}

From a double application of the recursion relation for the Jones-Wenzl projections the even $SU(2)$ Jones-Wenzl projections $f^{(2j)}$, $j=0,1,2,\ldots,2m$, satisfy the recursion relation
\begin{align*}
\raisebox{-.45\height}{\includegraphics[width=16mm]{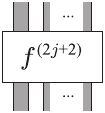}} \;\; &= \;\; \raisebox{-.45\height}{\includegraphics[width=19mm]{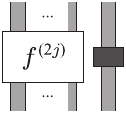}} \;\; - \frac{[2j-1]_q}{[2j+1]_q} \;\; \raisebox{-.45\height}{\includegraphics[width=22mm]{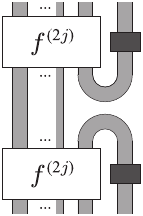}} \;\; - \frac{[2]_q[2j]_q}{[2j+2]_q} \;\; \raisebox{-.45\height}{\includegraphics[width=20mm]{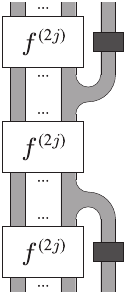}} \\
&= \;\; \raisebox{-.45\height}{\includegraphics[width=19mm]{fig-SU2-evenJW_recursion-2}} \;\; - \frac{[4j]_q}{[2j+1]_q[2j+2]_q} \;\; \raisebox{-.45\height}{\includegraphics[width=22mm]{fig-SU2-evenJW_recursion-3}} \;\; - \frac{[2]_q[2j]_q}{[2j+2]_q} \;\; \raisebox{-.45\height}{\includegraphics[width=22mm]{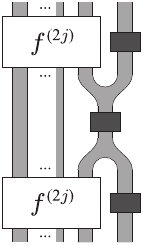}}
\end{align*}
Under $\varphi$ this recursion relations yields the following recursion relation on the $SO(3)$-Jones-Wenzl projections $f_j$, $j=0,1,2,\ldots,2m$:
\begin{align*}
\raisebox{-.45\height}{\includegraphics[width=17mm]{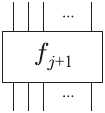}} \;\; &= \;\; \raisebox{-.45\height}{\includegraphics[width=17mm]{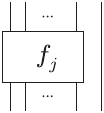}} \;\; - \frac{[2j-1]_q}{[2j+1]_q} \;\; \raisebox{-.45\height}{\includegraphics[width=18mm]{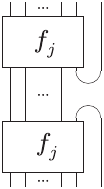}} \;\; - \frac{[4]_q[2j]_q}{[2]_q[2j+2]_q} \;\; \raisebox{-.45\height}{\includegraphics[width=18mm]{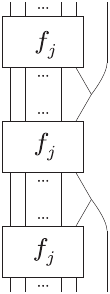}} \\
&= \;\; \raisebox{-.45\height}{\includegraphics[width=17mm]{fig-SO3-JW_recursion-2}} \;\; - \frac{[4j]_q}{[2j+1]_q[2j+2]_q} \;\; \raisebox{-.45\height}{\includegraphics[width=18mm]{fig-SO3-JW_recursion-3}} \;\; - \frac{[4]_q[2j]_q}{[2]_q[2j+2]_q} \;\; \raisebox{-.45\height}{\includegraphics[width=18mm]{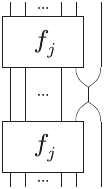}}
\end{align*}

The simple objects of $\mathcal{C}^m$ are $\rho_j$, $j=0,1,2,\ldots,m-1$, given by $f_j = \varphi(f^{(2j)})$, and $Q_{\pm}$ given by $\varphi(P_{\pm}) = \frac{1}{2}(f_{m} \pm t)$.
The trace on $\mathcal{V}^s_{2n}$ restricts to a trace on $\mathcal{V}^{(2)}_n$. Under $\varphi$ this passes to a trace $\text{tr}_T$ on $\mathcal{W}_n=\varphi(\mathcal{V}^{(2)}_n)$, defined on diagrams by attaching a string joining the $j^{\mathrm{th}}$ vertex along the top with the $j^{\mathrm{th}}$ vertex along the bottom, for each $j\in\{1,2,\ldots,n\}$, with each resulting closed loop now contributing a factor of $[3]_q$.
Let $\mathfrak{f}_{j} := \text{id}_{\rho_j}$. We have
\begin{equation} \label{eqn:tr(fj)}
\text{tr}_T(\mathfrak{f}_{j}) = \text{trace}(\mathfrak{f}^{(2j)}) = [2j+1]_q.
\end{equation}
Note that the trace of $\mathfrak{f}_j$ is non-zero for any $j$, however $\text{tr}_T (\mathfrak{f}_{m+j}-\mathfrak{f}_{m-j}) = [2m+2j+1]_q - [2m-2j+1]_q = 0$, consistent with $\rho_{m+j} \cong \rho_{m-j}$ in $\mathcal{C}^m$.

The $SO(3)$-Jones-Wenzl projections satisfy the fusion rules
\begin{equation} \label{eqn:fusion_rules-f1}
\rho_j \otimes \rho_1 \cong \rho_{j-1} \oplus \rho_j \oplus \rho_{j+1},
\end{equation}
for $j=1,2,\ldots,2m-1$, which follows from the fusion graph for tensoring with $\rho^{(2)}$ in $\tilde{\mathcal{C}}$, in Figure \ref{fig-Dk+2_fusion_graph-f2}. For $j=0$ we have the trivial identity $\rho_0 \otimes \rho_1 \cong \rho_1$. We can explicitly construct isomorphisms $\psi: \rho_j \otimes \rho_1 \to \rho_{j-1} \oplus \rho_j \oplus \rho_{j+1}$ and $\psi^{-1} = \psi^{\ast}$ so that $\psi^{\ast} \psi = \mathfrak{f}_j \otimes \mathfrak{f}_1 = \text{id}_{\rho_j} \otimes \text{id}_{\rho_1}$ and $\psi \psi^{\ast} = \mathfrak{f}_{j-1} \oplus \mathfrak{f}_j \oplus \mathfrak{f}_{j+1} = \text{id}_{\rho_{j-1}} \oplus \text{id}_{\rho_{j}} \oplus \text{id}_{\rho_{j+1}}$, by
$$\psi = \left( \begin{array}{c} \psi_1 \\ \psi_2 \\ \psi_3 \end{array} \right),$$
where
\begin{equation}
\psi_1 = \frac{\sqrt{[2m-1]_q}}{\sqrt{[2m+1]_q}} \;\; \raisebox{-.45\height}{\includegraphics[width=17mm]{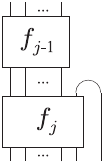}} \, , \quad
\psi_2 = \frac{\sqrt{[4]_q[2m]_q}}{\sqrt{[2]_q[2m+2]_q}} \;\; \raisebox{-.45\height}{\includegraphics[width=18mm]{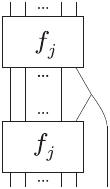}} \, , \quad
\psi_3 = \raisebox{-.45\height}{\includegraphics[width=17mm]{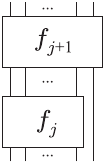}}
\end{equation}
and with $\psi^{\ast}$ defined as the transpose of $\psi$, with each entry $\psi_i$ replaced by its reflection $\psi_i^{\ast}$ about the horizontal axis. When $j=0$, $\psi_1=\psi_2=0$ and $\psi_3$ yields the isomorphism between $\rho_0 \otimes \rho_1$ and $\rho_1$.
The fusion graph for tensoring by the fundamental generator $\rho_1$ is given in Figure \ref{fig-SO3_fusion_graph}.

\begin{figure}
\begin{center}
  \includegraphics[width=60mm]{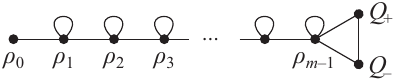}
 \caption{Fusion graph for tensoring with $\rho_1$ in $\mathcal{C}^m$} \label{fig-SO3_fusion_graph}
\end{center}
\end{figure}

\begin{Rem} \label{rem:ast-structure-C}
There is a natural $\ast$-structure on $\mathcal{C}^m$, defined on diagrams by reflecting about a horizontal axis, and extended conjugate linearly to linear combinations of diagrams.
\end{Rem}

\begin{Rem} \label{SO3_q-generic_q}
For generic $q$, the diagrammatic algebra $\mathcal{W}:= \bigcup_n \varphi \left( \mathcal{V}_n^{(2)} \right)$ for $SO(3)_q$ is given by diagrams generated by the trivalent vertex subject to the relations \eqref{eqn:SO3-relation1}-\eqref{eqn:SO3-relation4} -- there is no additional generator $t$ in this case. Thus for generic $q$, $SO(3)_q$ is isomorphic to the chromatic algebra \cite{martin/woodcock:1998, fendley/krushkal:2010}.
\end{Rem}

Over- and under-crossings in $TL^{(4m)}$ transport to the following over- and under-crossings in $\mathcal{C}^m$
\begin{equation} \label{eqn:braiding-SO3}
\raisebox{-.45\height}{\includegraphics[width=12mm]{fig-braiding}} \;\; = q^2 \;\; \raisebox{-.45\height}{\includegraphics[width=12mm]{fig-SO3-diagram-1}} \;\; +(q^{-2}-1) \;\; \raisebox{-.45\height}{\includegraphics[width=12mm]{fig-SO3-diagram-E}} \;\; - (q^2 + q^{-2}) \;\; \raisebox{-.45\height}{\includegraphics[width=12mm]{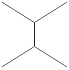}}
\end{equation}
which can be found by mapping the L.H.S. of (\ref{eqn:braiding-SO3}) to $\widetilde{\mathcal{C}}^{(2)}$ using $\varphi^{-1}$, which yields a diagram involving four crossings that can be expanded using relation (\ref{eqn:braiding-TL}) in $\widetilde{\mathcal{C}}^{(2)}$, and finally mapping the resulting expression back into $\mathcal{C}^m$ using $\varphi$.

\begin{Lemma}
Over- and under-crossings in $\mathcal{C}^m$ are uniquely given by (\ref{eqn:braiding-SO3}), up to interchanging $q \leftrightarrow q^{-1}$.
\end{Lemma}

\begin{proof}
The space of $SO(3)$ diagrams with four free end points is spanned by the diagrams \;\; \raisebox{-.2\height}{\includegraphics[width=8mm]{fig-SO3-diagram-1}} \;, \;\; \raisebox{-.2\height}{\includegraphics[width=8mm]{fig-SO3-diagram-E}} \;\; and \;\; \raisebox{-.2\height}{\includegraphics[width=8mm]{fig-SO3-diagram-U}} \;. Thus a crossing should be given by a linear combination of these diagrams. The type II Reidemeister move and the relation
$$\raisebox{-.45\height}{\includegraphics[width=15mm]{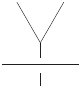}} \;\; = \;\; \raisebox{-.45\height}{\includegraphics[width=15mm]{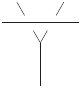}}$$
are sufficient to fix the coefficients in the linear combination.
\end{proof}

\subsection{Modular structure of $\mathcal{C}^m$}

The crossings defined in \eqref{eqn:braiding-SO3} yield a braiding in the $SO(3)_{2m}$ category $\mathcal{C}^m$, thus making it a braided tensor category.
The $SO(3)_{2m}$ category $\mathcal{C}^m$ is in fact a modular tensor category.
Define a matrix $Y$, indexed by the simple objects of $\mathcal{C}^m$, by \cite{rehren:1990}
$$Y_{\lambda,\mu} = \;\; \raisebox{-.45\height}{\includegraphics[width=50mm]{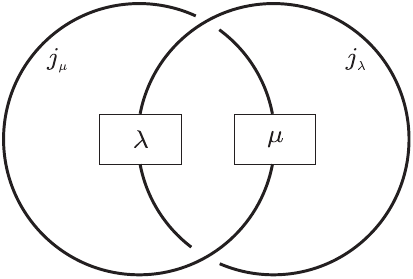}}$$
where the thick lines denote multiple strands, with multiplicity $j_{\lambda}$, $j_{\mu}$ respectively.
For $i,j \in \{ 0,1, \ldots, m-1 \}$, $Y_{\rho_i,\rho_j} = Y_{\rho^{(2i)},\rho^{(2j)}}$, which is given by the corresponding diagram in $TL^{(4m)}$ under $\varphi^{-1}$, thus $Y_{\rho_i,\rho_j} = [(2i+1)(2j+1)]_q$ (see e.g. \cite[XII. $\S$5]{turaev:1994}. Note that the convention used there is that closed loops contribute a factor of $-[2]_q$, thus there is an additional factor of $(-1)^{i+j}$ in $Y_{i,j}$).
Define $Y_{\rho_j,t}$ by the closed diagram
$$\raisebox{-.45\height}{\includegraphics[width=50mm]{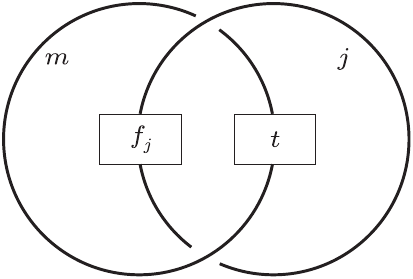}} \;\; = \;\; \raisebox{-.45\height}{\includegraphics[width=50mm]{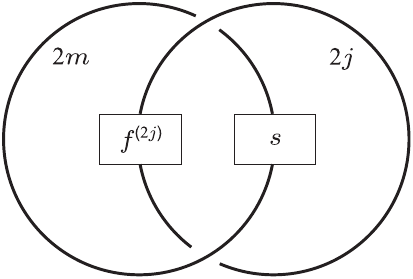}}$$
Note that $Y_{\rho_0,t} = 0$ by \eqref{eqn:T-relations}. We will show that $Y_{\rho_j,t} = 0$ for all $j = 0,1,2,\ldots,m$ by induction.
Using the recursion relation for the Jones-Wenzl projections, the diagram on the R.H.S. is given by the linear combination
$$\raisebox{-.45\height}{\includegraphics[width=50mm]{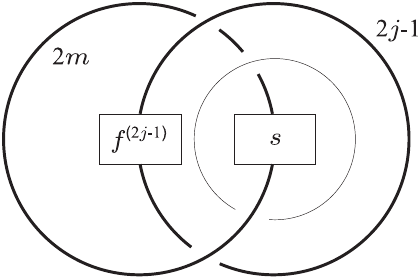}} \;\; - \frac{[2j-1]_q}{[2j]_q} \;\; \raisebox{-.45\height}{\includegraphics[width=50mm]{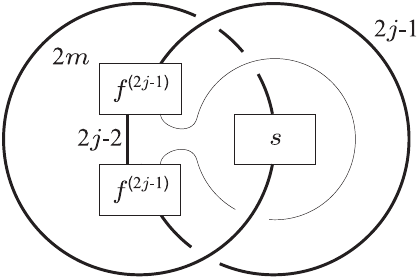}}$$
The first term is zero by induction, whilst the second term yields
\begin{center}
\includegraphics[width=50mm]{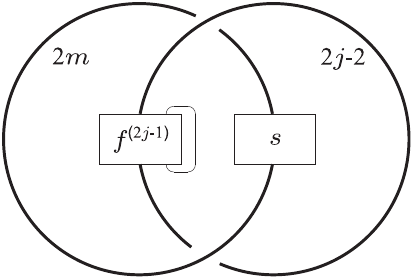}
\end{center}
by sliding the top $f^{(2j-1)}$ box around the $2j-1$ strands on the right side of the diagram. By \eqref{eqn:fj_partial_trace} this term is also zero by induction.
Thus $Y_{\rho_j,Q_{\pm}} = \frac{1}{2} \left( Y_{\rho_j,\rho_m} \pm Y_{\rho_j,t} \right) = \frac{1}{2} Y_{\rho_j,\rho_m} = \frac{1}{2} [(2j+1)(2m+1)]_q$.

We now compute $Y_{Q_{\epsilon_1},Q_{\epsilon_2}}$ for $\epsilon_1,\epsilon_2 \in \{ \pm1 \}$.
Define $Y_{t,t}$ by the closed diagram
$$\raisebox{-.45\height}{\includegraphics[width=50mm]{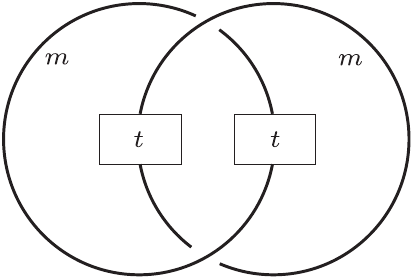}} \;\; = \;\; \raisebox{-.45\height}{\includegraphics[width=50mm]{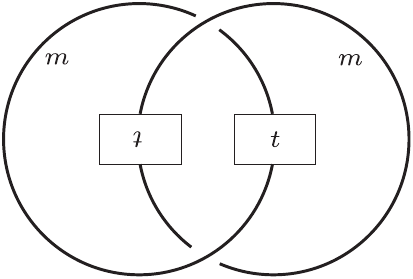}} \;\; = \;\; \raisebox{-.45\height}{\includegraphics[width=25mm]{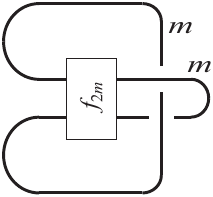}}$$
where the first equality follows from relation \eqref{eqn:T-relations} and the second from \eqref{eqn:diagrammatic_relations-T}. Then by \eqref{eqn:T-relations}, $Y_{t,t}$ is equal to
$$(-1)^m [2m+1]_q \;\; \raisebox{-.45\height}{\includegraphics[width=25mm]{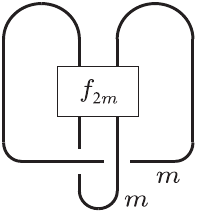}} \;\; \stackrel{\varphi^{-1}}{=} (-1)^m [2m+1]_q \;\; \raisebox{-.45\height}{\includegraphics[width=25mm]{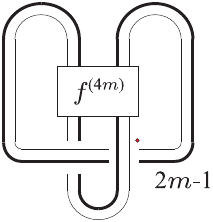}} \;\; =: (-1)^m [2m+1]_q y_{4m},$$
where $y_{4m}$ is a diagram in $TL^{(4m)}$. Expanding the crossing marked by a dot using \eqref{eqn:braiding-TL}, we obtain
\begin{equation} \label{eqn:Y_t,t-1}
y_{4m} = iq^{1/2} \;\; \raisebox{-.45\height}{\includegraphics[width=25mm]{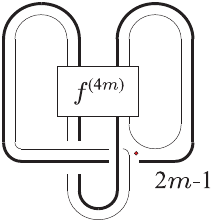}} \;\; -iq^{-1/2} \;\; \raisebox{-.45\height}{\includegraphics[width=25mm]{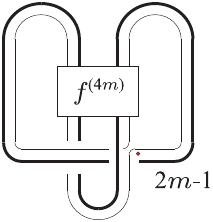}}
\end{equation}
Expanding the crossings marked by a dot, the first term gives
\begin{align*}
(iq^{1/2})^{2m} \;\; \raisebox{-.45\height}{\includegraphics[width=25mm]{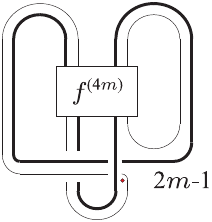}} \;\; &= (iq^{1/2})^{4m-1} \;\; \raisebox{-.45\height}{\includegraphics[width=25mm]{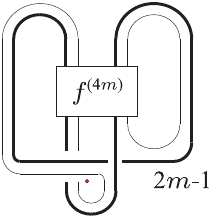}} \;\; = -(iq^{1/2})^{4m+2} \;\; \raisebox{-.45\height}{\includegraphics[width=25mm]{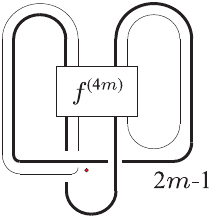}} \\
&= -(iq^{1/2})^{6m+1} \;\; \raisebox{-.45\height}{\includegraphics[width=25mm]{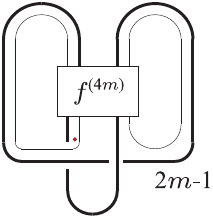}} \;\; = -(iq^{1/2})^{8m} \;\; \raisebox{-.45\height}{\includegraphics[width=25mm]{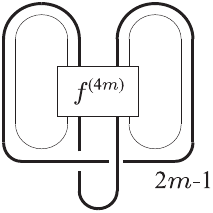}},
\end{align*}
which by \eqref{eqn:fj_partial_trace} is equal to $-q^{4m} [4m+1]_q \, [4m-1]_q^{-1} y_{4m-2}$.
The second term on the R.H.S. of \eqref{eqn:Y_t,t-1} gives
\begin{align*}
(-iq^{-1/2})^{2m} &\;\; \raisebox{-.45\height}{\includegraphics[width=25mm]{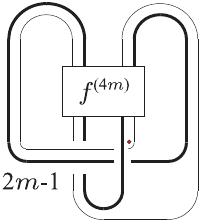}} \;\; = (-iq^{-1/2})^{4m-1} \;\; \raisebox{-.45\height}{\includegraphics[width=25mm]{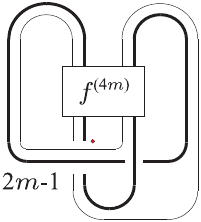}} \\
&= (-iq^{-1/2})^{2m-1} \;\; \raisebox{-.45\height}{\includegraphics[width=25mm]{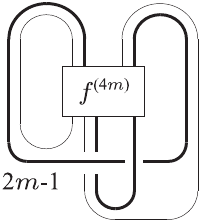}} \;\; = (-iq^{-1/2})^{2m-1} \;\; \raisebox{-.45\height}{\includegraphics[width=25mm]{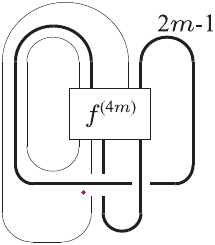}} \\
&= \;\; \raisebox{-.45\height}{\includegraphics[width=25mm]{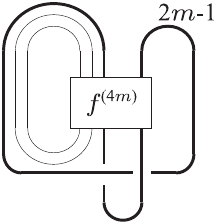}},
\end{align*}
which by \eqref{eqn:fj_partial_trace} is equal to $[4m+1]_q \, [4m-1]_q^{-1} y_{4m-2}$.
Thus we obtain the recursion relation $y_{2j+2} = [2j+1]_q \, [2j-1]_q^{-1} (1-q^{2j}) y_{2j-2}$, which yields
$$y_{4m} = [4m+1]_q \, [1]_q^{-1} y_0 \, \prod_{j=0}^{2m-1} (1-q^{4m-2j}) = \prod_{j=0}^{2m-1} (1-q^{4m-2j}),$$
since $[4m+1]_q=1$ and $y_0$ is the empty diagram.

\begin{Lemma} \label{lemma:prod_1-q}
For $q$ a primitive $(8m+4)$th root of unity, $\prod_{j=0}^{2m-1} (1-q^{4m-2j}) = (-i)^{m} \sqrt{2m+1}$.
\end{Lemma}

\begin{proof}
Let $\tilde{q}$ denote a primitive $2(2m+1)$th root of unity. We use the identity (see e.g. \cite[Appendix]{spiridonov/zhedanov:1997})
$$\prod_{j=1}^{4m+1} (1-\tilde{q}^j) = 2(2m+1).$$
Now
\begin{align*}
\prod_{j=1}^{4m+1} (1-\tilde{q}^j) &= (1-\tilde{q}^{2m+1}) \prod_{j=1}^{2m} (1-\tilde{q}^j) \prod_{l=2m+2}^{4m+1} \tilde{q}^l (1-\tilde{q}^{-l}) = 2(-1)^m \left( \prod_{j=1}^{2m} (1-\tilde{q}^{2m-j}) \right)^2 \\
&= 2(-1)^m \left( \prod_{j=1}^{2m} (1-q^{2(2m-j)}) \right)^2,
\end{align*}
thus $\prod_{j=0}^{2m-1} (1-q^{4m-2j}) = \tau_m \sqrt{2m+1}$ where $\tau_m \in \{ \pm1, \pm i \}$. Note that
\begin{align*}
\prod_{j=0}^{2m-1} (1-q^{4m-2j}) &= (2i)^{2m} \prod_{j=0}^{2m-1} q^{2m-j} \sin \left(\frac{(2m-j)\pi}{2(2m+1)}\right) = (-4)^{m} q^{m(2m+1)} X = (-4i)^{m} X,
\end{align*}
where $X := \prod_{j=0}^{2m-1} \sin \left((2m-j)\pi/(2(2m+1))\right) \geq 0$. Thus $\tau_m = (-i)^m$.
\end{proof}

By Lemma \ref{lemma:prod_1-q} we have $Y_{t,t} = i^{m} [2m+1]_q \sqrt{2m+1}$.
Then
\begin{align*}
Y_{Q_{\epsilon_1},Q_{\epsilon_2}} &= \frac{1}{4} \left( Y_{\rho_m,\rho_m} + (\epsilon_1+\epsilon_2) Y_{\rho_m,t} + \epsilon_1 \epsilon_2 Y_{t,t} \right) \\
&= \frac{1}{4} \left( [(2m+1)^2]_q + \epsilon_1 \epsilon_2 i^{m} [2m+1]_q \sqrt{2m+1} \right) \\
&= \frac{[2m+1]_q}{4} \left( (-1)^m  + \epsilon_1 \epsilon_2 i^{m} \sqrt{2m+1} \right)
\end{align*}

By \cite{rehren:1990}, the matrix $Y$ is invertible if and only if the only index $\lambda$ for which $Y_{\lambda,\mu} = \text{tr}_T(\text{id}_{\lambda}) \text{tr}_T(\text{id}_{\mu})$ is true for all $\mu$ is $\lambda=\rho_0$. This equality is clearly true for $\lambda = \rho_0$.
Consider $\mu = \rho_1$. Then $Y_{\rho_j,\rho_1} = [3(2j+1)]_q = \alpha \sin(3(2j+1)\pi/(4m+2))$ where $\alpha = \sin(\pi/(4m+2))^{-1}$, and $\text{tr}_T(\mathfrak{f}_j) \text{tr}_T(\mathfrak{f}_1) = [3]_q [2j+1]_q = [2j-1]_q + [2j+1]_q + [2j+3]_q = \alpha(\sin((2j-1)\pi/(4m+2)) + \sin((2j+1)\pi/(4m+2)) + \sin((2j+3)\pi/(4m+2)))$. Thus we need to find all values of $j$ for which $\sin(3(2j+1)\pi/(4m+2)) = \sin((2j-1)\pi/(4m+2)) + \sin((2j+1)\pi/(4m+2)) + \sin((2j+3)\pi/(4m+2))$. Using the triple-angle formula for sine on the the L.H.S., and writing the sum on the R.H.S. as a product of sine and cosine functions, this equality reduces to $\cos(2(2j+1)\pi/(4m+2)) = \sin(2\pi/(4m+2))$, which for $j \in \{0,1,2,\ldots,m-1\}$ is only satisfied for $j=0$. A similarly computation for $Y_{Q_{\pm},\rho_1} = \frac{1}{2} [3(2m+1)]_q$ and $\text{tr}_T \left( \frac{1}{2} \left( \mathfrak{f}_m \pm \mathfrak{t} \right) \right) \text{tr}_T(\mathfrak{f}_1) = \frac{1}{2} [3]_q [2m+1]_q$ shows that these are not equal. Thus $Y_{\lambda,\rho_1} \neq \text{tr}_T(\text{id}_{\lambda}) \text{tr}_T(\mathfrak{f}_1)$ for all $\lambda \neq \rho_0$, therefore $Y$ is invertible.

The statistics phase $\omega(\lambda)$ for $\lambda$ a simple object is given by the diagram
$$\raisebox{-.45\height}{\includegraphics[width=16mm]{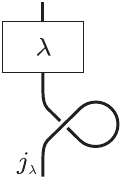}} \;\; = \omega(\lambda) \;\; \raisebox{-.45\height}{\includegraphics[width=11mm]{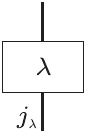}}$$
We have
\begin{align*}
\raisebox{-.45\height}{\includegraphics[width=16mm]{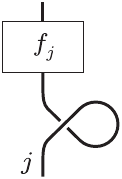}} \;\; &\stackrel{\varphi^{-1}}{=} \;\; \raisebox{-.45\height}{\includegraphics[width=17mm]{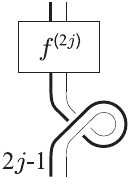}} \;\; = -iq^{-3/2} \;\; \raisebox{-.45\height}{\includegraphics[width=17mm]{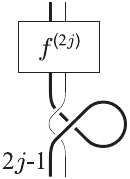}} \;\; = -(-iq^{-1/2})^{4j+1} \;\; \raisebox{-.45\height}{\includegraphics[width=17mm]{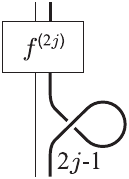}} \\
&= -(-iq^{-1/2})^{4j+1} \;\; \raisebox{-.45\height}{\includegraphics[width=18mm]{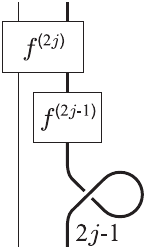}} \;\; = (-iq^{-1/2})^{8j} \;\; \raisebox{-.45\height}{\includegraphics[width=18mm]{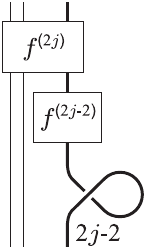}}
\end{align*}
where the last equality follows by the same procedure.
Thus we have the recursion $\omega(\rho_j) = (-iq^{-1/2})^{8j} \omega(\rho_{j-1}) = q^{-4j} \omega(\rho_{j-1})$, which yields $\omega(\rho_j) = q^{-2j(j+1)}$. Since $t=tf_m$, we have
$$\raisebox{-.45\height}{\includegraphics[width=16mm]{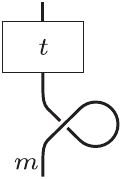}} \;\; = \;\; \raisebox{-.45\height}{\includegraphics[width=16mm]{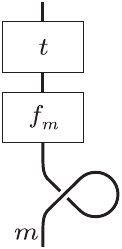}}$$
Hence $\omega(t) = \omega(\rho_m) = q^{-2m(m+1)}$, and we have $\omega(Q_{\pm}) = q^{-2m(m+1)}$.

The $S$-, $T$-matrices are then given by \cite{rehren:1990}
\begin{equation}
S:= |\sigma|^{-1} Y, \qquad T:=\left( \frac{\sigma}{|\sigma|} \right)^{1/3} \mathrm{Diag}(\omega(\lambda)),
\end{equation}
where $\sigma := \sum_{\lambda} \left(\text{tr}_T(\text{id}_{\lambda})\right)^2 \omega(\lambda)^{-1}$ where the summation is over all simple objects $\lambda$. The number $\omega$ satisfies $|\sigma|^2 = \sum_{\lambda} \left(\text{tr}_T(\text{id}_{\lambda})\right)^2$. These $S$-, $T$-matrices agree with those obtained from \eqref{eqn:B-ST} in Section \ref{sect:intro}.

\section{Module categories over $\mathcal{C}^m$} \label{sect:module_cat}

A module category over $\mathcal{C}^m$ is a category $\mathcal{M}$, an action bifunctor $\otimes_{\mathcal{M}} : \mathcal{C}^m \times \mathcal{M} \to \mathcal{M}$ and functorial associativity isomorphism $m_{c,c',x}: (c \otimes c') \otimes_{\mathcal{M}} x \to c \otimes_{\mathcal{M}} (c' \otimes_{\mathcal{M}} x)$ and unit isomorphism $l_x: 1 \otimes_{\mathcal{M}} x \to x$ for any $c,c' \in \mathcal{C}^m$, $x \in \mathcal{M}$, such that the diagrams
$$\xymatrix{
& ((c \otimes c')\otimes c'') \otimes_{\mathcal{M}} x \ar[dl]_{a_{c,c',c''} \otimes \mathrm{id}_x} \ar[dr]^{m_{c \otimes c',c'',x}} & \\
(c \otimes (c' \otimes c'')) \otimes_{\mathcal{M}} x \ar[d]^{m_{c,c' \otimes c'',x}} && (c \otimes c') \otimes_{\mathcal{M}} (c'' \otimes_{\mathcal{M}} x) \ar[d]_{m_{c,c',c'' \otimes_{\mathcal{M}} x}} \\
c \otimes_{\mathcal{M}} ((c' \otimes c'') \otimes_{\mathcal{M}} x) \ar[rr]^{\mathrm{id}_c \otimes m_{c',c'',x}} && c \otimes_{\mathcal{M}} (c' \otimes_{\mathcal{M}} (c'' \otimes_{\mathcal{M}} x))
}$$
and
$$\xymatrix{
(c \otimes 1) \otimes_{\mathcal{M}} x \ar[rr]^{m_{c,1,x}} \ar[dr]^{r_c \otimes \mathrm{id}_x} && c \otimes_{\mathcal{M}} (1 \otimes_{\mathcal{M}} x) \ar[dl]_{\mathrm{id}_c \otimes_{\mathcal{M}} l_x} \\
& c \otimes_{\mathcal{M}} x &
}$$
commute, where $a_{c,c',c''}: (c \otimes c')\otimes c'' \to c \otimes (c' \otimes c'')$ is the associativity isomorphism and $r_c: c \otimes 1 \to c$ the right unit isomorphism in $\mathcal{C}^m$.

The action bifunctor $\otimes_{\mathcal{M}}$ on objects defines a nimrep $G$ indexed by the (non-isomorphic) indecomposable objects of $\mathcal{M}$, so that $c \otimes_{\mathcal{M}} x \cong \sum_y G_c(x,y) y$, where the summation is over all indecomposable objects $y$ in $\mathcal{M}$. In particular, for the $SO(3)_{2m}$ fundamental generator $\rho_1$, we have that $\rho_1 \otimes_{\mathcal{M}} x \cong \sum_y \Delta_{\mathcal{G}} (x,y) y$, where $\Delta_{\mathcal{G}}$ is the adjacency matrix of $\mathcal{G}=G_{\rho_1}$, the fundamental graph which classifies the nimrep. Further, by the associativity isomorphism on $\mathcal{M}$, $\rho_1^j \otimes_{\mathcal{M}} x \cong \sum_y \Delta_{\mathcal{G}}^j (x,y) y$, where $\Delta_{\mathcal{G}}^j (x,y)$ counts the paths of length $j$ from $x$ to $y$.

On morphisms in $\mathcal{M}$, $\otimes_{\mathcal{M}}$ defines the action of diagrams in $\mathcal{W}_{m,n} := \varphi(\mathcal{V}^{(2)}_{m,n})$, which are generated by compositions of cups, caps, triple points and possibly one $t$-box.
For $x,y$ indecomposable objects in $\mathcal{M}$, a basis for $\mathrm{Hom}_{\mathcal{M}} (\rho_1^j \otimes_{\mathcal{M}} x,y)$ is given by all paths of length $j$ on $\mathcal{G}$ from $x$ to $y$.

An alternative definition of a module category is given by \cite{etingof/gelaki/nikshych/ostrik:2015}: the structure of a $\mathcal{C}^m$-module category on a category $\mathcal{M}$ is given by a tensor functor $F$ from $\mathcal{C}^m$ to $\mathrm{Fun}(\mathcal{M},\mathcal{M})$, the category of additive functors from $\mathcal{M}$ to itself.
Such a functor is given on simple objects $c$ by
\begin{equation} \label{eqn:functorF}
F(c) = \bigoplus_{x,y \in \mathcal{G}} G_c(x,y) \, \mathbb{C}_{x,y},
\end{equation}
where the $\mathbb{C}_{x,y}$ are 1-dimensional $R$-$R$ bimodules, $R = (\mathbb{C}\mathcal{G})_0$. The category of $R$-$R$ bimodules has a natural monoidal structure given by $\otimes_R$ (see e.g. \cite{evans/pugh:2010ii}), and we have $R$-$R$ bimodules $F(\rho_1) = \bigoplus_{x,y \in \mathcal{G}} \Delta_{\mathcal{G}}(x,y) \, \mathbb{C}_{x,y} = (\mathbb{C}\mathcal{G})_1$ and $F(\rho_1^m) = \otimes^m (\mathbb{C}\mathcal{G})_1 = (\mathbb{C}\mathcal{G})_m$.
The functor $F$ sends a morphism $\mathfrak{p} \in \mathcal{V}_{m,n}$ of $\mathcal{C}^m$ to a morphism $F(\mathfrak{p}):(\mathbb{C}\mathcal{G})_n \to (\mathbb{C}\mathcal{G})_m$.

Two $\mathcal{C}^m$-module categories $\mathcal{M}, \mathcal{M}'$ with associativity isomorphisms $m,m'$ respectively, are equivalent if there is a functor $H:\mathcal{M} \to \mathcal{M}'$ and a natural isomorphism $s_{c,x}:H(c \otimes_{\mathcal{M}} x) \to c \otimes_{\mathcal{M}'} H(x)$, for any $c \in \mathcal{C}^m$, $x \in \mathcal{M}$ such that diagrams
$$\xymatrix{
& H((c \otimes c') \otimes_{\mathcal{M}} x) \ar[dl]_{H(m_{c,c',x})} \ar[dr]^{s_{c \otimes c',x}} & \\
H(c \otimes_{\mathcal{M}} (c' \otimes_{\mathcal{M}} x)) \ar[d]^{s_{c,c' \otimes_{\mathcal{M}} x}} && (c \otimes c') \otimes_{\mathcal{M}'} H(x) \ar[d]_{m_{c,c',H(x)}'} \\
c \otimes_{\mathcal{M}'} H(c' \otimes_{\mathcal{M}} x) \ar[rr]^{\mathrm{id}_c \otimes s_{c',x}} && c \otimes_{\mathcal{M}'} (c' \otimes_{\mathcal{M}'} H(x))
}$$
and
$$\xymatrix{
H(1 \otimes_{\mathcal{M}} x) \ar[rr]^{s_{1,x}} \ar[dr]^{H(l_x)} && 1 \otimes_{\mathcal{M}'} H(x) \ar[dl]_{l_{H(x)}}\\
& H(x) &
}$$
commute.

Any two equivalent module categories $\mathcal{M}_1$, $\mathcal{M}_2$ will have the same indecomposable objects up to isomorphism.
The nimrep matrices $G_c$ are uniquely determined by $G_{\rho}$, for any simple object $c \not \cong Q_{\pm}$ in $\mathcal{C}$. 
The classification of nimreps is given in Section \ref{sect:mod-inv-nimreps}. Given a module category $\mathcal{M}$, the automorphism $\iota$ of $\mathcal{C}$ that interchanges $Q_+ \leftrightarrow Q_-$ yields a functor $H:\mathcal{M} \to \mathcal{M}_{\iota}$, where $\mathcal{M}_{\iota}$ is a module category where the nimrep matrices $G_{Q_+}$, $G_{Q_-}$ are interchanged. We determine the equivalence or inequivalence of  $\mathcal{M}$ and $\mathcal{M}_{\iota}$ in Section \ref{sect:mod-inv-nimreps}.
Equivalence of $\otimes_{\mathcal{M}_1}$, $\otimes_{\mathcal{M}_2}$ on morphisms is given by a linear transformation $\mathrm{Hom}_{\mathcal{M}_1} (\rho_1^j \otimes_{\mathcal{M}_1} x,y) \to \mathrm{Hom}_{\mathcal{M}_2} (\rho_1^j \otimes_{\mathcal{M}_2} H(x),H(y))$, where $z \in \mathcal{M}_1$, $H(z) \in \mathcal{M}_2$ both label the same vertex of $\mathcal{G}$, $z \in \{ x,y \}$. Any such linear transformation is determined by linear transformations $GL(\Delta_{\mathcal{G}}(x,y),\bbC) \ni \eta^{(xy)}: \mathrm{Hom}_{\mathcal{M}_1} (\rho_1 \otimes_{\mathcal{M}_1} x,y) \to \mathrm{Hom}_{\mathcal{M}_2} (\rho_1 \otimes_{\mathcal{M}_2} H(x),H(y))$.
The classification of all possible actions of any diagram in $\mathcal{V}_{m,n}$ is given in Section \ref{sect:cell-systems}.

\section{Classification of modular invariants and nimreps} \label{sect:mod-inv-nimreps}

In this section we classify $SO(3)$ modular invariants and their associated nimreps.

If $Z$ is an $SO(3)$ modular invariant, then from \eqref{eqn:B-ST} that $X:=bZb^T$ is an $SU(2)$ modular invariant (at $SU(2)$ level $4m$), where $b$ is the branching coefficient matrix. Hence $X$ must be one of $Z_{A_{4m+1}}$ or $Z_{D_{2m+2}}$, or for $m=4,7$, it could also be $Z_{E_7}$, $Z_{E_8}$ respectively. In fact, due to the reflection symmetry of $b$ about its middle row, indexed by $\lambda^{(2m)}$, $X$ has reflection symmetry about both its middle row and middle column, and hence cannot be $Z_{A_{4m+1}}$ which is the identity matrix. Thus for all levels except $m=4,7$, $X$ must be $Z_{D_{2m+2}}$. When $X=Z_{D_{2m+2}}$, the equation $X=bZb^T$ fixes all entries of $Z$ to be $Z_{\lambda,\mu} = \delta_{\lambda,\mu}$, apart from when both $\lambda,\mu \in \{ \lambda_m^{\pm} \}$. The identity $ZS=SZ$ forces $Z_{\lambda_m^+,\lambda_m^+} = Z_{\lambda_m^-,\lambda_m^-}$ and $Z_{\lambda_m^+,\lambda_m^-} = Z_{\lambda_m^-,\lambda_m^+}$. Thus there are only two possibilities for $Z$, the identity invariant which we denote by $Z_{\mathcal{A}_{2m}}$, whilst we denote the non-trivial invariant by $Z_{\sigma_{2m}}$.

At $SO(3)$ level $8$ with $X=Z_{E_7}$, the equation $X=bZb^T$ fixes all entries of $Z$ apart from when one of $\lambda,\mu \in \{ \lambda_1, \lambda_4^{\pm} \}$ and the other is $\lambda_4^{\pm}$. From $ZS=SZ$ there are only four possibilities $Z_{a,b,c}$ for $(a,b,c) \in \{ (0,0,1), (0,1,0), (1,0,0), (1,1,0) \}$, where $a,b,c$ denote the values of the entries $Z_{\lambda_1,\lambda_4^+}, Z_{\lambda_4^+,\lambda_1}, Z_{\lambda_4^+,\lambda_4^+}$ respectively. Note that $Z_{0,0,1}$ and $Z_{1,1,0}$ are related by interchanging the roles of $\lambda_4^+ \leftrightarrow \lambda_4^-$, and $Z_{1,0,0} = Z_{0,1,0}^T$. We denote $Z_{1,1,0}$ by $Z_{\mathcal{E}_8}$ and $Z_{1,0,0}$ by $Z_{\mathcal{E}_8^c}$.
At $SO(3)$ level $14$ with $X=Z_{E_8}$, the equation $X=bZb^T$ fixes all entries of $Z$, which we denote by $Z_{\mathcal{E}_{14}}$.

The complete list of $SO(3)$ modular invariants is thus
\begin{align*}
Z_{\mathcal{A}_{2m}} &= \sum_{j=0}^{m-1} |\chi_j|^2 + |\chi_{m,+}|^2 + |\chi_{m,-}|^2 \\
Z_{\sigma_{2m}} &= \sum_{j=0}^{m-1} |\chi_j|^2 + \chi_{m,+}\overline{\chi_{m,-}} + \text{c.c.} \\
Z_{\mathcal{E}_8} &= |\chi_0|^2 + |\chi_2|^2 + |\chi_3|^2 + |\chi_{4,-}|^2 + \chi_1\overline{\chi_{4,+}} + \text{c.c.} \\
Z_{\mathcal{E}_8^c} &= Z_{\mathcal{E}_8}Z_{\sigma_8} = |\chi_0|^2 + |\chi_2|^2 + |\chi_3|^2 + \chi_1\overline{\chi_{4,-}} + \chi_{4,-}\overline{\chi_{4,+}} + \chi_{4,+}\overline{\chi_1} \\
Z_{\mathcal{E}_{14}} &= Z_{\mathcal{E}_{14}^c} = Z_{\mathcal{E}_{14}}Z_{\sigma_{14}} = |\chi_0+\chi_5|^2 + |\chi_3+\chi_6|^2
\end{align*}
where it is convenient to write $Z$ as a quadratic form.

The eigenvalues of the fusion graph $\mathcal{A}_{2m}$ for $\mathcal{C}^m$ are of the form $\beta_{\lambda} := S_{\rho_1,\lambda}/S_{\rho_0,\lambda}$ for $\lambda$ a simple object of $\mathcal{C}^m$. For $\lambda=\rho_{j-1}$, $j=1,2,\ldots,m$, we have \begin{equation} \label{eqn:evaluesA}
\beta^{(j)} := \beta_{\lambda} = \frac{\sin (3 \pi j/(4m+2))}{\sin(\pi j/(4m+2))} = 2 \cos (2 \pi j/(4m+2)) + 1.
\end{equation}
For $\lambda = Q_{\pm}$, since $Y_{\rho_j,Q_{\pm}} = \frac{1}{2} Y_{\rho_j,\rho_m}$, we again have that $\beta_{Q_{\pm}}$ is given by \eqref{eqn:evaluesA} with $j=m+1$. The values $j=1,2,\ldots,m+1$ are called the exponents of $\mathcal{A}_{2m}$.
The exponents of any nimrep are a subset (ignoring multiplicities) of the exponents for $\mathcal{A}_{2m}$, and thus by \eqref{eqn:evaluesA} the eigenvalues of the nimrep graph $\mathcal{G} = G_{\rho_1}$ lie in $[-1,3)$. Since its eigenvalues are real, any such nimrep $\mathcal{G}$ must be a symmetric, or undirected, graph.
Graves \cite{graves:2010} classifies all symmetric graphs whose eigenvalues lie in $[-1,3]$.
For each $SO(3)$ modular invariant there is only one nimrep with the correct exponents (note that the exponent $j$ corresponds to the exponent $(4m-4j+2,2j-1,2j-1)$ in \cite{graves:2010}). These are $AC_{m+2} \equiv \mathcal{A}_{2m}$, $I_4 \equiv \mathcal{E}_8$ (see \cite[Tables A.1, A.2]{graves:2010}) and $BT_m \equiv \sigma_{2m}$, $H_4 \equiv \mathcal{E}_{14}^c$, $S_{0,2,1} \equiv \mathcal{E}_8^c$, $S_{0,1,2,0} \equiv \mathcal{E}_{14}$ (see \cite[Tables 8.1, 8.2]{graves:2010}.
These graphs are illustrated in Figure \ref{fig-SO3_nimrep_graphs}.
Note that the vertices of the graphs $\mathcal{A}_{2m}$, $\sigma_{2m}$ are the even, odd vertices respectively of the Dynkin diagram $D_{2m+2}$, with the edges given by multiplication by the $SO(3)_{2m}$ fundamental generator $\rho_1$ (which is a vertex of $\mathcal{A}_{2m}$). Similarly the vertices of the graphs $\mathcal{E}_8$, $\mathcal{E}_8^c$ are the even, odd vertices respectively of the Dynkin diagram $E_7$, whilst the vertices of $\mathcal{E}_{14}$, $\mathcal{E}_{14}^c$ are the even, odd vertices respectively of the Dynkin diagram $E_8$, with edges again given by multiplication by the $SO(3)_{2m}$ fundamental generator $\rho_1$.

\begin{figure}
\begin{center}
\includegraphics[width=140mm]{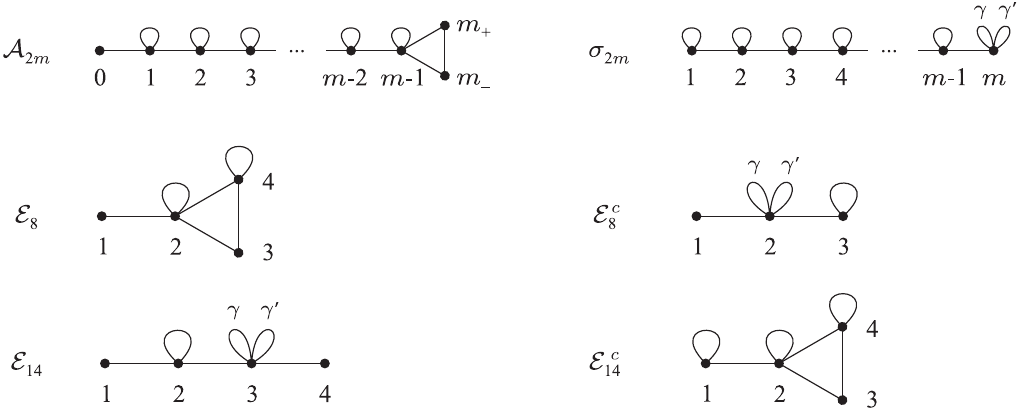} \\
\caption{$SO(3)_{2m}$ nimrep graphs} \label{fig-SO3_nimrep_graphs}
\end{center}
\end{figure}

\begin{table}[tb]
\begin{center}
\begin{tabular}{|c|c|c|c|}
\hline  level $2m$   & nimrep $\mathcal{G}$  & $SU(2)$ ancestor $\mathcal{H}$   & $\mathcal{G}$ the even/odd part of $\mathcal{H}$ \\
\hline
\hline  $2m$        & $\mathcal{A}_{2m}$    & $D_{2m+2}$                        & even          \\
\hline  $2m$        & $\sigma_{2m}$         & $D_{2m+2}$                        & odd           \\
\hline  8           & $\mathcal{E}_8$       & $E_7$                             & even          \\
\hline  8           & $\mathcal{E}_8^c$     & $E_7$                             & odd           \\
\hline  14          & $\mathcal{E}_{14}$    & $E_8$                             & even          \\
\hline  14          & $\mathcal{E}_{14}^c$  & $E_8$                             & odd           \\
\hline
\end{tabular}\\
\caption{Classification of $SO(3)$ modular invariants} \label{table:SO3-MI}
\end{center}
\end{table}

The nimrep graphs $\mathcal{G} = G_{\rho_1}$ illustrated in Figure \ref{fig-SO3_nimrep_graphs} uniquely determine all other nimrep graphs $G_c$ for simple objects $c \not \cong Q_{\pm}$. Using the fusion rules of $\mathcal{C}$, the nimrep graphs $G_{Q_{\pm}}$ are both uniquely determined up to interchanging $G_{Q_+} \leftrightarrow G_{Q_-}$. 
For all cases apart from $\mathcal{G} = \mathcal{A}_{2m}, \mathcal{E}_8$, the nimrep graphs $G_{Q_+}$, $G_{Q_-}$ are equal. Thus $\mathcal{M} = \mathcal{M}_{\iota}$. For $\mathcal{A}_{2m}$ the functor $\iota$ is an equivalence. For $\mathcal{E}_8$, the nimrep graphs $G_{Q_+}$, $G_{Q_-}$ are given by
$$\left( \begin{array}{cccc}
0 & 1 & 0 & 0 \\
1 & 1 & 1 & 1 \\
0 & 1 & 0 & 1 \\
0 & 1 & 1 & 1
\end{array} \right) ,
\left( \begin{array}{cccc}
1 & 0 & 1 & 0 \\
0 & 2 & 0 & 1 \\
1 & 0 & 1 & 1 \\
0 & 1 & 1 & 1
\end{array} \right),$$
where the labelling of rows and columns corresponds to the numbering of the vertices in Figure \ref{fig-SO3_nimrep_graphs}. These are easily found from the fusion rules for $\mathcal{C}^m$, using the fact that $G_{f_m} = G_{Q_+} + G_{Q_-}$.
The functor $\iota$ which interchanges $G_{Q_+} \leftrightarrow G_{Q_-}$ yields inequivalent module categories, that is, for $\mathcal{E}_8$, $\mathcal{M}$ and $\mathcal{M}_{\iota}$ are inequivalent module categories.

\section{Cell system for $SO(3)_{2m}$} \label{sect:cell-systems}

Morphisms in $\mathcal{C}^m$ are given by linear combinations of planar diagrams generated by triple points and $\mathfrak{t}$-boxes, where $\mathfrak{t}$ denotes the morphism given by a $t$-box, subject to relations \eqref{eqn:SO3-relation1}-\eqref{eqn:diagrammatic_relations-T}, where \eqref{eqn:diagrammatic_relations-T} should be interpreted as a relation on morphisms $\mathfrak{t}$ and $\mathfrak{f}_{2m} = \text{id}_{\rho_{2m}}$.
The strings in any such diagram may be isotoped so that in any horizontal strip there is only one of the following elements: a cup, cap, triple point or $\mathfrak{t}$-box. In this section we will classify the actions of these generating elements on morphisms in any module category $\mathcal{M}$ with associated nimrep $G$, where $\mathcal{G} = G_{\rho_1}$ is one of the $SO(3)_{2m}$ graphs in Figure \ref{fig-SO3_nimrep_graphs}.

For $x,y$ indecomposable objects in $\mathcal{M}$, a basis of morphisms in $\mathrm{Hom}_{\mathcal{M}} (\rho_1^j \otimes_{\mathcal{M}} x,y)$ are given by all paths of length $j$ on $\mathcal{G}$ from $x$ to $y$.
Recall from Section \ref{sect:module_cat} that the structure of a module category $\mathcal{M}$ is equivalent to a tensor functor $F:\mathcal{C}^m \to \mathrm{Fun}(\mathcal{M},\mathcal{M})$.
For $\mathfrak{D} \in \mathcal{W}_{m,n} = \varphi(\mathcal{V}^{(2)}_{m,n})$ a morphism in $\mathcal{C}^m$, $F(\mathfrak{D}) : \mathrm{Hom}_{\mathcal{M}} (\rho_1^n \otimes_{\mathcal{M}} x,y) \to \mathrm{Hom}_{\mathcal{M}} (\rho_1^m \otimes_{\mathcal{M}} x,y)$. We will represent $F(\mathfrak{D})$ by the same diagram as $\mathfrak{D}$ but using thick lines to distinguish it from $\mathfrak{D}$. Since $F$ is a functor, the diagrams with thick lines satisfy the same relations as diagrams in $\mathcal{C}^m$. These diagrams will be maps taking an input along the bottom edge and giving output along the top edge, where for $F(\mathfrak{D})$ acting on a single path $\sigma$, the $i$th vertex along the bottom edge of the diagram (i.e. the endpoint of a strand) has as input the $i$th edge of $\sigma$.
Thus the operator $F(\mathfrak{D})$ acting on a path $\sigma$ of length $n$ on $\mathcal{G}$ from $x$ to $y$ will yield an element in $(\bbC \mathcal{G})_m$ given by a linear combination of paths of length $m$ on $\mathcal{G}$ from $x$ to $y$.
We require that a diagram consisting only of vertical strands acts as the identity, e.g.
$$\begin{array}{c} \includegraphics[width=20mm]{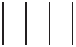} \\ a \quad b \quad c \quad d \end{array} = a \, b \, c \, d,$$
for a path $a \to b \to c \to d$ of length 4, where $a$, $b$, $c$ and $d$ are edges on $\mathcal{G}$.

Let $A_0 = \bbC^n$ where $n$ is the number of vertices of $\mathcal{G}$, and let $A_0 \subset A_1 \subset A_2 \subset \cdots$ be the path algebra (in the usual operator algebraic sense \cite{evans/kawahigashi:1998}), where paths may start at any vertex of $\mathcal{G}$, with $A_j$ given by all pairs of paths $(\alpha_1,\alpha_2)$ of length $j$ such that $s(\alpha_1)=s(\alpha_2)$ and $r(\alpha_1)=r(\alpha_2)$, where $s(\alpha),r(\alpha)$ denote the source, range vertex respectively of a path $\alpha$. The $A_i$ are thus finite dimensional von Neumann algebras, with the Bratteli diagrams for each inclusion $A_j \subset A_{j+1}$ given by the bipartite unfolding of $\mathcal{G}$.
The bipartite unfolding $\mathcal{G}'$ of a non-bipartite graph $\mathcal{G}$ with vertex set $V$ and adjacency matrix $\Delta$ is the bipartite graph with vertex set $V \times \{0,1\}$, and where the number of edges from vertex $(v,i)$ to $(v',j)$ is given by $(1-\delta_{i,j}) \Delta_{v,v'}$.
Multiplication is defined on basis elements $(\alpha_1,\alpha_2), (\alpha_1',\alpha_2') \in A_j$ by $(\alpha_1,\alpha_2) (\alpha_1',\alpha_2') = \delta_{\alpha_2,\alpha_1'} (\alpha_1,\alpha_2')$, and the embedding of a basis element $(\alpha_1,\alpha_2) \in A_j$ in $A_l$, $l>j$, is given by $\sum_{\mu} (\alpha_1 \cdot \mu,\alpha_2 \cdot \mu)$ where the summation is over all paths $\mu$ of length $l-j$ such that $s(\mu)=r(\alpha_1)$.
Note that $A_n \cong F(\mathcal{W}_{n})$.

Consider two module categories $\mathcal{M}_1$, $\mathcal{M}_2$ which have the same set of objects and nimrep graphs. These two module categories are equivalent if there are linear transformations  $GL(\Delta_{\mathcal{G}}(x,y),\bbC) \ni \eta^{(x,y)}: \mathrm{Hom}_{\mathcal{M}_1} (\rho_1 \otimes_{\mathcal{M}_1} x,y) \to \mathrm{Hom}_{\mathcal{M}_2} (\rho_1 \otimes_{\mathcal{M}_2} x,y)$.

\subsection{Canonical module categories} \label{sect:cells-intro}

To begin with, we consider the action of cups and caps.
For each pair of vertices $x$, $y$ of $\mathcal{G}$, denote by $V_{x,y} := \mathrm{Hom}_{\mathcal{M}} (\rho_1 \otimes_{\mathcal{M}} x,y)$. For any functor $F$ there exists two bilinear forms $E^{\cup}_{x,y}, E^{\cap}_{x,y} : V_{x,y} \times V_{y,x} \to \bbC$ such that the creation, annihilation operators associated to cups, caps respectively are given by
$$\begin{array}{c} \includegraphics[width=10mm]{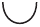} \\ e_x \end{array} = \sum_{b,c:s(b)=r(c)=x} E^{\cup}_{x,r(b)}(b,c) \, bc, \qquad\qquad
\begin{array}{c} \includegraphics[width=10mm]{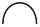} \\ a \qquad b \end{array} = E^{\cap}_{s(a),r(a)}(a,b) \, e_{s(a)}$$
where $e_x$ denotes a basis element for the one-dimensional space of paths of length 0 at vertex $x$ of $\mathcal{G}$.
Since diagrams in $\mathcal{C}^m$ are invariant under isotopy we have that
\begin{equation} \label{eqn:SU2isotopy}
\raisebox{-.45\height}{\includegraphics[width=15mm]{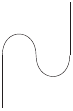}} \;\; = \;\; \raisebox{-.45\height}{\includegraphics[width=1mm]{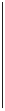}} \;\; = \;\; \raisebox{-.45\height}{\includegraphics[width=15mm]{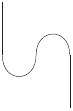}}
\end{equation}
Under the functor $F$ the left hand side yields
\begin{align*}
\begin{array}{c} \includegraphics[width=15mm]{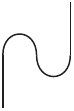} \\ a \qquad\;\; \mbox{ } \end{array} = \sum_{b,c:s(b)=r(c)=r(a)} E^{\cup}_{r(a),r(b)}(b,c) \; \begin{array}{c} \includegraphics[width=15mm]{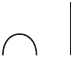} \\ a \quad\; b \quad\; c \end{array} = \sum_{b,c} E^{\cup}_{r(a),r(b)}(b,c) E^{\cap}_{s(a),r(a)}(a,b) \, c,
\end{align*}
where the first equality comes from applying the creation operator corresponding to the cup on the right side of the diagram.
Thus the first equality in \eqref{eqn:SU2isotopy} yields that
$$\sum_{b:s(b)=r(a)} E^{\cap}_{s(a),r(a)}(a,b) E^{\cup}_{r(a),r(b)}(b,c) = \left( E^{\cap}_{s(a),r(a)} E^{\cup}_{r(a),r(a)} \right) (a,c) = \delta_{a,c},$$
whilst the second equality similarly yields $ \left( E^{\cup}_{s(a),r(a)} E^{\cap}_{r(a),s(a)} \right) (c,a) = \delta_{a,c}$. Thus the bilinear forms $E^{\cup}_{x,y}$, $E^{\cap}_{y,x}$ are non-degenerate and
\begin{equation} \label{eqn:Ecup=Ecap-1}
E^{\cup}_{x,y} = \left( E^{\cap}_{y,x} \right)^{-1}
\end{equation}
for any pair of vertices $x$, $y$ of $\mathcal{G}$.
Since a closed loop contributes a factor $[3]_q$ in $\mathcal{C}^m$, we also have that
$$\begin{array}{c} \includegraphics[width=10mm]{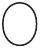} \\ e_x \end{array} = \sum_{b,c:s(b)=r(c)=x} E^{\cup}_{x,r(b)}(b,c) \; \begin{array}{c} \includegraphics[width=10mm]{fig-annihilation} \\ b \qquad c \end{array} = \sum_{b,c} E^{\cup}_{x,r(b)}(b,c) E^{\cap}_{x,r(b)}(b,c) e_x = [3]_q e_x,$$
which, by \eqref{eqn:Ecup=Ecap-1}, yields (c.f. \cite[eq. (3)]{etingof/ostrik:2004})
\begin{align} \label{eqn:cells-circ=[3]}
\sum_{y} \text{Tr}\left( E_{x,y}^{\cup} \left( (E_{y,x}^{\cup})^t \right)^{-1} \right) = [3]_q,
\end{align}
where the summation is over all vertices $y$ of $\mathcal{G}$.

\begin{Def}
A {\bf canonical} $SO(3)$ module category is a module category $\mathcal{M}$ where the bilinear forms are given by
\begin{equation} \label{eqn:canonical-E}
E^{\cup}_{x,y}(a^{(1)}_{xy},a^{(1)}_{yx}) = \sqrt{\phi_y}/\sqrt{\phi_x}, \;\; x \neq y, \qquad E^{\cup}_{x,x}(a^{(1)}_{xx},b^{(1)}_{xx}) = \delta_{a^{(1)}_{xx},b^{(1)}_{xx}}.
\end{equation}
\end{Def}

We will now show that any $SO(3)$ module category is equivalent to a canonical module category. First we need the following Lemma:

\begin{Lemma} \label{Lemma:E/E=PF}
Let $\mathcal{G}$ be an $SO(3)_{2m}$ graph as in Figure \ref{fig-SO3_nimrep_graphs}, and let $x,y$ be any adjacent vertices of $\mathcal{G}$. Then
\begin{equation} \label{eqn:E/E=PF}
\phi_x E^{\cup}_{x,y}(a_{xy},a_{yx})= \phi_y E^{\cup}_{y,x}(a_{yx},a_{xy}),
\end{equation}
where $a_{xy}$ is an element in the basis of $V_{x,y}$ and $(\phi_x)$ is the Perron-Frobenius eigenvector of $\mathcal{G}$ corresponding to the Perron-Frobenius eigenvalue $[3]_q$.
\end{Lemma}

\begin{proof}
For $E^{\cup}_{y,x}(a_{yx},a_{xy}) \neq 0$, let $F_{x,y} := E^{\cup}_{x,y}(a_{xy},a_{yx})/E^{\cup}_{y,x}(a_{yx},a_{xy})$. Consider a subgraph of $\mathcal{G}$ of the form
\begin{center}
\includegraphics[width=20mm]{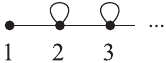}
\end{center}
with $n$ vertices, where all edges attached to vertex $i < n$ in $\mathcal{G}$ are included in this subgraph. Denote by $a_i$, $\tilde{a}_i$ a basis for $V_{i,i+1}$, $V_{i+1,i}$ respectively. Then by induction on $i$ we see that \eqref{eqn:cells-circ=[3]} gives $F_{i,i+1} = [2i+1]_q/[2i-1]_q = \phi_{i+1}/\phi_i$. Similarly, for a subgraph of $\mathcal{G}$ of the form
\begin{center}
\includegraphics[width=20mm]{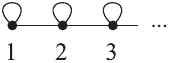}
\end{center}
with $n$ vertices, where all edges attached to vertex $i < n$ in $\mathcal{G}$ are included in this subgraph, \eqref{eqn:cells-circ=[3]} gives $F_{i,i+1} = [2i]_q/[2i]_q = \phi_{i+1}/\phi_i$.

For the edges of $\mathcal{A}_{2m}$ connecting the vertices $m, p_+, p_-$, \eqref{eqn:cells-circ=[3]} yields the equations $F_{m,p_+} + F_{m,p_-} = [2m+1]_q/[2m-1]_q$, $F_{p_{\pm},m} + F_{p_{\pm},p_{\mp}} = [3]_q$, which have the unique solution $F_{m,p_{\pm}} = [2m+1]_q/(2[2m-1]_q) = \phi_{p_{\pm}}/\phi_{m}$ and $F_{p_{\pm},p_{\mp}} = 1 = \phi_{p_{\mp}}/\phi_{p_{\pm}}$.

For the edges of $\mathcal{E}_{8}$ connecting the vertices $2,3,4$, \eqref{eqn:cells-circ=[3]} yields the equations $F_{2,3} + F_{2,4} = [5]_q/[3]_q$, $F_{3,2} + F_{3,4} = [3]_q$ and $F_{4,2} + F_{4,3} = [4]_q/[2]_q$, which have the unique solution $F_{2,3} = [4]_q/([2]_q[3]_q) = \phi_{3}/\phi_{2}$, $F_{2,4} = [6]_q/([2]_q[3]_q) = \phi_{4}/\phi_{2}$ and $F_{3,4} = [6]_q/[4]_q = \phi_{4}/\phi_{3}$.

For the edges of $\mathcal{E}_{14}^c$ connecting the vertices $2,3,4$, \eqref{eqn:cells-circ=[3]} yields the equations $F_{2,3} + F_{2,4} = [6]_q/[4]_q$, $F_{3,2} + F_{3,4} = [3]_q$ and $F_{4,2} + F_{4,3} = [4]_q/[2]_q$, which have the unique solution $F_{2,3} = [3]_q/[5]_q = \phi_{3}/\phi_{2}$, $F_{2,4} = [2]_q^2/[5]_q = \phi_{4}/\phi_{2}$ and $F_{3,4} = [2]_q^2/[3]_q = \phi_{4}/\phi_{3}$.

For the double self-loop $(\gamma,\gamma')$ of the graph $\sigma_{2m}$, \eqref{eqn:cells-circ=[3]} yields
\begin{align*}
F_{m,m-1} + \frac{2E^{\cup}_{m,m}(\gamma,\gamma) E^{\cup}_{m,m}(\gamma',\gamma') - E^{\cup}_{m,m}(\gamma,\gamma')^2 - E^{\cup}_{m,m}(\gamma',\gamma)^2}{E^{\cup}_{m,m}(\gamma,\gamma) E^{\cup}_{m,m}(\gamma',\gamma') - E^{\cup}_{m,m}(\gamma,\gamma') E^{\cup}_{m,m}(\gamma',\gamma)} &= [3]_q \\
\Rightarrow \qquad \frac{2E^{\cup}_{m,m}(\gamma,\gamma) E^{\cup}_{m,m}(\gamma',\gamma') - E^{\cup}_{m,m}(\gamma,\gamma')^2 - E^{\cup}_{m,m}(\gamma',\gamma)^2}{E^{\cup}_{m,m}(\gamma,\gamma) E^{\cup}_{m,m}(\gamma',\gamma') - E^{\cup}_{m,m}(\gamma,\gamma') E^{\cup}_{m,m}(\gamma',\gamma)} &= 2 \\
\Rightarrow \qquad (E^{\cup}_{m,m}(\gamma,\gamma') - E^{\cup}_{m,m}(\gamma',\gamma))^2 &= 0,
\end{align*}
which implies \eqref{eqn:E/E=PF} since $\phi_x=\phi_y$. For the double self-loop $(\gamma,\gamma')$ of the graphs $\mathcal{E}_8^c$ and $\mathcal{E}_{14}$, \eqref{eqn:cells-circ=[3]} yields the same result.
\end{proof}

Suppose that $\mathcal{M}_1$, $\mathcal{M}_2$ are two module categories with the same sets of objects. For each pair of adjacent vertices $x$, $y$ in $\mathcal{G}$, denote by $E^{(1)}_{x,y}$, $E^{(2)}_{x,y}$ the bilinear forms $E_{x,y}^{\cup}$ for $\mathcal{M}_1$, $\mathcal{M}_2$ respectively, and let $V_{x,y}^{(i)}:= \mathrm{Hom}_{\mathcal{M}_i} (\rho_1 \otimes_{\mathcal{M}_i} x,y)$, $i=1,2$. Denote by $\eta^{(x,y)} \in GL(\Delta_{\mathcal{G}}(x,y),\bbC)$ a linear transformation $\eta^{(x,y)}: V_{x,y}^{(1)} \to V_{x,y}^{(2)}$.
Then
\begin{equation} \label{eqn:equivalent-E}
E^{(2)}_{x,y}(a^{(2)}_{xy},a^{(2)}_{yx}) = \sum_{a^{(1)}_{xy},a^{(1)}_{yx}} \eta^{(x,y)}_{a^{(2)}_{xy},a^{(1)}_{xy}} \eta^{(y,x)}_{a^{(2)}_{yx},a^{(1)}_{yx}} E^{(1)}_{x,y}(a^{(1)}_{xy},a^{(1)}_{yx}),
\end{equation}
where $a^{(i)}_{xy}$ denote elements in the basis of $V_{x,y}^{(i)}$, $i=1,2$.
Note that the coefficient of $E^{(1)}_{x,y}(a^{(1)}_{xy},a^{(1)}_{yx})$ on the R.H.S. of \eqref{eqn:equivalent-E} only depend on $x$ and $y$ and is invariant under interchanging $x$ and $y$.

When $x \neq y$, there is at most one edge $a_{x,y}$ between $x$ and $y$, so that \eqref{eqn:equivalent-E} becomes
\begin{equation} \label{eqn:equivalent-E-2}
E^{(2)}_{x,y}(a^{(2)}_{xy},a^{(2)}_{yx}) = \eta^{(x,y)}_{a^{(2)}_{xy},a^{(1)}_{xy}} \eta^{(y,x)}_{a^{(2)}_{yx},a^{(1)}_{yx}} E^{(1)}_{x,y}(a^{(1)}_{xy},a^{(1)}_{yx}).
\end{equation}
Suppose $E^{(2)}_{x,y}(a^{(2)}_{xy},a^{(2)}_{yx}) = \sqrt{\phi_y}/\sqrt{\phi_x}$. With $\eta^{(x,y)}_{a^{(2)}_{xy},a^{(1)}_{xy}} = \sqrt{\phi_y}/(\eta^{(y,x)}_{a^{(2)}_{yx},a^{(1)}_{yx}} \sqrt{\phi_x} E^{(1)}_{x,y}(a^{(1)}_{xy},a^{(1)}_{yx}))$, we have from \eqref{eqn:equivalent-E-2} that
\begin{align*}
\eta^{(y,x)}_{a^{(2)}_{yx},a^{(1)}_{yx}} \eta^{(x,y)}_{a^{(2)}_{xy},a^{(1)}_{xy}} E^{(1)}_{y,x}(a^{(1)}_{yx},a^{(1)}_{xy})
&= \frac{\eta^{(y,x)}_{a^{(2)}_{yx},a^{(1)}_{yx}} \sqrt{\phi_y} E^{(1)}_{y,x}(a^{(1)}_{yx},a^{(1)}_{xy})}{\eta^{(y,x)}_{a^{(2)}_{yx},a^{(1)}_{yx}} \sqrt{\phi_x} E^{(1)}_{x,y}(a^{(1)}_{xy},a^{(1)}_{yx})} = \frac{\sqrt{\phi_y}}{\sqrt{\phi_x}}
\end{align*}
by Lemma \ref{Lemma:E/E=PF}.

When $x=y \neq s(\gamma)$ for $\gamma$ one of the double self-loops, \eqref{eqn:equivalent-E} gives
\begin{equation} \label{eqn:equivalent-E-3}
E^{(2)}_{x,x}(a^{(2)}_{xx},a^{(2)}_{xx}) = \left( \eta^{(x,x)}_{a^{(2)}_{xx},a^{(1)}_{xx}} \right)^2 E^{(1)}_{x,x}(a^{(1)}_{xx},a^{(1)}_{xx}).
\end{equation}
Suppose $E^{(2)}_{x,x}(a^{(2)}_{xx},a^{(2)}_{xx}) = 1$. Then with $\eta^{(x,x)}_{a^{(2)}_{xx},a^{(1)}_{xx}} = 1/\sqrt{E^{(1)}_{x,x}(a^{(1)}_{xx},a^{(1)}_{xx})}$, equation \eqref{eqn:equivalent-E-3} is satisfied.
When $x=y = s(\gamma)$, \eqref{eqn:equivalent-E} gives
\begin{equation} \label{eqn:equivalent-E-4}
E^{(2)}_{x,x}(a^{(2)}_{xx},b^{(2)}_{xx}) = \sum_{a^{(1)}_{xx},b^{(1)}_{xx}} \eta^{(x,x)}_{a^{(2)}_{xx},a^{(1)}_{xx}} \eta^{(x,x)}_{b^{(2)}_{xx},b^{(1)}_{xx}} E^{(1)}_{x,x}(a^{(1)}_{xx},b^{(1)}_{xx}).
\end{equation}
Suppose $E^{(2)}_{x,x}(a^{(2)}_{xx},b^{(2)}_{xx}) = \delta_{a^{(2)}_{xx},b^{(2)}_{xx}}$. Denote by $\{c_i,d_i\}$ a basis for $V_{x,x}^{(i)}$, $i=1,2$. Then the choice
\begin{align*}
\eta^{(x,x)}_{c_2,c_1} &= -\frac{\sqrt{E^{(1)}_{x,x}(c_1,c_1)}}{\sqrt{E^{(1)}_{x,x}(c_1,c_1)E^{(1)}_{x,x}(d_1,d_1)-E^{(1)}_{x,x}(c_1,d_1)^2}}, \\
\eta^{(x,x)}_{c_2,d_1} &= \frac{\sqrt{E^{(1)}_{x,x}(c_1,d_1)}}{\sqrt{E^{(1)}_{x,x}(c_1,c_1)E^{(1)}_{x,x}(d_1,d_1)-E^{(1)}_{x,x}(c_1,d_1)^2}}, \\
\eta^{(x,x)}_{d_2,c_1} &= 0, \\
\eta^{(x,x)}_{d_2,d_1} &= \frac{1}{\sqrt{E^{(1)}_{x,x}(c_1,c_1)}},
\end{align*}
satisfies \eqref{eqn:equivalent-E-4}. Note that Lemma \ref{Lemma:E/E=PF} implies that $E^{(1)}_{x,x}(c_1,d_1) = E^{(1)}_{x,x}(d_1,c_1)$.

Thus we can find linear transformations $\eta^{(x,y)}$ such that any module category $\mathcal{M}_2$ is equivalent to a canonical module category.
From now on we will only consider canonical module categories. For $a=a_{xy} \in V_{x,y}$, we will use the notation $\tilde{a}$ to denote the basis element $a_{y,x} \in V_{y,x}$ for $x \neq y$, whilst for $x=y$ we will identify $\tilde{a}=a$.
From \eqref{eqn:equivalent-E} we see that any two canonical module categories are equivalent only if there exists linear transformations $\eta^{(x,y)}$ such that
\begin{equation} \label{eqn:equivalent-E-canonical}
\eta^{(x,y)}_{a^{(2)}_{xy},a^{(1)}_{xy}} \eta^{(y,x)}_{a^{(2)}_{yx},a^{(1)}_{yx}} = 1 \quad (x \neq y), \qquad
\sum_{a^{(1)}_{xx},b^{(1)}_{xx}} \delta_{a^{(1)}_{xx},b^{(1)}_{xx}} \eta^{(x,x)}_{a^{(2)}_{xx},a^{(1)}_{xx}} \eta^{(x,x)}_{b^{(2)}_{xx},b^{(1)}_{xx}} = \delta_{a^{(2)}_{xx},b^{(2)}_{xx}},
\end{equation}
The second equation in \eqref{eqn:equivalent-E-canonical} yields that $\eta^{(x,x)}$ must be orthogonal.

\begin{Rem}
The choice of bilinear form for a canonical module category is consistent with the $\ast$-structure in $\mathcal{C}^m$, since for $a \in V_{x,y}$, $x \neq y$,
\begin{align*}
\left\langle \begin{array}{c} \includegraphics[width=10mm]{fig-annihilation} \\ a \quad\;\; \tilde{a} \end{array}, s(a) \right\rangle &= \frac{\sqrt{\phi_{r(a)}}}{\sqrt{\phi_{s(a)}}} = \left\langle a\tilde{a}, \phi_{s(a)} \sum_{b} \frac{\sqrt{\phi_{r(b)}}}{\sqrt{\phi_{s(b)}}} b\tilde{b} \right\rangle = \left\langle a\tilde{a}, \begin{array}{c} \includegraphics[width=10mm]{fig-creation} \\ s(a) \end{array} \right\rangle.
\end{align*}
\end{Rem}

\subsection{Trivalent cell systems} \label{sect:trivalent-cell-systems}

We now consider the action of a trivalent vertex. Let the operator associated to a trivalent vertex be given by
\begin{equation*} \label{eq:def:trivalent-vertex}
\begin{array}{c} \includegraphics[width=12mm]{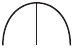} \\ a \quad b \quad c \end{array} = Y(a,b,c) \, e_{s(a)},
\end{equation*}
where the path $abc$ forms a closed loop of length 3 on $\mathcal{G}$ and $Y(a,b,c) \in \bbC$. By isotopy of strings and functoriality this is equal to
\begin{align*}
\begin{array}{l} \includegraphics[width=24mm]{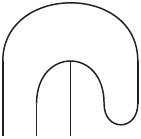} \\ a \quad b \quad c \end{array} \;\; &= \sum_{d:s(d)=r(c)} \frac{\sqrt{\phi_{r(d)}}}{\sqrt{\phi_{r(c)}}} \, \begin{array}{c} \includegraphics[width=24mm]{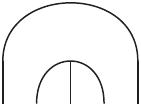} \\ a \quad b \quad c \quad d \quad \tilde{d} \end{array} \;\; = \sum_{d} \frac{\sqrt{\phi_{r(d)}}}{\sqrt{\phi_{r(c)}}} \, Y(b,c,d) \, \begin{array}{c} \includegraphics[width=6mm]{fig-annihilation} \\ a \quad \tilde{d} \end{array} \\
&= \sum_{d} \frac{\sqrt{\phi_{r(d)}}}{\sqrt{\phi_{r(c)}}} \, Y(b,c,d) \, \, \delta_{a,d} \, \frac{\sqrt{\phi_{r(a)}}}{\sqrt{\phi_{s(a)}}} = \frac{\phi_{r(a)}}{\phi_{s(a)}} Y(b,c,a).
\end{align*}
Thus the number $W(abc) := \phi_{s(a)} Y(a,b,c) = \phi_{s(b)} Y(b,c,a) = \phi_{s(c)} Y(c,a,b) \in \bbC$ does not depend on the cyclic permutation of $a,b,c$, i.e. it only depends on the closed loop $abc$ of length 3. We will call $W(abc)$ a trivalent cell, and the choice of such $W$ for each closed loop of length 3 a trivalent cell system.

The relations for the diagrammatic calculus yield relations for the cell system:
\begin{equation} \label{eqn:cell-relations-1}
\sum_{b,c} W(abc) W(\tilde{c}\tilde{b}\tilde{d}) = \delta_{a,d} \phi_{s(a)}\phi_{r(a)},
\end{equation}
\begin{equation} \label{eqn:cell-relations-2}
\sum_b \sqrt{\phi_{r(b)}} \, W(ab\tilde{b}) = 0,
\end{equation}
\begin{align}
\sqrt{\phi_{s(a)}\phi_{r(b)}} \sum_m W(am\tilde{c}) & W(b\tilde{m}\tilde{d}) - \sqrt{\phi_{r(a)}\phi_{r(c)}} \sum_m W(abm) W(\tilde{m}\tilde{d}\tilde{c}) \nonumber \\
&= \frac{[2]_q}{[4]_q} \left( \phi_{r(a)}\phi_{r(b)}\phi_{r(c)} \delta_{a,\tilde{b}} \delta_{c,\tilde{d}} - \phi_{s(a)}\phi_{r(a)}\phi_{r(b)} \delta_{a,c} \delta_{b,d} \right). \label{eqn:cell-relations-3}
\end{align}

Suppose that $\mathcal{M}_1$, $\mathcal{M}_2$ are two canonical module categories. Denote by $W_1$, $W_2$ the trivalent cell systems for $\mathcal{M}_1$, $\mathcal{M}_2$ respectively.
We now deduce the conditions on $W_1$, $W_2$ in order for $\mathcal{M}_2$ to be equivalent to $\mathcal{M}_1$. With linear transformations $\eta^{(x,y)}$ as in \eqref{eqn:equivalent-E-canonical}, the image of a closed loop $a_2 b_2 c_2$ of length 3 on $\mathcal{G}$ under the operator \includegraphics[width=10mm]{fig-trivalent_cap} in $\mathcal{M}_2$ transforms as follows:
\begin{align*}
\begin{array}{c} \includegraphics[width=12mm]{fig-trivalent_cap} \\ a_2 \;\; b_2 \;\; c_2 \end{array} &= \sum_{a_1,b_1,c_1} \eta^{(s(a_2),r(a_2))}_{a_2,a_1} \eta^{(s(b_2),r(b_2))}_{b_2,b_1} \eta^{(s(c_2),r(c_2))}_{c_2,c_1} \begin{array}{c} \includegraphics[width=12mm]{fig-trivalent_cap} \\ a_1 \;\; b_1 \;\; c_1 \end{array} \\
&= \sum_{a_1,b_1,c_1} \eta^{(s(a_2),r(a_2))}_{a_2,a_1} \eta^{(s(b_2),r(b_2))}_{b_2,b_1} \eta^{(s(c_2),r(c_2))}_{c_2,c_1} \phi_{s(a_1)} W_1(a_1,b_1,c_1) e_{s(a_1)}.
\end{align*}
Thus trivalent cell systems $W_1$, $W_2$ are equivalent if there exists linear transformations $\eta^{(x,y)}$ as in \eqref{eqn:equivalent-E-canonical} such that:
\begin{equation} \label{eqn:equivalent-W}
W_2(a_2,b_2,c_2) = \sum_{a_1,b_1,c_1} \eta^{(s(a_2),r(a_2))}_{a_2,a_1} \eta^{(s(b_2),r(b_2))}_{b_2,b_1} \eta^{(s(c_2),r(c_2))}_{c_2,c_1} W_1(a_1,b_1,c_1).
\end{equation}
For any closed loop of the form $a\tilde{a}b$ or $bbb$, where $b$ is a self-loop (but not $\gamma$ or $\gamma'$), this reduces to $W_1$ and $W_2$ differing by a sign.

\begin{Rem} \label{rem:equivalent-unitary-W}
If we consider $\ast$-module categories, that is, module categories compatible with the $\ast$-structure on $\mathcal{C}^m$ (see Remark \ref{rem:ast-structure-C}), then the identity
\begin{align*}
\phi_{s(a)} W(abc) &= \left\langle \begin{array}{c} \includegraphics[width=12mm]{fig-trivalent_cap} \\ a \quad b \quad c \end{array}, s(a) \right\rangle = \left\langle abc, \begin{array}{c} \includegraphics[width=12mm]{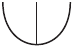} \\ s(a) \end{array} \right\rangle = \left\langle abc, \phi_{s(a)} \sum_{d,e,f} W(\tilde{f}\tilde{e}\tilde{d}) \, def \right\rangle \\
&= \phi_{s(a)} \overline{W(\tilde{c}\tilde{b}\tilde{a})}
\end{align*}
yields that $W(\tilde{c}\tilde{b}\tilde{a}) = \overline{W(abc)}$. In particular this implies that for any closed loop $\triangle$ of the form $a\tilde{a}b$ or $bbb$, where $b$ is a self-loop (but not $\gamma$ or $\gamma'$), that $W(\triangle) \in \bbR$. Combining this identity with \eqref{eqn:equivalent-W} yields that
\begin{equation} \label{eqn:eta_eta_eta}
\eta^{(r(a_2),s(a_2))}_{\tilde{a_2},\tilde{a_1}} \eta^{(r(b_2),s(b_2))}_{\tilde{b_2},\tilde{b_1}} \eta^{(r(c_2),s(c_2))}_{\tilde{c_2},\tilde{c_1}} = \overline{\eta^{(s(a_2),r(a_2))}_{a_2,a_1} \eta^{(s(b_2),r(b_2))}_{b_2,b_1} \eta^{(s(c_2),r(c_2))}_{c_2,c_1}},
\end{equation}
for any closed loop $a_i b_i c_i$ of length 3 on $\mathcal{G}$. Since $\eta^{(x,x)} \in \{ \pm 1 \}$ for any $x$ which is not the source vertex of the double edge $(\gamma,\gamma')$, \eqref{eqn:eta_eta_eta} is satisfied for any trivalent cell where at least one of the edges is a self-loop (but not $\gamma$ or $\gamma'$), using the first equation in \eqref{eqn:equivalent-E-canonical}. When at least one of the edges is $\gamma$ or $\gamma'$, \eqref{eqn:eta_eta_eta} yields that $\eta^{(s(\gamma),s(\gamma))}$ is unitary.
The only closed loops of length 3 which do not involve any self-loops are $[m,p_{\pm},p_{\mp}]$ on $\mathcal{A}_{2m}$; $[234]$, $[243]$ on $\mathcal{E}_8$; and $[234]$, $[243]$ on $\mathcal{E}_{14}^c$. For these closed paths we obtain the condition $\left| \eta^{(s(a_2),r(a_2))}_{a_2,a_1} \eta^{(s(b_2),r(b_2))}_{b_2,b_1} \eta^{(s(c_2),r(c_2))}_{c_2,c_1} \right| = 1$.
\end{Rem}

In the next section we determine equivalence classes of cell systems for the $SO(3)_{2m}$ nimrep graphs.
We will complete the classification of the action of generating diagrams by considering the action of the morphism $\mathfrak{t}$ defined by \eqref{eqn:diagrammatic_relations-T} in Section \ref{sect:T-cell-system}.

\subsection{Classification of trivalent cell systems for $SO(3)_{2m}$ nimrep graphs} \label{sect:classification-W}

In this section we find all equivalence classes of cell systems for the $SO(3)_{2m}$ nimrep graphs. In the proofs we will at times make implicit use of various quantum integer identities. For example, at level $2m$ we have $[2m+1+j] = [2m+1-j]$ for $j=1,\ldots,2m+1$. As a consequence of this, one can deduce other identities at specific levels. For example, at level 14 we have $[10][15] = 2[5][10]$ which implies that $[15]=2[5]$, from which many other identities can be deduced, e.g. $2[2][5]=[2][15]=2[14]$, so that $[14] = [2][5]$. The details of all quantum integer identities used are not provided explicitly in the proofs. The labels of the vertices in this section are as in Figure \ref{fig-SO3_nimrep_graphs}.

\begin{Thm} \label{thm:cell_system-A2m}
Any trivalent cell system in a canonical module category $\mathcal{M}$ with nimrep graph $\mathcal{A}_{2m}$ is equivalent to
\begin{align}
W_{l-1,l,l} &= \frac{\sqrt{[2][2l-1][2l+1][2l+2]}}{\sqrt{[4][2l]}}, \label{eqn:A2k-cell-W(l-1)ll} \\
W_{l,l,l} &= \frac{[4l+2]\sqrt{[2]}}{\sqrt{[4][2l][2l+2]}}, \\
W_{l,l,(l+1)} &= -\frac{\sqrt{[2][2l][2l+1][2l+3]}}{\sqrt{[4][2l+2]}}, \label{eqn:A2k-cell-Wll(l+1)}
\end{align}
for $l=1,2,\ldots,m-2$, and
\begin{align}
W_{m-2,m-1,m-1} &= \frac{\sqrt{[2][2m-3][2m-1][2m]}}{\sqrt{[4][2m-2]}}, \label{eqn:A2k-cell-W(k-1)kk} \\
W_{m-1,m-1,m-1} &= \frac{[4m-2]\sqrt{[2]}}{\sqrt{[4][2m-2][2m]}}, \label{eqn:A2k-cell-Wkkk} \\
W_{m-1,m-1,m_+} &= W_{m-1,m-1,m_-} = -\frac{\sqrt{[2][2m-2][2m-1][2m+1]}}{\sqrt{2[4][2m]}}, \\
W_{m-1,m_+,m_-} W_{m-1,m_-,m_+} &= \frac{[2m]^2}{[2]^2}.
\end{align}
If $\mathcal{M}$ is a $\ast$-module category, then any trivalent cell system is equivalent to the one given above, with $W_{m-1,m_+,m_-} = W_{m-1,m_-,m_+} = [2m]/[2]$.
\end{Thm}

\begin{proof}
For $\mathcal{A}_{2m}$ we have Perron-Frobenius weights $\phi_l = [2l+1]$, $l=0,1,\ldots,m-1$, and $\phi_{m_{\pm}} = [2m]/[2] = [2m+1]/2$.
Suppose
\begin{equation} \label{eqn:W(l-1)ll-1}
W_{l-1,l,l} = \varepsilon_l \frac{\sqrt{[2][2l-1][2l+1][2l+2]}}{\sqrt{[4][2l]}},
\end{equation}
where $\varepsilon_l \in \{ \pm 1 \}$.
Relation (\ref{eqn:cell-relations-2}) with $r(a)=l$ yields
\begin{equation} \label{eqn:Wll(l+1)-1}
W_{l,l,l+1} = \frac{-1}{\sqrt{[2l+3]}} \left( \varepsilon_l \frac{[2l-1]\sqrt{[2][2l+1][2l+2]}}{\sqrt{[4][2l]}} + \sqrt{[2l+1]} W_{l,l,l} \right).
\end{equation}
Substituting this expression for $W_{l,l,l+1}$ in relation (\ref{eqn:cell-relations-1}) with $s(a)=r(a)=l$ and $d=a$ yields the quadratic equation
$$[2][2l][2l+2]\sqrt{[4]} W_{l,l,l}^2 + 2 \varepsilon_l [2l-1][2l+1] \sqrt{[2][2l][2l+2]} W_{l,l,l} - \varepsilon_l^2 [2l+1]^2[4l+2]\sqrt{[4]} = 0,$$
which has solutions
\begin{align*}
W_{l,l,l} &= \frac{[2l+1]}{\sqrt{[2][4][2l][2l+2]}} (-[2l-1] \pm [2l+3]) \varepsilon_l \\
&= \frac{[2l+1]}{\sqrt{[2][4][2l][2l+2]}} (-[2l-1] + \epsilon_l [2l+3]) \varepsilon_l = \begin{cases} \disp \frac{-[2l+1]^2\sqrt{[4]}}{\sqrt{[2]^3[2l][2l+2]}}\varepsilon_l & \text{ if } \epsilon_l=-1, \\ \disp \frac{[4l+2]\sqrt{[2]}}{\sqrt{[4][2l][2l+2]}}\varepsilon_l & \text{ if } \epsilon_l=1. \end{cases}
\end{align*}
Then from (\ref{eqn:Wll(l+1)-1}),
\begin{align*}
W_{l,l,l+1} &= \begin{cases} \disp \frac{-\varepsilon_l}{\sqrt{[2l+3]}} \left( \frac{[2l-1]\sqrt{[2][2l+1][2l+2]}}{\sqrt{[4][2l]}}-\frac{\sqrt{[4][2l+1]^5}}{\sqrt{[2]^3[2l][2l+2]}} \right) & \text{ if } \epsilon_l=-1, \\ \disp \frac{-\varepsilon_l}{\sqrt{[2l+3]}} \left( \frac{[2l-1]\sqrt{[2][2l+1][2l+2]}}{\sqrt{[4][2l]}}+\frac{[4l+2]\sqrt{[2][2l+1]}}{\sqrt{[4][2l][2l+2]}} \right) & \text{ if } \epsilon_l=1, \end{cases} \\
&= \begin{cases} \disp \frac{[4l]\sqrt{[2l+1][2l+3]}}{[2]\sqrt{[4][2l]^3[2l+2]}} \varepsilon_l & \text{ if } \epsilon_l=-1, \\ \disp  \frac{-\sqrt{[2][2l][2l+1][2l+3]}}{\sqrt{[4][2l+2]}} \varepsilon_l & \text{ if } \epsilon_l=1. \end{cases}
\end{align*}
We now consider relation (\ref{eqn:cell-relations-3}) with $s(a)=l-1$, $r(a)=s(b)=l$, $r(b)=l+1$ and $c=a$, $d=b$. If $\epsilon_l=-1$, \eqref{eqn:cell-relations-3}) gives
$$\frac{[2l-1][2l+1][2l+3][4l]}{[4][2l]^2\sqrt{[2]}} = -\frac{[2][2l-1][2l+1][2l+3]}{[4]} \quad \Rightarrow \quad \frac{[4l]}{\sqrt{[2]^3} [2l]^2} = -1,$$
which is not possible since the L.H.S. is positive.
On the other hand, for $\epsilon_l = 1$ relation (\ref{eqn:cell-relations-3}) is satisfied.
Finally, relation (\ref{eqn:cell-relations-1}) with $s(a)=l$, $r(a)=l+1$ and $d=a$ yields
$$W_{l,l+1,l+1}^2 = [2l+1][2l+3] - \frac{[2][2l][2l+1][2l+3]}{[4][2l+2]} = \frac{[2][2l+1][2l+3][2l+4]}{[4][2l+2]},$$
so that
$$W_{l,l+1,l+1} = \varepsilon_{l+1} \frac{\sqrt{[2][2l+1][2l+3][2l+4]}}{{[4][2l+2]}},$$
with $\varepsilon_{l+1} \in \{ \pm 1 \}$ since $W_{l,l+1,l+1}^2 > 0$.
For the base step, relation (\ref{eqn:cell-relations-1}) with $s(a)=1$, $r(a)=2$ and $d=a$ gives $W_{1,2,2}^2 = [3]$, which yields (\ref{eqn:W(l-1)ll-1}) when $l=1$.

From relation (\ref{eqn:cell-relations-3}) with $s(a)=m_+$, $r(a)=s(b)=m-1$, $r(b)=m_-$ and $c=a$, $d=b$, we obtain
\begin{align*}
2W_{m-1,m-1,m_+} W_{m-1,m-1,m_-} &= \frac{[6][2m-1][2m+1]}{[3][4]} \\
\Rightarrow \qquad W_{m-1,m-1,m_-} &= \frac{[6][2m-1][2m+1]}{2W_{m-1,m-1,m_+}[3][4]}.
\end{align*}
Substituting this expression for $W_{m-1,m-1,m_-}$ in relation (\ref{eqn:cell-relations-2}) with $r(a)=m-1$ yields
\begin{align*}
& W_{m-1,m-1,m_+}^2 + \varepsilon_m \frac{\sqrt{2[2][2m-2][2m-1][2m+1]}}{\sqrt{[4][2m]}} W_{m-1,m-1,m_+} + \frac{[6][2m-1][2m+1]}{2[3][4]} = 0 \\
&\Rightarrow \quad W_{m-1,m-1,m_+} = -\varepsilon_m \frac{\sqrt{[2][2m-2][2m-1][2m+1]}}{\sqrt{2[4][2m]}}.
\end{align*}
The same procedure for $W_{m-1,m-1,m_-}$ shows that $W_{m-1,m-1,m_-} = W_{m-1,m-1,m_+}$. Relation (\ref{eqn:cell-relations-3}) with $s(a)=m-2$, $r(a)=s(b)=r(b)=m-1$ and $c=a$, $d=b$, yields (\ref{eqn:A2k-cell-Wkkk}).
Finally, relation (\ref{eqn:cell-relations-1}) with $s(a)=m_+$, $r(a)=m_-$ and $d=a$ yields $W_{m-1,m_+,m_-} W_{m-1,m_-,m_+} = [2m]^2/[2]^2$.
In a $\ast$-module category we have $W_{m-1,m_+,m_-} = \overline{W_{m-1,m_-,m_+}}$, thus $W_{m-1,m_+,m_-} = \tau [2m]/[2]$ for $\tau \in \bbT$.

We now show uniqueness of the cell system. Since all cells apart from $W_{m-1,m_{\pm},m_{\mp}}$ are fixed up to sign, it is clear that they are equivalent.
If $W$, $\hat{W}$ are two cell systems with $\hat{W}_{m-1,m_{\pm},m_{\mp}} = \alpha^{\pm1} W_{m-1,m_{\pm},m_{\mp}}$, $\alpha \in \bbC^{\times}$, equivalence requires the existence of $\eta^{(m-1,m_+)}, \eta^{(m_+,m_-)}, \eta^{(m_-,m-1)} \in \bbC$ such that
$\eta^{(m-1,m_+)} \eta^{(m_+,m_-)} \eta^{(m_-,m-1)} = \alpha$,
so we may take $\eta^{(m_+,m_-)} = \eta^{(m_-,m-1)} = 1$ and $\eta^{(m-1,m_+)} = \alpha$.
To be a $\ast$-module category we have the restriction $\alpha \in \bbT$, and hence $\eta^{(m_+,m-1)} = (\eta^{(m-1,m_+)})^{-1} = \overline{\eta^{(m-1,m_+)}}$ as required by \eqref{eqn:eta_eta_eta}.
\end{proof}

\begin{Thm} \label{thm:cell_system-E8}
Any trivalent cell system in a canonical module category $\mathcal{M}$ with nimrep graph $\mathcal{E}_8$ is equivalent to
\begin{align*}
W_{1,2,2} &= \sqrt{[3]}, &
W_{2,2,2} &= \frac{[6]}{[4]}, &
W_{2,2,3} &= -\frac{\sqrt{[2][3]}}{\sqrt{[4]}}, \\
W_{2,2,4} &= -\frac{\sqrt{[2][3][6]}}{[4]}, &
W_{2,4,4} &= \frac{\sqrt{[3][6][10]}}{[4]\sqrt{[5]}}, &
W_{3,4,4} &= \frac{\sqrt{[6][10]}}{\sqrt{[2][4][5]}}, \\
W_{4,4,4} &= -\frac{[6]\sqrt{[10]}}{\sqrt{[2]^3[5]}}, & W_{2,3,4} W_{2,4,3} &= \frac{[3][6]}{[4]}, &
\end{align*}
If $\mathcal{M}$ is a $\ast$-module category, then any trivalent cell system is equivalent to the one given above, with $W_{2,3,4} = W_{2,4,3} = \sqrt{[3][6]}/\sqrt{[4]}$.
\end{Thm}

\begin{proof}
The Perron-Frobenius weights for $\mathcal{E}_8$ are $\phi_1=1$, $\phi_2=[3]$, $\phi_3=[4]/[2]$, $\phi_4=[6]/[2]$.
Relation (\ref{eqn:cell-relations-1}) with $s(a)=1$, $r(a)=2$ and $d=a$ yields $W_{1,2,2}^2 = [3]$, thus $W_{1,2,2}=\varepsilon_1 \sqrt{[3]}$, $\varepsilon_1 \in \{ \pm 1 \}$.
From relation (\ref{eqn:cell-relations-3}) with $s(a)=1$, $r(a)=s(b)=r(b)=2$ and $c=a$, $d=b$, we obtain
$$\varepsilon_1 [3] W_{2,2,2} - [3]^2 = - \frac{[2][3]^2}{[4]} \qquad \Rightarrow \qquad W_{2,2,2} = \varepsilon_1 \frac{[6]}{[4]}.$$
Relation (\ref{eqn:cell-relations-2}) with $s(a)=2$ yields
$$\varepsilon_1 \sqrt{[3]} + \varepsilon_1 \frac{[6]\sqrt{[3]}}{[4]} + \frac{\sqrt{[4]}}{\sqrt{[2]}} W_{2,2,3} + \frac{\sqrt{[6]}}{\sqrt{[2]}} W_{2,2,4} = 0$$
so we can express $W_{2,2,4}$ in terms of $W_{2,2,3}$ as
\begin{equation} \label{eqn:W224-W223}
W_{2,2,4} = -\varepsilon_1 \frac{[5]\sqrt{[2]^3[3]}}{[4]\sqrt{[6]}} - \frac{\sqrt{[4]}}{\sqrt{[6]}} W_{2,2,3}.
\end{equation}
From (\ref{eqn:cell-relations-1}) with $s(a)=r(a)=2$ and $d=a$ we have
$$[3] + \frac{[6]^2}{[4]^2} + W_{2,2,3}^2 + W_{2,2,4}^2 = [3]^2.$$
Substituting in for $W_{2,2,4}$ from (\ref{eqn:W224-W223}) we obtain the quadratic equation
\begin{align*}
\frac{[2][5]}{[6]} W_{2,2,3}^2 + 2\varepsilon_1 \frac{[5]\sqrt{[2]^3[3]}}{[6]\sqrt{[4]}} W_{2,2,3} + \frac{[2]^2[3][5]}{[4][6]} &= 0, 
\end{align*}
i.e. $W_{2,2,3} = -\varepsilon_1 \sqrt{[2][3]}/\sqrt{[4]}$, and from (\ref{eqn:W224-W223}) we obtain $W_{2,2,4} = -\varepsilon_1 \sqrt{[2][3][6]}/[4]$.

Relation (\ref{eqn:cell-relations-1}) with $s(a)=2$, $r(a)=3$ and $d=a$ yields
$$\frac{[2][3]}{[4]} + W_{2,3,4} W_{2,4,3} = \frac{[3][4]}{[2]} \qquad \Rightarrow \qquad W_{2,3,4} W_{2,4,3} = \frac{[3][6]}{[4]}.$$
Relation (\ref{eqn:cell-relations-1}) with $s(a)=3$, $r(a)=4$ and $d=a$ yields
$$\frac{[3][6]}{[4]} + W_{3,4,4}^2 = \frac{[4][6]}{[2]^2} \qquad \Rightarrow \qquad W_{3,4,4} = \varepsilon_2 \frac{\sqrt{[6][10]}}{\sqrt{[2][4][5]}},$$
where $\varepsilon_2 \in \{ \pm 1 \}$.
From (\ref{eqn:cell-relations-3}) with $s(a)=2$, $r(a)=s(b)=4$, $r(b)=3$ and $c=a$, $d=b$, we obtain
$$\varepsilon_2 \frac{\sqrt{[3][6][10]}}{[2]\sqrt{[5]}} W_{2,4,4} - \frac{[3][6]^2}{[2][4]} = - \frac{[3][6]}{[2]} \qquad \Rightarrow \qquad W_{2,4,4} = \varepsilon_2 \frac{\sqrt{[3][6][10]}}{[4]\sqrt{[5]}}.$$
Finally, (\ref{eqn:cell-relations-2}) with $s(a)=4$ yields
$$\varepsilon_2 \frac{[3]\sqrt{[6][10]}}{[4]\sqrt{[5]}} + \varepsilon_2 \frac{\sqrt{[6][10]}}{[2]\sqrt{[5]}} + \frac{\sqrt{[6]}}{\sqrt{[2]}} W_{4,4,4} = 0,$$
which yields $W_{4,4,4} = - \varepsilon_2 [6]\sqrt{[10]}/\sqrt{[2]^3[5]}$.
Uniqueness follows as in the case of $\mathcal{A}_{2m}$.
\end{proof}

\begin{Thm} \label{thm:cell_system-E14c}
Any trivalent cell system in a canonical module category $\mathcal{M}$ with nimrep graph $\mathcal{E}_{14}^c$ is equivalent to
\begin{align*}
W_{1,1,1} &= -\frac{\sqrt{[4]}}{\sqrt{[2][3]}}, &
W_{1,1,2} &= \frac{1}{\sqrt{[3]}}, &
W_{1,2,2} &= \frac{\sqrt{[5]}}{\sqrt{[3]}}, \\
W_{2,2,2} &= \frac{[8]}{\sqrt{[2][3][4][5]}}, &
W_{2,2,3} &= -\frac{[3]\sqrt{[4]}}{[5]\sqrt{[2]}}, &
W_{2,2,4} &= -\frac{\sqrt{[2][3][4]}}{[5]}, \\
W_{2,4,4} &= \frac{\sqrt{[2][4]}}{\sqrt{[3][5]}}, &
W_{3,4,4} &= \frac{\sqrt{[2][4]}}{[5]}, &
W_{4,4,4} &= -\frac{\sqrt{[2][4]^3}}{[5]\sqrt{[3]}}, \\
W_{2,3,4} W_{2,4,3} &= \frac{[2]^2[3][4]^2}{[5]^3}.
\end{align*}
If $\mathcal{M}$ is a $\ast$-module category, then any trivalent cell system is equivalent to the one given above, with $W_{2,3,4} = W_{2,4,3} = [2][4]\sqrt{[3]}/\sqrt{[5]^3}$.
\end{Thm}

\begin{proof}
The Perron-Frobenius weights for $\mathcal{E}_{14}^c$ are $\phi_1=1$, $\phi_2=[4]/[2]$, $\phi_3=[3][4]/[2][5]$, $\phi_4=[2][4]/[5]$.
Relation (\ref{eqn:cell-relations-1}) with $s(a)=r(a)=1$ and $d=a$ yields $W_{1,1,1}^2 + W_{1,1,2}^2 = 1$, whilst from (\ref{eqn:cell-relations-3}) with $s(a)=r(a)=s(b)=1$, $r(b)=2$ and $c=a$, $d=b$, we obtain
$$\frac{\sqrt{[4]}}{\sqrt{[2]}} W_{1,1,1}W_{1,1,2} - W_{1,1,2}^2 = - 1.$$
Solving these two equations yields two possibilities:
$$W_{1,1,1} = \begin{cases} -\varepsilon_1 \frac{\sqrt{[4]}}{\sqrt{[2][3]}}, & \text{ Case I}, \\ 0, & \text{ Case II}, \end{cases} \qquad
W_{1,1,2} = \begin{cases} \frac{\varepsilon_1}{\sqrt{[3]}}, & \text{ Case I}, \\ \varepsilon_1, & \text{ Case II}. \end{cases}$$
From relation (\ref{eqn:cell-relations-1}) with $s(a)=1$, $r(a)=2$ and $d=a$ we obtain $W_{1,1,2}^2 + W_{1,2,2}^2 = [4]/[2]$, giving
$$W_{1,2,2} = \begin{cases} \varepsilon_2 \frac{\sqrt{[6]}}{\sqrt{[2][3]}}, & \text{ Case I}, \\ \varepsilon_2 \frac{\sqrt{[5]}}{\sqrt{[3]}}, & \text{ Case II}. \end{cases}$$
Relation (\ref{eqn:cell-relations-3}) with $s(a)=1$, $r(a)=s(b)=r(b)=2$ and $c=a$, $d=b$, gives
$$W_{1,2,2}W_{2,2,2} - \frac{\sqrt{[4]}}{\sqrt{[2]}} W_{1,2,2}^2 = - \frac{\sqrt{[4]}}{\sqrt{[2]}}, \qquad \Rightarrow \qquad
W_{2,2,2} = \begin{cases} \varepsilon_2 \frac{[8]}{\sqrt{[2][3][4][5]}}, & \text{ Case I}, \\ - \frac{\varepsilon_2}{[2]\sqrt{[4][5][6]}}, & \text{ Case II}, \end{cases}$$
since $[6]-[2][3]=-1/[4]$. From (\ref{eqn:cell-relations-3}) with $s(a)=1$, $r(a)=s(b)=2$, $r(b)=3$ and $c=a$, $d=b$, we get
$$W_{1,2,2} W_{2,2,3} = - \frac{\sqrt{[3][4]}}{\sqrt{[2][5]}} \qquad \Rightarrow \qquad
W_{2,2,3} = \begin{cases} -\varepsilon_2 \frac{[3]\sqrt{[4]}}{[5]\sqrt{[2]}}, & \text{ Case I}, \\ -\varepsilon_2 \frac{[3]\sqrt{[4]}}{\sqrt{[5][6]}}. & \text{ Case II}. \end{cases}$$
However, Case II is not consistent with relation (\ref{eqn:cell-relations-3}) with $s(a)=r(a)=s(b)=2$, $r(b)=3$ and $c=a$, $d=b$, thus only Case I is possible.
From (\ref{eqn:cell-relations-3}) with $s(a)=1$, $r(a)=s(b)=2$, $r(b)=4$ and $c=a$, $d=b$, we get
$$W_{1,2,2} W_{2,2,4} = - \frac{\sqrt{[2][4]}}{\sqrt{[5]}} \qquad \Rightarrow \qquad
W_{2,2,4} = -\varepsilon_2 \frac{\sqrt{[2][3][4]}}{[5]}.$$
Relation (\ref{eqn:cell-relations-3}) with $s(a)=3$, $r(a)=s(b)=2$, $r(b)=4$ and $c=a$, $d=b$, gives
$$\frac{[3]^2[4]^2}{[5]^3} - \frac{[4]}{[2]} \, W_{2,3,4}W_{2,4,3} = -\frac{[3][4]^2}{[5]^2} \qquad \Rightarrow \qquad W_{2,3,4}W_{2,4,3} = \frac{[2]^2[3][4]^2}{[5]^3}.$$
Relation (\ref{eqn:cell-relations-1}) with $s(a)=3$, $r(a)=4$ and $d=a$ yields
$$\frac{[2]^2[3][4]^2}{[5]^3} + W_{3,4,4}^2 = \frac{[3][4]^2}{[5]^2} \qquad \Rightarrow \qquad W_{3,4,4} = \varepsilon_3 \frac{\sqrt{[2][3][4][12]}}{\sqrt{[5]^3[6]}} = \varepsilon_3 \frac{\sqrt{[2][4]}}{[5]},$$
where the last equality follows from the relation $[3][12]=[5][6]$ when $q = \exp(i \pi/30)$. Then from (\ref{eqn:cell-relations-3}) with $s(a)=s(c)=r(c)=s(d)=4$, $r(a)=s(b)=2$, $r(b)=r(d)=3$, we have
$$W_{2,4,4} = \varepsilon_3 \frac{\sqrt{[2][4]}}{\sqrt{[3][5]}}.$$
Finally, (\ref{eqn:cell-relations-2}) with $s(a)=r(a)=4$ yields
$$\varepsilon_3 \frac{[4]}{\sqrt{[3][5]}} + \varepsilon_3 \frac{[4]\sqrt{[3]}}{\sqrt{[5]^3}} + \frac{\sqrt{[2][4]}}{\sqrt{[5]}} W_{4,4,4} = 0 \qquad \Rightarrow \qquad
W_{4,4,4} = -\varepsilon_3 \frac{\sqrt{[2][4]^3}}{[5]\sqrt{[3]}}.$$
Uniqueness follows as in the case of $\mathcal{A}_{2m}$.
\end{proof}

For the $SO(3)$ graphs $\sigma_{2m}$, $\mathcal{E}_8^c$ and $\mathcal{E}_{14}$ we only have results for $\ast$-module categories, where we have the identification $W(\tilde{c}\tilde{b}\tilde{a}) = \overline{W(abc)}$. Since all closed loops of length 3 on these graphs are of the form $a\tilde{a}b$ where $b$ is a self-loop (so $b = \tilde{b}$), any cell system must be real, since $W(a\tilde{a}b) = \overline{W(\tilde{b}a\tilde{a})} = \overline{W(ba\tilde{a})} = \overline{W(a\tilde{a}b)}$.

\begin{Thm} \label{thm:cell_system-sigma2m}
Any trivalent cell system in a canonical $\ast$-module category $\mathcal{M}$ with nimrep graph $\sigma_{2m}$ is equivalent to the cell system given by
\begin{align*}
W_{1,1,1} &= \frac{\sqrt{[4]}}{\sqrt{[2][3]}}, &
W_{1,1,2} &= -\frac{1}{\sqrt{[3]}}, \\
W_{(l-1),l,l} &= \frac{\sqrt{[2l-2][2l][2l+1]}}{\sqrt{[2][4][2l-1]}}, &
W_{l,l,l} &= \frac{[4l]}{\sqrt{[2][4][2l-1][2l+1]}}, \\
W_{l,l,(l+1)} &= -\frac{\sqrt{[2l-1][2l][2l+2]}}{\sqrt{[2][4][2l+1]}}, \\
W_{(m-1),m,m}^{\gamma} &= W_{(m-1),m,m}^{\gamma'} = \frac{[6]\sqrt{[2m]}^3}{2[2][3][4]\sqrt{[2m-2]}}, \\
W(\gamma,\gamma,\gamma) &= W(\gamma',\gamma',\gamma') = -\frac{[2][6]}{2[3][4][2m-2]}, &
W(\gamma,\gamma,\gamma') &=  W(\gamma,\gamma',\gamma') = -\frac{[2m]}{2[2]}.
\end{align*}
where $W_{(m-1),m,m}^{\eta}$ denotes the cell for the closed loop $(m-1)\to m \to m \to (m-1)$ which goes along the loop $\eta \in \{ \gamma, \gamma' \}$.
\end{Thm}

\begin{proof}
The cells $W_{1,1,2}$ and $W_{(l-1),l,l}$, $W_{l,l,l}$, $W_{l,l,(l+1)}$ for $l=2,3,\ldots, m-1$, can be determined in a similar way to those for $\mathcal{A}_{2m}$. Uniqueness of these cells also follows similarly.

For cells involving the double self-loop, we set $W_0 := W(\gamma,\gamma,\gamma)$, $W_1 := W(\gamma,\gamma,\gamma')$, $W_2 := W(\gamma,\gamma',\gamma')$ and $W_3 := W(\gamma',\gamma',\gamma')$, which are all real since $W(\tilde{c}\tilde{b}\tilde{a}) = \overline{W(abc)}$ in a $\ast$-module category and $\tilde{\gamma}=\gamma$, $\tilde{\gamma}'=\gamma'$.
From (\ref{eqn:cell-relations-2}) with $a=\gamma$ we obtain
\begin{equation} \label{eqn:Wm_gamma}
W_{(m-1),m,m}^{\gamma} = -\frac{\sqrt{[2m]}}{\sqrt{[2m-2]}} \, (W_0+W_2),
\end{equation}
and similarly with $a=\gamma'$ we obtain
\begin{equation} \label{eqn:Wm_gamma'}
W_{(m-1),m,m}^{\gamma'} = -\frac{\sqrt{[2m]}}{\sqrt{[2m-2]}} \, (W_1+W_3).
\end{equation}
From relation (\ref{eqn:cell-relations-1}) with $a=d=\gamma$, relation (\ref{eqn:cell-relations-1}) with $a=d=\gamma'$ and relation (\ref{eqn:cell-relations-3}) with $a=c=\gamma$ and $b=d=\gamma'$, we obtain that $W_i$ are any real solutions to the system of quadratic equations
\begin{align}
W_0^2 + 2\frac{[2m-2]}{[2][2m-1]} \, W_1^2 + W_2^2 + \frac{[2m+1]}{[2m-1]} \, W_0 W_2 &= \frac{[2m-2][2m]^2}{[2]^3[2m-1]}, \label{eqn:cell_equations-sigma2m-one_parameter-1} \\
W_1^2 + 2\frac{[2m-2]}{[2][2m-1]} \, W_2^2 + W_3^2 + \frac{[2m+1]}{[2m-1]} \, W_1 W_3 &= \frac{[2m-2][2m]^2}{[2]^3[2m-1]}, \\
W_1^2 + W_2^2 - W_0 W_2 - W_1 W_3 &= \frac{[2m]^2}{[2][4]}. \label{eqn:cell_equations-sigma2m-one_parameter-3}
\end{align}
All other relations (\ref{eqn:cell-relations-1})-(\ref{eqn:cell-relations-3}) involving $\gamma$, $\gamma'$ are satisfied when equations \eqref{eqn:cell_equations-sigma2m-one_parameter-1}-\eqref{eqn:cell_equations-sigma2m-one_parameter-3} are satisfied, and thus we have a one-parameter family of solutions.
We will first find all solutions where $W_2 = 0$. We will then show that these solutions are all equivalent, and that any other solution must also be equivalent to one of these. Finally we will deduce the solutions presented in the statement of the theorem.

Solving equations \eqref{eqn:cell_equations-sigma2m-one_parameter-1}-\eqref{eqn:cell_equations-sigma2m-one_parameter-3} with $W_2 = 0$ we obtain the solutions
\begin{equation} \label{soln:cell_system-sigma2m-W2=0}
W_0 = \frac{\epsilon_1 [2m]\sqrt{[6][2m-2]}}{[3]\sqrt{[2][4][2m-1]}}, \quad W_1 = \frac{\epsilon_2[2m]}{\sqrt{[2][3][4]}}, \quad W_3 = - \frac{\epsilon_2[2m-1]\sqrt{[2]}}{\sqrt{[3][4]}},
\end{equation}
where $\epsilon_1, \epsilon_2 \in \{ \pm 1 \}$.

We now turn to the equivalence of solutions for cells involving $\gamma$, $\gamma'$. Suppose $W$, $\widetilde{W}$ are two equivalent solutions, and denote by $u_{11} := \eta^{(m,m)}_{\gamma,\gamma}$, $u_{12} := \eta^{(m,m)}_{\gamma,\gamma'}$, $u_{21} := \eta^{(m,m)}_{\gamma',\gamma}$ and $u_{22} := \eta^{(m,m)}_{\gamma',\gamma'}$ the entries of the unitary, orthogonal linear transformation $\eta^{(m,m)}$ which gives this equivalence as in \eqref{eqn:equivalent-W}. Then from \eqref{eqn:equivalent-W} we have
\begin{align}
\widetilde{W}_{(m-1),m,m}^{\gamma} &= \eta^{(m-1,m)}\eta^{(m,m-1)} u_{11} W_{(m-1),m,m}^{\gamma} + \eta^{(m-1,m)}\eta^{(m,m-1)} u_{12} W_{(m-1),m,m}^{\gamma'} \nonumber \\
&= u_{11} W_{(m-1),m,m}^{\gamma} + u_{12} W_{(m-1),m,m}^{\gamma'}, \label{eqn:equivWijj-1} \\
\widetilde{W}_{(m-1),m,m}^{\gamma'} &= u_{21} W_{(m-1),m,m}^{\gamma} + u_{22} W_{(m-1),m,m}^{\gamma'}. \label{eqn:equivWijj-2}
\end{align}
Then from \eqref{eqn:Wm_gamma} and \eqref{eqn:Wm_gamma'} we obtain the following equations involving the $W_i$, $\widetilde{W}_i$:
\begin{align}
\widetilde{W}_0 + \widetilde{W}_2 &= u_{11} (W_0+W_2) + u_{12} (W_1+W_3), \label{eqn:equivalence-Wm_gamma} \\
\widetilde{W}_1 + \widetilde{W}_3 &= u_{21} (W_0+W_2) + u_{22} (W_1+W_3).
\end{align}
Equivalence between the $W_i$, $\widetilde{W}_i$ gives
\begin{align}
\widetilde{W}_0 &= u_{11}^3 W_0 + 3u_{11}^2 u_{12} W_1 + 3u_{11} u_{12}^2 W_2 + u_{12}^3 W_3, \label{eqn:equivalence-W0} \\
\widetilde{W}_1 &= u_{11}^2 u_{21} W_0 + (u_{11}^2 u_{22} + 2u_{11} u_{12}u_{21}) W_1 + (u_{12}^2 u_{21} + 2u_{11}u_{12}u_{22}) W_2 + u_{12}^2 u_{22} W_3, \\
\widetilde{W}_2 &= u_{11} u_{21}^2 W_0 + (u_{12} u_{21}^2 + 2u_{11} u_{21}u_{22}) W_1 + (u_{11} u_{22}^2 + 2u_{12}u_{21}u_{22}) W_2 + u_{12} u_{22}^2 W_3, \label{eqn:equivalence-W2} \\
\widetilde{W}_3 &= u_{21}^3 W_0 + 3u_{21}^2 u_{22} W_1 + 3u_{21} u_{22}^2 W_2 + u_{22}^3 W_3. \label{eqn:equivalence-W3}
\end{align}

Then it is easy to see that any solution for which $W_2=0$, i.e. any solution given by \eqref{soln:cell_system-sigma2m-W2=0} with $\epsilon_1=\epsilon_1'$, $\epsilon_2=\epsilon_2'$, for some $\epsilon_1', \epsilon_2' \in \{ \pm 1 \}$, is equivalent to the solution with $\epsilon_1=\epsilon_2=1$ -- that is, equations \eqref{eqn:equivalence-Wm_gamma}-\eqref{eqn:equivalence-W3} are satisfied -- by choosing $u_{11}=\epsilon_1'$, $u_{22}=\epsilon_2'$ and $u_{12}=u_{21}=0$.

We now show that any other solution $(\widetilde{W}_0,\widetilde{W}_1,\widetilde{W}_2,\widetilde{W}_3)$ is equivalent to the solution $(W_0,W_1,0,W_3)$ above with $\varepsilon_1=\varepsilon_2=1$. Let $\eta^{(m,m)} = u^{\theta}$ be the unitary given by setting $u_{11} = u_{22} = \cos \theta$, $u_{12} = \sin \theta$ and $u_{21} = -\sin \theta$, and let $(W_0^{\theta},W_1^{\theta},W_2^{\theta},W_3^{\theta})$ be the solution equivalent to $(\widetilde{W}_0,\widetilde{W}_1,\widetilde{W}_2,\widetilde{W}_3)$ obtained by using the unitary $u^{\theta}$ in \eqref{eqn:equivalent-W}. We thus have a continuous family of solutions $(W_0^{\theta},W_1^{\theta},W_2^{\theta},W_3^{\theta})$ for $\theta \in [0,2\pi)$. We show that there exists a choice of $\theta$ such that $W_2^{\theta} = 0$.
Now \eqref{eqn:equivalent-W} yields
$$W_2^{\theta} = \cos \theta \sin^2 \theta \widetilde{W}_0 + (\sin^3 \theta -2 \cos^2 \theta \sin \theta)\widetilde{W}_1 + (\cos^3 \theta-2\cos \theta \sin^2 \theta) \widetilde{W}_2 + \cos^2 \theta \sin \theta \widetilde{W}_3,$$
so that $W_2^0 = \widetilde{W}_2$, whilst $W_2^{\pi} = -\widetilde{W}_2$. Then by the intermediate value theorem there exists $\theta \in [0,\pi]$ such that $W_2^{\theta} = 0$. Thus the equivalent solution $(W_0^{\theta},W_1^{\theta},W_2^{\theta},W_3^{\theta})$ is one of the equivalent solutions given by (\ref{soln:cell_system-sigma2m-W2=0}).

Since all solutions are equivalent, we may seek one with a symmetry between the self-loops $\gamma$ and $\gamma'$ such that $W(\gamma,\gamma,\gamma) = W(\gamma',\gamma',\gamma')$ and $W(\gamma,\gamma,\gamma') =  W(\gamma,\gamma',\gamma')$. Solving \eqref{eqn:cell_equations-sigma2m-one_parameter-1}-\eqref{eqn:cell_equations-sigma2m-one_parameter-3} with $W_0=W_3$, $W_1=W_2$, one such solution is
\begin{align*}
W_{(m-1),m,m}^{\gamma} &= W_{(m-1),m,m}^{\gamma'} = \frac{[6]\sqrt{[2m]}^3}{2[2][3][4]\sqrt{[2m-2]}}, \\
W(\gamma,\gamma,\gamma) &= W(\gamma',\gamma',\gamma') = -\frac{[2][6]}{2[3][4][2m-2]}, \\
W(\gamma,\gamma,\gamma') &=  W(\gamma,\gamma',\gamma') = -\frac{[2m]}{2[2]}.
\end{align*}
\end{proof}

\begin{Rem} \label{Rem:conjecture-unique-sigma2m}
We conjecture that the trivalent cell system in any canonical module category with nimrep graph $\sigma_{2m}$ is equivalent to the one presented in Theorem \ref{thm:cell_system-sigma2m}. We have that $W_i$, $i=0,1,2,3$, are now any complex solutions to \eqref{eqn:cell_equations-sigma2m-one_parameter-1}-\eqref{eqn:cell_equations-sigma2m-one_parameter-3}. As in the proof of Theorem \ref{thm:cell_system-sigma2m}, we can construct a continuous family of equivalent solutions $(W_0^{\theta},W_1^{\theta},W_2^{\theta},W_3^{\theta})$ for $\theta \in [0,2\pi)$, where $W_2^0=\alpha \in \bbC$ and $W_2^{\pi} = -\alpha \in \bbC$. Thus by the intermediate value theorem any solution for the $W_i$ is equivalent to a solution where $W_2 \in \bbR$. If $W_2 = 0$ we have the solution given in \eqref{soln:cell_system-sigma2m-W2=0}, since solving the equations \eqref{eqn:cell_equations-sigma2m-one_parameter-1}-\eqref{eqn:cell_equations-sigma2m-one_parameter-3} for $W_2=0$ does not require any assumptions about $W_i \in \bbR$. If we now assume $0 \neq W_2 \in \bbR$, the imaginary parts of equations \eqref{eqn:cell_equations-sigma2m-one_parameter-1}-\eqref{eqn:cell_equations-sigma2m-one_parameter-3} yield the following possibilities only: (1) the $W_i$ are all real, which is the case considered in the proof of Theorem \ref{thm:cell_system-sigma2m}; (2) $W_0,W_2\in\bbR$ and $W_1,W_3\in i \bbR$; (3) $W_0 \in \bbC$ with $\text{Re}(W_0) =-W_2 [2m+1]/(2[2m-1])$, $W_3 \in \bbC$ with $\text{Re}(W_3) \neq 0$ and $\text{Im}(W_3) = \text{Re}(W_0) \text{Im}(W_1)/W_2$, and $W_1 = -i W_2 \text{Im}(W_0)/\text{Re}(W_3)$; or (4),
\begin{align*}
\text{Re}(W_0) &= W_2 \frac{[2](\text{Re}(W_3)^2 -3 \text{Re}(W_1)^2) - 2[4] \text{Re}(W_1) \text{Re}(W_3)}{[2]^3\text{Re}(W_1)^2 + [4]\text{Re}(W_3) (2\text{Re}(W_1) - \text{Re}(W_3))}, \\
\text{Im}(W_1) &= W_2 \text{Im}(W_0) \frac{[2]\text{Re}(W_1) + [4]\text{Re}(W_3)}{[2]^3\text{Re}(W_1)^2 + [4] \text{Re}(W_3) (2\text{Re}(W_1) - \text{Re}(W_3))}, \\
\text{Im}(W_3) &= 2 \text{Im}(W_1) - \frac{W_2 \text{Im}(W_0) + \text{Re}(W_3) \text{Im}(W_1)}{\text{Re}(W_1)},
\end{align*}
where the denominators on the R.H.S. are all non-zero, and $\text{Im}(W_0)$, $\text{Re}(W_1)$, $W_2$ and $\text{Re}(W_3)$ are otherwise arbitrary. Cases (2) and (3) can be shown explicitly not to give any solutions to the full equations \eqref{eqn:cell_equations-sigma2m-one_parameter-1}-\eqref{eqn:cell_equations-sigma2m-one_parameter-3}. For case (4), numerical evidence for $m=2,20,50$ suggests that no solutions exist for any value of $m$, but we have not been able to verify this.
\end{Rem}

\begin{Thm} \label{thm:cell_system-E8c}
Any trivalent cell system in a canonical $\ast$-module category $\mathcal{M}$ with nimrep graph $\mathcal{E}_8^c$ is equivalent to either the cell system $W$ given by
\begin{align*}
W_{1,2,2}^{\gamma} &=0, &
W_{1,2,2}^{\gamma'} &=\sqrt{[3]}, \\
W_{2,2,3}^{\gamma} &= \frac{[5]}{[3]}, &
W_{2,2,3}^{\gamma'} &= -\frac{[5]\sqrt{[2]}}{[3]\sqrt{[4]}}, \\
W(\gamma,\gamma,\gamma) &= -\frac{[5]\sqrt{[2]}}{[3]\sqrt{[4]}}, &
W(\gamma,\gamma,\gamma') &= -\frac{[5]}{[3]}, \\
W(\gamma,\gamma',\gamma') &= 0, &
W(\gamma',\gamma',\gamma') &= \frac{[6]}{[4]}, \\
W_{2,3,3} &= \frac{[2]\sqrt{[3]}}{[4]}, &
W_{3,3,3} &= -\frac{\sqrt{[2][3]}}{[4]},
\end{align*}
or the cell system $\widehat{W}$ given by
\begin{align*}
\widehat{W}_{1,2,2}^{\gamma} &=\frac{\sqrt{[2][5]}}{\sqrt{[4]}}, &
\widehat{W}_{1,2,2}^{\gamma'} &=\frac{\sqrt{[2]}}{\sqrt{[4]}}, \\
\widehat{W}_{2,2,3}^{\gamma} &= -\frac{\sqrt{[2]^3[6]}}{[4]}, &
\widehat{W}_{2,2,3}^{\gamma'} &= \frac{[2]\sqrt{[3]}}{[4]}, \\
\widehat{W}(\gamma,\gamma,\gamma) &= \frac{[2]}{[4]\sqrt{[3]}}, &
\widehat{W}(\gamma,\gamma,\gamma') &= \frac{\sqrt{[5]}}{\sqrt{[3]}}, \\
\widehat{W}(\gamma,\gamma',\gamma') &= 0, &
\widehat{W}(\gamma',\gamma',\gamma') &= -\frac{\sqrt{[3][4]}}{\sqrt{[2]}}, \\
\widehat{W}_{2,3,3} &= \frac{[2]\sqrt{[3]}}{[4]}, &
\widehat{W}_{3,3,3} &= -\frac{\sqrt{[2][3]}}{[4]},
\end{align*}
where $W_{i,2,2}^{\eta}$, $\widehat{W}_{i,2,2}^{\eta}$ denote the cell for the closed loop $i\to 2 \to 2 \to i$ ($i=1,3$) which goes along the loop $\eta \in \{ \gamma, \gamma' \}$.
The cell systems $W$ and $\widehat{W}$ are inequivalent.
\end{Thm}

\begin{proof}
The Perron-Frobenius weights for $\mathcal{E}_8^c$ are $\phi_1=1$, $\phi_2=[3]$, $\phi_3=[2][3]/[4]$.
Relation (\ref{eqn:cell-relations-2}) with $s(a)=r(a)=3$ yields
$$\sqrt{[3]} W_{2,3,3} + \frac{\sqrt{[2][3]}}{\sqrt{[4]}} W_{3,3,3} = 0 \qquad \Rightarrow \qquad W_{3,3,3} = -\frac{\sqrt{[4]}}{\sqrt{[2]}} W_{2,3,3}.$$
Then relation (\ref{eqn:cell-relations-1}) with $s(a)=r(a)=3$ and $d=a$ gives
$$W_{2,3,3}^2 + W_{3,3,3}^2 = \frac{[2]^2[3]^2}{[4]^2} \qquad \Rightarrow \qquad W_{2,3,3} = \varepsilon_3 \frac{[2]\sqrt{[3]}}{[4]}, \quad W_{3,3,3} = -\varepsilon_3 \frac{\sqrt{[2][3]}}{\sqrt{[4]}},$$
where $\varepsilon_3 \in \{ \pm 1 \}$.

For the remaining cells, we obtain the following equations, where we set $W_0:=W(\gamma,\gamma,\gamma)$, $W_1:=W(\gamma,\gamma,\gamma')$, $W_2:=W(\gamma,\gamma',\gamma')$ and $W_3:=W(\gamma',\gamma',\gamma')$.
From relation (\ref{eqn:cell-relations-1}) with $s(a)=1$, $r(a)=2$ and $d=a$ we obtain
\begin{equation} \label{eqn:cells_E8c-1}
(W_{1,2,2}^{\gamma})^2 + (W_{1,2,2}^{\gamma'})^2 = [3],
\end{equation}
and similarly with $s(a)=2$, $r(a)=3$ and $d=a$ we obtain
\begin{equation}
(W_{2,2,3}^{\gamma})^2 + (W_{2,2,3}^{\gamma'})^2 = \frac{\sqrt{[2]^3[3]^3}}{\sqrt{[4]^3}}.
\end{equation}
Relation (\ref{eqn:cell-relations-1}) with $a=d=\gamma$ yields
\begin{equation}
(W_{1,2,2}^{\gamma})^2 + W_0^2 + 2W_1^2 + W_2^2 + (W_{2,2,3}^{\gamma})^2 = [3]^2,
\end{equation}
and similarly with $a=d=\gamma'$ we obtain
\begin{equation}
(W_{1,2,2}^{\gamma'})^2 + W_1^2 + 2W_2^2 + W_3^2 + (W_{2,2,3}^{\gamma'})^2 = [3]^2,
\end{equation}
whilst $a=\gamma$, $d=\gamma'$ yields
\begin{align}
W_{1,2,2}^{\gamma} W_{1,2,2}^{\gamma'} + W_0 W_1 + 2W_1 W_2 + W_2 W_3 + W_{2,2,3}^{\gamma} W_{2,2,3}^{\gamma'} = 0.
\end{align}
Relation (\ref{eqn:cell-relations-2}) with $a=\gamma$ yields
\begin{equation}
W_{1,2,2}^{\gamma} + \sqrt{[3]} W_0 + \sqrt{[3]} W_2 + \frac{\sqrt{[2][3]}}{\sqrt{[4]}} W_{2,2,3}^{\gamma} = 0,
\end{equation}
and similarly with $a=\gamma'$ yields
\begin{equation}
W_{1,2,2}^{\gamma'} + \sqrt{[3]} W_1 + \sqrt{[3]} W_3 + \frac{\sqrt{[2][3]}}{\sqrt{[4]}} W_{2,2,3}^{\gamma'} = 0.
\end{equation}
Relation (\ref{eqn:cell-relations-3}) with $s(a)=1$, $r(a)=s(b)=2$, $r(b)=3$ and $c=a$, $d=b$ yields
\begin{equation}
W_{1,2,2}^{\gamma} W_{2,2,3}^{\gamma} + W_{1,2,2}^{\gamma'} W_{2,2,3}^{\gamma'} = -\frac{\sqrt{[2]^3[3]^3}}{\sqrt{[4]^3}}.
\end{equation}
Relation (\ref{eqn:cell-relations-3}) with $s(a)=1$, $r(a)=2$, $b=\gamma$, $c=a$, and $d=b$ yields
\begin{equation}
W_{1,2,2}^{\gamma} W_0 + W_{1,2,2}^{\gamma'} W_1 - \sqrt{[3]} (W_{1,2,2}^{\gamma})^2 = -\frac{[2]\sqrt{[3]^3}}{[4]},
\end{equation}
and similarly, with $s(a)=1$, $r(a)=2$, $b=\gamma'$, $c=a$, and $d=b$ we have
\begin{equation}
W_{1,2,2}^{\gamma} W_2 + W_{1,2,2}^{\gamma'} W_3 - \sqrt{[3]} (W_{1,2,2}^{\gamma'})^2 = -\frac{[2]\sqrt{[3]^3}}{[4]}.
\end{equation}
Similarly, with $s(a)=3$ and $r(a)=2$ we obtain
\begin{align}
W_{2,2,3}^{\gamma} W_0 + W_{2,2,3}^{\gamma'} W_1 - \frac{\sqrt{[4]}}{\sqrt{[2]}} (W_{2,2,3}^{\gamma})^2 &= -\frac{[3]^2\sqrt{[2]^3}}{\sqrt{[4]^3}}, \\
W_{2,2,3}^{\gamma} W_2 + W_{2,2,3}^{\gamma'} W_3 - \frac{\sqrt{[4]}}{\sqrt{[2]}} (W_{2,2,3}^{\gamma'})^2 &= -\frac{[3]^2\sqrt{[2]^3}}{\sqrt{[4]^3}}.
\end{align}
Finally, relation (\ref{eqn:cell-relations-3}) with $s(a)=1$, $r(a)=2$, $c=a$, $b=\gamma$ and $d=\gamma'$ yields
\begin{equation}
W_{1,2,2}^{\gamma} W_1 + W_{1,2,2}^{\gamma'} W_2 - \sqrt{[3]} W_{1,2,2}^{\gamma} W_{1,2,2}^{\gamma'} = 0,
\end{equation}
and similarly, with $s(a)=3$, $r(a)=2$, $c=a$, $b=\gamma$ and $d=\gamma'$ yields
\begin{equation} \label{eqn:cells_E8c-14}
W_{2,2,3}^{\gamma} W_1 + W_{2,2,3}^{\gamma'} W_2 - \frac{\sqrt{[4]}}{\sqrt{[2]}} W_{2,2,3}^{\gamma} W_{2,2,3}^{\gamma'} = 0.
\end{equation}

Just as in the proof of Theorem \ref{thm:cell_system-sigma2m}, the cells $W_0, W_1, W_2, W_3$ in any equivalent cell system satisfies equations \eqref{eqn:equivalence-W0}-\eqref{eqn:equivalence-W3}, and by the same argument used in that proof, we see that any cell system is therefore equivalent to a cell system where $W_2 = 0$.
Solving equations \eqref{eqn:cells_E8c-1}-\eqref{eqn:cells_E8c-14} with $W_2=0$, we obtain the following solutions:
\begin{align*}
W_{1,2,2}^{\gamma} &= 0, & W_{1,2,2}^{\gamma'} &= \epsilon_1 \sqrt{[3]}, & W_{2,2,3}^{\gamma} &= -\epsilon_2 \frac{[5]}{[3]}, & W_{2,2,3}^{\gamma'} &= -\epsilon_1 \frac{[5]\sqrt{[2]}}{[3]\sqrt{[4]}}, \\
W_0 &= \epsilon_2 \frac{[5]\sqrt{[2]}}{[3]\sqrt{[4]}}, & W_1 &= -\epsilon_1 \frac{[5]}{[3]}, & W_2 &= 0, & W_3 &= \epsilon_1 \frac{[6]}{[4]},
\end{align*}
where $\epsilon_1, \epsilon_2 \in \{ \pm 1 \}$, or
\begin{align*}
\widetilde{W}_{1,2,2}^{\gamma} &= -\epsilon_1 \frac{[3]}{\sqrt{[5]}}, & \widetilde{W}_{1,2,2}^{\gamma'} &= -\epsilon_2 , & \widetilde{W}_{2,2,3}^{\gamma} &= 0, & \widetilde{W}_{2,2,3}^{\gamma'} &= \epsilon_2 \frac{\sqrt{[5]^3}}{\sqrt{[3]^3}}, \\
\widetilde{W}_0 &= \epsilon_1 \frac{\sqrt{[3]}}{\sqrt{[5]}}, & \widetilde{W}_1 &= -\epsilon_2 \sqrt{[3]}, & \widetilde{W}_2 &= 0, & \widetilde{W}_3 &= \epsilon_2 \frac{[2]\sqrt{[3]}}{[4]},
\end{align*}
where $\epsilon_1, \epsilon_2 \in \{ \pm 1 \}$, or
\begin{align*}
\widehat{W}_{1,2,2}^{\gamma} &= \epsilon_2 \frac{\sqrt{[2][5]}}{\sqrt{[4]}}, & \widehat{W}_{1,2,2}^{\gamma'} &= \epsilon_1 \frac{\sqrt{[2]}}{\sqrt{[4]}}, & \widehat{W}_{2,2,3}^{\gamma} &= -\epsilon_2 \frac{\sqrt{[2]^3[6]}}{[4]}, & \widehat{W}_{2,2,3}^{\gamma'} &= \epsilon_1 \frac{[2]\sqrt{[3]}}{[4]}, \\
\widehat{W}_0 &= \epsilon_2 \frac{[2]}{[4]\sqrt{[3]}}, & \widehat{W}_1 &= \epsilon_1 \frac{\sqrt{[5]}}{\sqrt{[3]}}, & \widehat{W}_2 &= 0, & \widehat{W}_3 &= -\epsilon_1 \frac{\sqrt{[3][4]}}{\sqrt{[2]}},
\end{align*}
where $\epsilon_1, \epsilon_2 \in \{ \pm 1 \}$.

The equivalence of the cell systems $W$ for different choices of $\epsilon_1, \epsilon_2 \in \{ \pm 1 \}$ follows as in the proof of Theorem \ref{thm:cell_system-sigma2m}, and similarly for $\widetilde{W}$ and $\widehat{W}$. Thus we will only consider cell systems with $\epsilon_1 = \epsilon_2 = 1$. Equivalence of cell systems $W$ and $\widetilde{W}$, for cells involving $\gamma$, $\gamma'$, is given by a unitary $u$ which satisfies relations \eqref{eqn:equivWijj-1} and \eqref{eqn:equivWijj-2} with $m=2$, similar equations for cells for the closed paths along vertices $2 \to 2 \to 3$, and relations \eqref{eqn:equivalence-Wm_gamma}-\eqref{eqn:equivalence-W3}. Equivalence of the cell systems $W$ and $\widetilde{W}$ above is then given by $u$ with $\disp u_{11} = u_{22} = -\frac{1}{\sqrt{[3]}}$ and $\disp u_{12} = -u_{21} = \frac{\sqrt{[3]}}{\sqrt{[5]}}$.
However, the cell systems $W$ and $\widehat{W}$ given above are not equivalent. This can be seen from \eqref{eqn:equivWijj-1}, the equivalent relation for $\widehat{W}^{\gamma}_{2,2,3}$, and \eqref{eqn:equivalence-Wm_gamma}, which give the inconsistent linear system
\begin{align*}
0 &= \frac{\sqrt{[2][5]}}{\sqrt{[4]}} u_{11} + \frac{\sqrt{[2]}}{\sqrt{[4]}} u_{12}, \\
\frac{[5]}{[3]} &= -\frac{\sqrt{[2]^3[6]}}{[4]} u_{11} + \frac{[2]\sqrt{[3]}}{[4]} u_{12}, \\
-\frac{[5]\sqrt{[2]}}{[3]\sqrt{[4]}} &= \frac{[2]}{[4]\sqrt{[3]}} u_{11} - \frac{[6]\sqrt{[5]}}{[2]\sqrt{[3]^3}}.
\end{align*}
\end{proof}

\begin{Thm} \label{thm:cell_system-E14}
Any trivalent cell system in a canonical $\ast$-module category $\mathcal{M}$ with nimrep graph $\mathcal{E}_{14}$ is equivalent to
\begin{align*}
W_{1,2,2} &= \sqrt{[3]}, \qquad &
W_{2,2,2} = \frac{[6]}{[4]}, & \qquad\qquad\quad
W_{2,2,3} = -\frac{[2]\sqrt{[3][5]}}{[4]}, \\
W_{2,3,3}^{\gamma} &= 0, &
W_{2,3,3}^{\gamma'} &= \frac{\sqrt{[2][3][5][6]}}{[4]}, \\
W_{3,3,4}^{\gamma} &= -\frac{[5]^3}{\sqrt{[4][6]}}, &
W_{3,3,4}^{\gamma'} &= -\frac{[5]\sqrt{[2]}}{\sqrt{[3][6]}}, \\
W(\gamma,\gamma,\gamma) &= \frac{\sqrt{[5]^3}}{\sqrt{[3][4][6]}}, &
W(\gamma,\gamma,\gamma') &= -\frac{[5]\sqrt{[2]}}{\sqrt{[6]}}, \\
W(\gamma,\gamma',\gamma') &= 0, &
W(\gamma',\gamma',\gamma') &= \frac{[10]\sqrt{[2]}}{[4]\sqrt{[6]}},
\end{align*}
where  $W_{i,3,3}^{\eta}$ denotes the cell for the closed loop $i\to 3 \to 3 \to i$ ($i=2,4$) which goes along the loop $\eta \in \{ \gamma, \gamma' \}$.
\end{Thm}

\begin{proof}
The Perron-Frobenius weights for $\mathcal{E}_{14}$ are $\phi_1=1$, $\phi_2=[3]$, $\phi_3=[5]$, $\phi_4=[12]/[6]$.
Relation (\ref{eqn:cell-relations-1}) with $s(a)=1$, $r(a)=2$ and $d=a$ yields $\left| W_{1,2,2} \right|^2 = [3]$. Now $W_{1,2,2} \in \bbR$ (see Remark \ref{rem:equivalent-unitary-W}), thus $W_{1,2,2}=\epsilon_1 \sqrt{[3]}$, $\epsilon_1 \in \{ \pm 1 \}$.
From relation (\ref{eqn:cell-relations-3}) with $s(a)=1$, $r(a)=s(b)=r(b)=2$ and $c=a$, $d=b$, we obtain
$$\epsilon_1 [3] W_{2,2,2} - [3]^2 = - \frac{[2][3]^2}{[4]} \qquad \Rightarrow \qquad W_{2,2,2} = \epsilon_1 \frac{[6]}{[4]}.$$
Relation (\ref{eqn:cell-relations-2}) with $s(a)=2$ yields
$$\epsilon_1 \sqrt{[3]} + \epsilon_1 \frac{[6]\sqrt{[3]}}{[4]} + \sqrt{[5]} W_{2,2,3} = 0 \qquad \Rightarrow \qquad W_{2,2,3} = - \epsilon_1 \frac{[2]\sqrt{[3][5]}}{[4]}.$$

For the remaining cells, we obtain the following equations, where we set $W_0:=W(\gamma,\gamma,\gamma)$, $W_1:=W(\gamma,\gamma,\gamma')$, $W_2:=W(\gamma,\gamma',\gamma')$ and $W_3:=W(\gamma',\gamma',\gamma')$.
From relation (\ref{eqn:cell-relations-1}) with $s(a)=2$, $r(a)=3$ and $d=a$ we obtain
\begin{equation} \label{eqn:cells_E14-1}
(W_{2,3,3}^{\gamma})^2 + (W_{2,3,3}^{\gamma'})^2 = \frac{[2][3][5][6]}{[4]^2},
\end{equation}
and similarly with $s(a)=3$, $r(a)=4$ and $d=a$ we obtain
\begin{equation}
(W_{3,3,4}^{\gamma})^2 + (W_{3,3,4}^{\gamma'})^2 = \frac{[5][12]}{[6]}.
\end{equation}
Relation (\ref{eqn:cell-relations-1}) with $a=d=\gamma$ yields
\begin{equation}
(W_{2,3,3}^{\gamma})^2 + W_0^2 + 2W_1^2 + W_2^2 + (W_{3,3,4}^{\gamma})^2 = [5]^2,
\end{equation}
and similarly with $a=d=\gamma'$ we obtain
\begin{equation}
(W_{2,3,3}^{\gamma'})^2 + W_1^2 + 2W_2^2 + W_3^2 + (W_{3,3,4}^{\gamma'})^2 = [5]^2,
\end{equation}
whilst $a=\gamma$, $d=\gamma'$ yields
\begin{align}
W_{2,3,3}^{\gamma} W_{2,3,3}^{\gamma'} + W_0 W_1 + 2W_1 W_2 + W_2 W_3 + W_{3,3,4}^{\gamma} W_{3,3,4}^{\gamma'} = 0.
\end{align}
Relation (\ref{eqn:cell-relations-2}) with $a=\gamma$ yields
\begin{equation}
\sqrt{[3]} W_{2,3,3}^{\gamma} + \sqrt{[5]} W_0 + \sqrt{[5]} W_2 + \frac{\sqrt{[12]}}{\sqrt{[6]}} W_{3,3,4}^{\gamma} = 0,
\end{equation}
and similarly with $a=\gamma'$ yields
\begin{equation}
\sqrt{[3]} W_{2,3,3}^{\gamma'} + \sqrt{[5]} W_1 + \sqrt{[5]} W_3 + \frac{\sqrt{[12]}}{\sqrt{[6]}} W_{3,3,4}^{\gamma'} = 0.
\end{equation}
Relation (\ref{eqn:cell-relations-3}) with $s(a)=2$, $r(a)=s(b)=3$, $r(b)=4$ and $c=a$, $d=b$ yields
\begin{equation}
W_{2,3,3}^{\gamma} W_{3,3,4}^{\gamma} + W_{2,3,3}^{\gamma'} W_{3,3,4}^{\gamma'} = -\frac{[2][5]\sqrt{[3][12]}}{[4]\sqrt{[6]}}.
\end{equation}
Relation (\ref{eqn:cell-relations-3}) with $s(a)=2$, $r(a)=3$, $b=\gamma$, $c=a$, and $d=b$ yields
\begin{equation}
W_{2,3,3}^{\gamma} W_0 + W_{2,3,3}^{\gamma'} W_1 - \frac{\sqrt{[5]}}{\sqrt{[3]}} (W_{2,3,3}^{\gamma})^2 = -\frac{[2]\sqrt{[3][5]^3}}{[4]},
\end{equation}
and similarly, with $s(a)=2$, $r(a)=3$, $b=\gamma'$, $c=a$, and $d=b$ we have
\begin{equation}
W_{2,3,3}^{\gamma} W_2 + W_{2,3,3}^{\gamma'} W_3 - \frac{\sqrt{[5]}}{\sqrt{[3]}} (W_{2,3,3}^{\gamma'})^2 = -\frac{[2]\sqrt{[3][5]^3}}{[4]}.
\end{equation}
Similarly, with $s(a)=4$ and $r(a)=3$ we obtain
\begin{align}
W_{3,3,4}^{\gamma} W_0 + W_{3,3,4}^{\gamma'} W_1 - \frac{\sqrt{[5][6]}}{\sqrt{[12]}} (W_{3,3,4}^{\gamma})^2 &= -\frac{[2]\sqrt{[5]^3[12]}}{[4]\sqrt{[6]}}, \\
W_{3,3,4}^{\gamma} W_2 + W_{3,3,4}^{\gamma'} W_3 - \frac{\sqrt{[5][6]}}{\sqrt{[12]}} (W_{3,3,4}^{\gamma'})^2 &= -\frac{[2]\sqrt{[5]^3[12]}}{[4]\sqrt{[6]}}.
\end{align}
Finally, relation (\ref{eqn:cell-relations-3}) with $s(a)=2$, $r(a)=3$, $c=a$, $b=\gamma$ and $d=\gamma'$ yields
\begin{equation}
W_{2,3,3}^{\gamma} W_1 + W_{2,3,3}^{\gamma'} W_2 - \frac{\sqrt{[5]}}{\sqrt{[3]}} W_{2,3,3}^{\gamma} W_{2,3,3}^{\gamma'} = 0,
\end{equation}
and similarly, with $s(a)=4$, $r(a)=3$, $c=a$, $b=\gamma$ and $d=\gamma'$ yields
\begin{equation} \label{eqn:cells_E14-14}
W_{3,3,4}^{\gamma} W_1 + W_{3,3,4}^{\gamma'} W_2 - \frac{\sqrt{[5][6]}}{\sqrt{[12]}} W_{3,3,4}^{\gamma} W_{3,3,4}^{\gamma'} = 0.
\end{equation}

Just as in the proof of Theorem \ref{thm:cell_system-sigma2m}, the cells $W_0, W_1, W_2, W_3$ in any equivalent cell system satisfies equations \eqref{eqn:equivalence-W0}-\eqref{eqn:equivalence-W3}, and by the same argument used in that proof, we see that any cell system is therefore equivalent to a cell system where $W_2 = 0$.
Solving equations \eqref{eqn:cells_E14-1}-\eqref{eqn:cells_E14-14} with $W_2=0$, we obtain the following solutions:
\begin{align*}
W_{2,3,3}^{\gamma} &= 0, & W_{2,3,3}^{\gamma'} &= \epsilon_1 \frac{\sqrt{[2][3][5][6]}}{[4]}, & W_{3,3,4}^{\gamma} &= -\epsilon_2 \frac{\sqrt{[5]^3}}{\sqrt{[4][6]}}, & W_{3,3,4}^{\gamma'} &= -\epsilon_1 \frac{[5]\sqrt{[2]}}{\sqrt{[3][6]}}, \\
W_0 &= \epsilon_2 \frac{\sqrt{[5]^3}}{\sqrt{[3][4][6]}}, & W_1 &= -\epsilon_1 \frac{[5]\sqrt{[2]}}{\sqrt{[6]}}, & W_2 &= 0, & W_3 &= \epsilon_1 \frac{[10]\sqrt{[2]}}{[4]\sqrt{[6]}},
\end{align*}
where $\epsilon_1, \epsilon_2 \in \{ \pm 1 \}$, or
\begin{align*}
\widetilde{W}_{2,3,3}^{\gamma} &= -\epsilon_2 \frac{[2]\sqrt{[3][5][8]}}{\sqrt{[4]^3}}, & \widetilde{W}_{2,3,3}^{\gamma'} &= -\epsilon_1 \frac{[2]\sqrt{[3][5]}}{[4]}, & \widetilde{W}_{3,3,4}^{\gamma} &= 0, & \widetilde{W}_{3,3,4}^{\gamma'} &= \epsilon_1 \frac{\sqrt{[5][12]}}{\sqrt{[6]}}, \\
\widetilde{W}_0 &= \epsilon_2 \frac{[2][3]\sqrt{[8]}}{\sqrt{[4]^3}}, & \widetilde{W}_1 &= -\epsilon_1 \frac{[2][5]}{[4]}, & \widetilde{W}_2 &= 0, & \widetilde{W}_3 &= \epsilon_1 \frac{[2][6]}{[5]},
\end{align*}
where $\epsilon_1, \epsilon_2 \in \{ \pm 1 \}$, or
\begin{align*}
\widehat{W}_{2,3,3}^{\gamma} &= -\epsilon_2 \frac{[7]\sqrt{[3]}}{[4]}, & \widehat{W}_{2,3,3}^{\gamma'} &= \epsilon_1 \frac{[3]\sqrt{[2]}}{\sqrt{[8]}}, & \widehat{W}_{3,3,4}^{\gamma} &= \epsilon_2 \frac{[2]\sqrt{[5]}}{\sqrt{[3]}}, & \widehat{W}_{3,3,4}^{\gamma'} &= \epsilon_1 \frac{\sqrt{[2][5]}}{\sqrt{[8]}}, \\
\widehat{W}_0 &= \epsilon_2 \frac{[2]\sqrt{[7]}}{[4]\sqrt{[3]}}, & \widehat{W}_1 &= \epsilon_1 \frac{\sqrt{[2][3][5]}}{\sqrt{[8]}}, & \widehat{W}_2 &= 0, & \widehat{W}_3 &= -\epsilon_1 [4],
\end{align*}
where $\epsilon_1, \epsilon_2 \in \{ \pm 1 \}$.

The equivalence of the cell systems $W$ for different choices of $\epsilon_1, \epsilon_2 \in \{ \pm 1 \}$ follows as in the proof of Theorem \ref{thm:cell_system-sigma2m}, and similarly for the equivalence of the cell systems $\widetilde{W}$, $\widehat{W}$. 
Equivalence of $W$ and $\widetilde{W}$ is given by the unitary $u$ with $\disp u_{11} = -u_{22} = \frac{[5]}{\sqrt{[8][12]}}$ and $\disp u_{12} = u_{21} = -\frac{[5]}{\sqrt{[4][12]}}$. Equivalence of $\widetilde{W}$ and $\widehat{W}$ ig given by $u$ with $\disp u_{11} = u_{22} = \frac{\sqrt{[2][6]}}{\sqrt{[8][12]}}$ and $\disp u_{12} = -u_{21} = \frac{[2]\sqrt{[6]}}{\sqrt{[3][12]}}$.
\end{proof}

\subsection{Classification of module categories for $SO(3)_{2m}$} \label{sect:T-cell-system}

As discussed in Section \ref{sect:construct-cat}, there is a trace $\text{tr}_T$ on the $SO(3)_{2m}$ algebra $\mathcal{W}_{n} = \varphi(\mathcal{V}_n^{(2)})$ defined on diagrams by attaching a string joining the $j^{\mathrm{th}}$ vertex along the top with the $j^{\mathrm{th}}$ vertex along the bottom, for each $j\in\{1,2,\ldots,n\}$, with each resulting closed loop contributing a factor of $[3]_q$. As noted in Section \ref{sect:SO3-cat}, we have $\text{tr}_T(\mathfrak{f}_{2m})=[4m+1]_q = 1$.

Let $A_0 \subset A_1 \subset A_2 \subset \cdots$ be the path algebra for $\mathcal{G}$, as in Section \ref{sect:cell-systems}, so that the algebra $A_j$ is the space of endomorphisms $F(\mathcal{W}_{j})$, where $F$ is the tensor functor defined as in \eqref{eqn:functorF}, and has basis indexed by all pairs $(\alpha_1,\alpha_2)$ of paths of length $j$ on $\mathcal{G}$, where $s(\alpha_1)=s(\alpha_2)$ and $r(\alpha_1)=r(\alpha_2)$.
By \cite[Theorem 6.1]{effros:1981} there is a unique normalised faithful trace $\text{tr}_A$ on $\bigcup_j A_j$, defined as in \cite{evans/kawahigashi:1994} by
\begin{equation} \label{eqn:trace}
\mathrm{tr}_A((\alpha_1,\alpha_2)) = \delta_{\alpha_1, \alpha_2} [3]_q^{-j} \phi_{r(\alpha_1)} \phi_{s(\alpha_1)}^{-1},
\end{equation}
for paths $\alpha_i$ on $\mathcal{G}$ of length $j$, where $(\phi_v)$ denotes the Perron-Frobenius eigenvector of $\mathcal{G}$.
We define another trace $\text{tr}_{A_j}$ by $\mathrm{tr}_{A_j}((\alpha_1,\alpha_2)) = [3]_q^j \mathrm{tr}_A((\alpha_1,\alpha_2))$ for $|\alpha_i|=j$.
Then $\text{tr}_{A_j}(\bf{1}_j)$ is the rank of the identity $\bf{1}_j$ of $A_j$, i.e. the number of paths of length $j$ on $\mathcal{G}$.

The algebra $A_j$ decomposes as $A_j = \bigoplus_v e_v A_j$, where the sum is over all vertices $v$ of $\mathcal{G}$ and $e_v$ denotes a basis element for the one-dimensional space of paths of length 0 at $v$, as in Section \ref{sect:cells-intro}.
Suppose $\mathfrak{D}$ is a morphism in $\mathcal{C}^m$ which is a diagram with $j$ vertices along the top and bottom, and let $D:= F(\mathfrak{D}) \in A_j$.
If we fix any vertex $v$ of $\mathcal{G}$, the action on $e_v D$ of the morphism given by attaching strings joining the vertices along the top with the those along the bottom is $\text{tr}_{A_j}(e_v D)$ by \eqref{eqn:canonical-E}, and we have $\text{tr}_{A_j}(e_v D) = \text{tr}_T(\mathfrak{D})$.

It was shown in Section \ref{sect:classification-W} that any module category $\mathcal{M}$ with nimrep graph $\mathcal{A}_{2m}$, $\mathcal{E}_8$ or $\mathcal{E}_{14}^c$ is equivalent to a canonical module category with a real trivalent cell system in which $W(\tilde{c}\tilde{b}\tilde{a})=W(abc)$, and similarly and $\ast$-module category with nimrep graph $\sigma_{2m}$, $\mathcal{E}_8^c$ or $\mathcal{E}_{14}$ has a real trivalent cell system. Therefore we may choose bases of the morphism spaces such that $F(\mathfrak{f}_{2m})$ is a real projection, where $2m$ is the $SO(3)$ level.
Then by \eqref{eqn:tr(fj)}, $\text{tr}_{A_j}(e_v F(\mathfrak{f}_{2m}))=\text{tr}_T(\mathfrak{f}_{2m})=1$, thus $e_v F(\mathfrak{f}_{2m})$ is a rank-one projection for each vertex $v$ of $\mathcal{G}$.
Thus, from \eqref{eqn:diagrammatic_relations-T}, $e_v F(\mathfrak{t})$ exists and is uniquely determined, up to a sign $\varepsilon_v \in \{ \pm 1 \}$.
Fix a vertex $x$ of $\mathcal{G}$. Since $\mathfrak{t}$ is invariant (up to a change of sign) under rotation, the phases $\varepsilon_v$ are all determined by the choice of $\varepsilon := \varepsilon_x$.

Thus $F(\mathfrak{t})$ is uniquely determined, up to sign, by the trivalent cell system $W$. Changing sign is equivalent to interchanging the actions of $Q_{\pm} \in \mathcal{C}$ on $\mathcal{M}$. As noted at the end of Section \ref{sect:module_cat}, any two such module categories are equivalent.
We therefore obtain the following classification of $\ast$-module categories for $SO(3)_{2m}$, which is the main result of the paper:

\begin{Thm} \label{Thm:classification}
The $\ast$-module categories for the $SO(3)_{2m}$ modular tensor category $\mathcal{C}^m$ are classified by the $SO(3)_{2m}$ nimrep graphs, illustrated in Figure \ref{fig-SO3_nimrep_graphs}. In particular,
\begin{itemize}
\item for $m=4$ there are six inequivalent $\ast$-module categories: 
there are unique module categories given by $\mathcal{A}_{8}$ and $\sigma_{8}$, and two inequivalent module categories for each of $\mathcal{E}_{8}$ and $\mathcal{E}_{8}^c$;
\item for $m=7$ there are four inequivalent $\ast$-module categories given by $\mathcal{A}_{14}$, $\sigma_{14}$, $\mathcal{E}_{14}$ and $\mathcal{E}_{14}^c$;
\item for all other $m$ there are two inequivalent $\ast$-module categories given by $\mathcal{A}_{2m}$ and $\sigma_{2m}$.
\end{itemize}
\end{Thm}

We conjecture that Theorem \ref{Thm:classification} is in fact a classification of module categories for $\mathcal{C}^m$. For any such module category the possible nimrep graphs are those given in Theorem \ref{Thm:classification}, and there is a unique $\ast$-module category associated to that nimrep. As noted in Remark \ref{Rem:conjecture-unique-sigma2m} in the case of the nimrep graph $\sigma_{2m}$, the numerical evidence suggests that all module categories associated with a given nimrep are indeed equivalent, but we have not been able to show this explicitly.

\section{$SO(3)$-Temperley-Lieb algebra and $\mathcal{A}_{2m}$ subfactor} \label{sect:realisation_MI_subfactors}

In Section \ref{sect:SO3-TL} we discuss an $SO(3)$-analogue of the Temperley-Lieb algebra. In Section \ref{sect:string_alg} we define the string algebras for an $SO(3)$ graph $\mathcal{G}$, that is, inclusions of finite dimensional algebras whose Bratteli diagram is given by $\mathcal{G}$, and realise the $SO(3)$-Temperley-Lieb elements inside this algebra using the cell systems classified in Section \ref{sect:cell-systems}. In Section \ref{sect:A2m-subfactor} we present a construction of a subfactor built from the $SO(3)$-Temperley-Lieb algebra, with principal graph coming from $\mathcal{A}_{2m}$, recovering a subfactor presented in \cite{izumi:1991}. 
In Section \ref{sect:GHJ} we comment on the application of the results of this paper to the realisation of the $SO(3)$ modular invariants by braided subfactors via an analogue of the Goodman-de la Harpe-Jones subfactors.

\subsection{$SO(3)$-Temperley-Lieb algebra} \label{sect:SO3-TL}

For generic $q$, Lehrer and Zhang \cite{lehrer/zhang:2010, lehrer/zhang:2008} gave a presentation of $\text{End}_{\mathcal{U}_q(\mathfrak{sl}_2)}(V_q^{\otimes r})$, where $\mathcal{U}_q(\mathfrak{sl}_2)$ is the quantised universal enveloping algebra of $\mathfrak{sl}_2$ over the function field $\mathbb{C}(q^{1/2})$, and $V_q$ is the irreducible 3-dimensional representation of quantum $\mathfrak{sl}_2$, i.e. the fundamental representation of $\mathfrak{so}_3$.
This algebra, the $SO(3)$-Temperley-Lieb algebra, is a quotient of the BMW algebra \cite{birman/wenzl:1989, murakami:1987} whose parameters are given in terms of $q$. This algebra $BMW(q)$ has generators $g_i^{\pm1}$, $e_i$, $i \in \bbN$, where the $g_i^{\pm1}$ satisfy the braid relations
\begin{equation} \label{eqn:braid-relations}
g_i g_{i+1} g_i = g_{i+1} g_i g_{i+1}, \qquad\qquad g_i g_j = g_j g_i, \quad (|i-j|>1),
\end{equation}
whilst the $e_i$ satisfy the Jones relations
\begin{equation}
e_i^2 = [3] e_i, \qquad\qquad e_i e_{i \pm 1} e_i = e_i, \qquad\qquad e_i e_j = e_j e_i, \quad (|i-j|>1).
\end{equation}
We also have mixed relations
\begin{equation} \label{eqn:BMW-relations-3}
g_i-g_i^{-1} = (q^2-q^{-2})(1-e_i), \qquad\qquad g_i e_i = e_i g_i = q^{-4} e_i,
\end{equation}
\begin{equation} \label{eqn:BMW-relations-4}
e_i g_{i-1}^{\pm1} e_i = q^{\pm4} e_i, \qquad\qquad e_i g_{i+1}^{\pm1} e_i = q^{\pm4} e_i.
\end{equation}
The following relations can also be derived from the above relations (\ref{eqn:braid-relations})-(\ref{eqn:BMW-relations-4}):
\begin{equation}
(g_i - q^2)(g_i+q^{-2}) = -q^{-4} (q^2-q^{-2}) e_i, \qquad\qquad g_i e_{i\pm1} e_i = g_{i\pm1}^{-1} e_i.
\end{equation}

Let $\mathcal{W}$ be the $SO(3)$ algebra for generic $q$, as in Remark \ref{SO3_q-generic_q}. As in \cite[Theorem 5.1]{fendley/krushkal:2010}, we can define a homomorphism of algebras $\eta: BMW(q) \to \mathcal{W}$ by $\eta(e_i) = E_i$, $\eta(g_i) = q^2+(q^{-2}-1)E_i-(q^2+q^{-2})U_i$, where $E_i$ is given by \eqref{def:E_i} and $U_i$ is given by
\begin{equation} \label{def:U_i}
U_i = \;\; \raisebox{-.45\height}{\includegraphics[width=40mm]{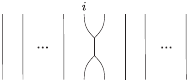}}
\end{equation}
Contrary to the claim in \cite[Remark 5.4]{fendley/krushkal:2010}, this homomorphism is in fact surjective:

\begin{Prop}
The homomorphism $\eta: BMW(q) \to \mathcal{W}$ defined above is surjective.
\end{Prop}

\begin{proof}
Due to the relations \eqref{eqn:SO3-relation1}, \eqref{eqn:SO3-relation4}, it is sufficient to only consider diagrams in $\mathcal{W}$ without elliptic faces -- diagrams which do not have any regions which are bounded on all sides by strings, that is, the intersection of any region in the diagram with the outer boundary of the tangle is non-empty. We construct an algorithm to write any such diagram in $\mathcal{W}$ as an element in $\text{alg}(1,E_i,U_i \, | \, i \in \bbN)$. Let $D$ be a diagram with $n$ vertices along the top and bottom, and without elliptic faces. Note that since the number $2n$ of vertices on the boundary of any diagram is even, the number of triple points in the diagram must also be even.

\textbf{Step 1}:
If $D$ does not contain any triple points then it is a Temperley-Lieb diagram, i.e. is a product of $e_i$.

\textbf{Step 2}:
If there is a string in the diagram $D$ which has one endpoint attached to a vertex $v$ along the top of the diagram and the other endpoint attached to a vertex adjacent to $v$, then isotope the diagram to pull this `cup' out above the rest of the diagram, as illustrated in Figure \ref{fig-chromatic_isom-1}.
\begin{figure}[htb]
\begin{center}
\includegraphics[width=20mm]{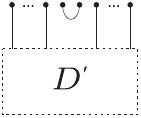} \\
\caption{Pulling a cup out of a diagram $D$} \label{fig-chromatic_isom-1}
\end{center}
\end{figure}

Suppose first that there is similarly a string whose endpoints are adjacent vertices along the bottom edge of the diagram. If this string is the boundary of a region which shares a common boundary with the top boundary of the diagram, then isotoping the strings we can pull this `cap' out above the rest of the diagram, as illustrated in the diagram on the right of Figure \ref{fig-chromatic_isom-2}, so that our diagram $D=XD_1$ where $X$ is a Temperley-Lieb diagram. Otherwise, by isotoping the strings as in the diagram on the left of Figure \ref{fig-chromatic_isom-2}, we can move the cap so that it can be pulled up as in the previous case, and our diagram is given by $D=X_1D_1X_2$ where $X_1$, $X_2$ are both Temperley-Lieb diagrams.
\begin{figure}[htb]
\begin{center}
\includegraphics[width=140mm]{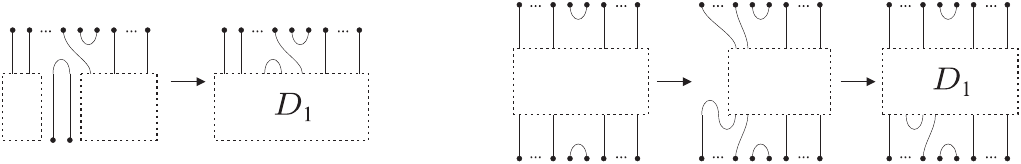} \\
\caption{Pulling a cap out of a diagram $D$} \label{fig-chromatic_isom-2}
\end{center}
\end{figure}

If there are no caps along the bottom edge, then since the number of vertices along the bottom boundary of the sub-diagram $D'$ is 2 greater than the number along the top edge, and since $D'$ contains no elliptic faces, there must be a triple point in $D'$ where exactly one string is attached to one of the vertices along the top boundary of $D'$ but the other two strings are attached either to vertices along the bottom boundary or to other triple points. Choose one of the latter two strings, and isotope the string to pull a cap out above the rest of the diagram, as illustrated in Figure \ref{fig-chromatic_isom-3}. We again have $D=XD_1$ where $X$ is a Temperley-Lieb diagram.
Repeat Step 2 for the diagram $D_1$, until there are no strings whose endpoints are adjacent vertices along the top edge of the diagram.
\begin{figure}[htb]
\begin{center}
\includegraphics[width=85mm]{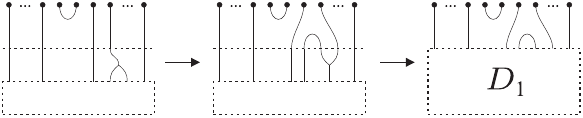} \\
\caption{Isotoping a triple point to pull a cap out of a diagram $D$} \label{fig-chromatic_isom-3}
\end{center}
\end{figure}

\textbf{Step 3}:
Since there are no strings attached to a vertex at the top of the diagram whose other endpoint is an adjacent vertex along the top edge of the diagram, for at least one string attached to a vertex at the top of the diagram its other endpoint will be attached to a triple point.

\textbf{Case 1}:
Assume first that such a triple point exists for which at most 2 of its strands are attached to a vertex along the top edge of the diagram. We isotope the strings to pull this triple point out above the rest of the diagram, as we did in Step 2 for a cup. By the same argument, there will be another string attached either to a vertex at the top of the diagram or to this triple point whose other endpoint will be attached to a triple point. From a combinatorial argument there one such triple point must have a second string which is attached to one of the vertices at the top of the diagram. We choose this triple point and pull it out above the rest of the diagram, and we have essentially one of the two situations illustrated in Figure \ref{fig-chromatic_isom-4}.
\begin{figure}[htb]
\begin{center}
\includegraphics[width=55mm]{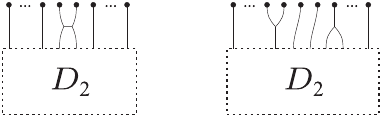} \\
\caption{Pulling two triple points out of a diagram $D$} \label{fig-chromatic_isom-4}
\end{center}
\end{figure}

In the case of the diagram on the left we have $D_1=XD_2$ where $X$ is $(U_i-\alpha(E_i-1))$ for some $i \in \{1,\ldots,n-1\}$, where $\alpha = [2]_q/[4]_q$.
We now consider the diagram on the right. Suppose that there are $m <n-2$ vertical strands separating the two triple points. We will show by induction on $m$ that this diagram can be written as a linear combination of diagrams of the form $X_1D_2X_2$ where $X_1,X_2 \in \text{alg}(1,E_i,U_i \, | \, i =1,\ldots,n-1)$ and $D_2$ has $m-1$ vertical strands separating the two triple points. Using relation \eqref{eqn:SO3-relation4} twice we have, with $\alpha$ as above,
\begin{align*}
\alpha \;\;& \raisebox{-.45\height}{\includegraphics[width=18mm]{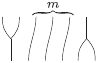}} \;\; = \;\; \raisebox{-.45\height}{\includegraphics[width=18mm]{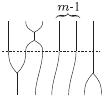}} \;\; - \;\; \raisebox{-.45\height}{\includegraphics[width=18mm]{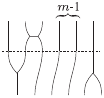}} \;\; + \alpha \;\; \raisebox{-.45\height}{\includegraphics[width=18mm]{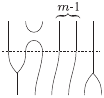}} \\
&\qquad = \;\; \raisebox{-.45\height}{\includegraphics[width=18mm]{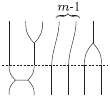}} \;\; - \;\; \raisebox{-.45\height}{\includegraphics[width=18mm]{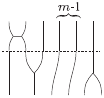}} \;\; + \left( \; \raisebox{-.45\height}{\includegraphics[width=18mm]{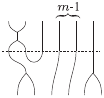}} \;\; - \;\; \raisebox{-.45\height}{\includegraphics[width=18mm]{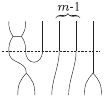}} \;\; + \alpha \;\; \raisebox{-.45\height}{\includegraphics[width=18mm]{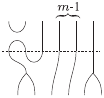}} \; \right) \\
&\qquad = \;\; \raisebox{-.45\height}{\includegraphics[width=18mm]{fig-chromatic_isom-5-5}} \;\; - \;\; \raisebox{-.45\height}{\includegraphics[width=18mm]{fig-chromatic_isom-5-6}} \;\; + \;\; \raisebox{-.45\height}{\includegraphics[width=18mm]{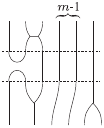}} \;\; - \;\; \raisebox{-.45\height}{\includegraphics[width=18mm]{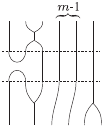}} \;\; + \alpha \;\; \raisebox{-.45\height}{\includegraphics[width=18mm]{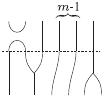}}
\end{align*}
For the base step, $m=0$, we have, again using relation \eqref{eqn:SO3-relation4} twice:
\begin{align*}
\alpha \;\; \raisebox{-.45\height}{\includegraphics[width=8mm]{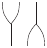}} \;\; &= \alpha \;\; \raisebox{-.45\height}{\includegraphics[width=10mm]{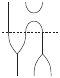}} \;\; - \;\; \raisebox{-.45\height}{\includegraphics[width=10mm]{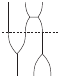}} \;\; + \;\; \raisebox{-.45\height}{\includegraphics[width=10mm]{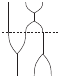}} \;\; = \alpha \;\; \raisebox{-.45\height}{\includegraphics[width=10mm]{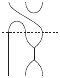}} \;\; - \;\; \raisebox{-.45\height}{\includegraphics[width=10mm]{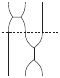}} \;\; + \;\; \raisebox{-.45\height}{\includegraphics[width=8mm]{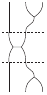}} \\
&= \alpha \;\; \raisebox{-.45\height}{\includegraphics[width=10mm]{fig-chromatic_isom-6-5}} \;\; - \;\; \raisebox{-.45\height}{\includegraphics[width=10mm]{fig-chromatic_isom-6-6}} \;\; + \left( \; \raisebox{-.45\height}{\includegraphics[width=8mm]{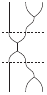}} \;\; - \alpha \;\; \raisebox{-.45\height}{\includegraphics[width=8mm]{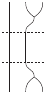}} \;\; + \alpha \;\; \raisebox{-.45\height}{\includegraphics[width=8mm]{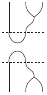}} \; \right) \\
&= \alpha \;\; \raisebox{-.45\height}{\includegraphics[width=10mm]{fig-chromatic_isom-6-5}} \;\; - \;\; \raisebox{-.45\height}{\includegraphics[width=10mm]{fig-chromatic_isom-6-6}} \;\; + \;\; \raisebox{-.45\height}{\includegraphics[width=8mm]{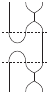}} \;\; - \alpha \;\; \raisebox{-.45\height}{\includegraphics[width=8mm]{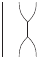}} \;\; + \alpha \;\; \raisebox{-.45\height}{\includegraphics[width=8mm]{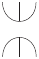}}
\end{align*}
So it remains to show that the last term, call it $F_1$, in the last line above is an element of $\text{alg}(1,E_i,U_i \, | \, i =1,2)$, which follows from using the relation \eqref{eqn:SO3-relation6}:
$$\raisebox{-.45\height}{\includegraphics[width=10mm]{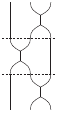}} \;\; = \frac{q^2-2+q^{-2}}{q^2+q^{-2}} \;\; \raisebox{-.45\height}{\includegraphics[width=10mm]{fig-chromatic_isom-6-11}} \;\; + \frac{q^2-1+q^{-2}}{(q^2+q^{-2})^2} \;\; \raisebox{-.45\height}{\includegraphics[width=8mm]{fig-chromatic_isom-6-13}} \;\; + \frac{1}{(q^2+q^{-2})^2} \;\; \raisebox{-.45\height}{\includegraphics[width=8mm]{fig-chromatic_isom-6-12}}$$

\textbf{Case 2}:
Now consider the case where if a triple point which has a strand attached to a vertex along the top edge of the diagram then it has all 3 of its strands attached to (necessarily adjacent) vertices along the top edge of the diagram. Suppose that there are $m <n-2$ vertical strands separating the two triple points. Then using relation \eqref{eqn:SO3-relation4} we have
\begin{align*}
\alpha \;\; \raisebox{-.45\height}{\includegraphics[width=24mm]{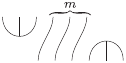}} \;\; &= \;\; \raisebox{-.45\height}{\includegraphics[width=26mm]{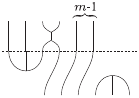}} \;\; - \;\; \raisebox{-.45\height}{\includegraphics[width=26mm]{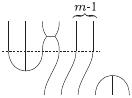}} \;\; + \alpha \;\; \raisebox{-.45\height}{\includegraphics[width=26mm]{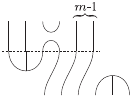}} \\
&= \;\; \raisebox{-.45\height}{\includegraphics[width=26mm]{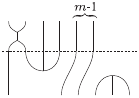}} \;\; - \;\; \raisebox{-.45\height}{\includegraphics[width=21mm]{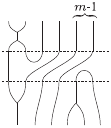}} \;\; + \alpha \;\; \raisebox{-.45\height}{\includegraphics[width=21mm]{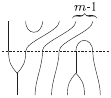}}
\end{align*}
and by induction we see that such a diagram is an element of $\text{alg}(1,E_i,U_i \, | \, i =1,\ldots,n-1)$.
\end{proof}

It was conjectured in \cite[Remark 5.4]{fendley/krushkal:2010} that the map $\eta$ is injective, however by \cite[Theorem 1.6]{lehrer/zhang:2008} it has a non-trivial kernel
\begin{align} \label{def:Phi_q}
\Phi_q &= aF_qE_2F_q-bF_q-cF_qE_2E_1E_3E_2F_q \nonumber \\
&\quad +dF_q \big( q^2E_2+(q^{-2}-1)E_2E_1E_3E_2-(q^2+q^{-2})E_2U_1E_3E_2 \big) F_q,
\end{align}
where $F_q=U_1U_3$ and the coefficients $a,b,c,d$ are given by:
$$a=1+(1-q^{-2})^2, \quad b=a+(1-q^2)^2, \quad c= \frac{q^2-q^{-2}+2q^{-4}-2q^{-6}+q^{-8}}{(q^2+q^{-2})^2}, \quad d=(q-q^{-1})^2.$$
It is a tedious computation to verify using \eqref{def:E_i}, \eqref{def:U_i} and the $SO(3)$ diagrammatic relations that $\Phi_q$ does indeed evaluate as zero in $\mathcal{W}$.

The generators $E_i$, $U_i$ of $\mathcal{W}$ satisfy the following relations, which are easily obtained from (\ref{eqn:braid-relations})-(\ref{eqn:BMW-relations-4}):
\begin{equation}
E_i^2 = [3] E_i, \qquad\qquad U_i^2 = U_i, \qquad\qquad U_iE_i = 0 = E_i U_i, \label{eqn:SO3_Hecke-relations-1}
\end{equation}
\begin{equation}
E_i E_j = E_j E_i, \qquad\qquad U_i U_j = U_j U_i, \qquad\qquad U_i E_j = E_j U_i, \quad (|i-j|>1)
\end{equation}
\begin{equation}
E_i E_{i \pm 1} E_i = E_i, \qquad\qquad E_i U_{i \pm 1} E_i = E_i,
\end{equation}
\begin{equation}
\frac{[3][4]}{[6]} U_i U_{i \pm 1} E_i = U_i E_{i \pm 1} E_i = U_{i+1} E_i + \frac{[2]}{[4]} \left( E_{i+1} E_i - E_i \right),
\end{equation}
\begin{equation}
U_i E_{i+1} U_i - \frac{[2]}{[4]}(U_i E_{i+1} + E_{i+1} U_i) + \frac{[2]^2}{[4]^2} E_{i+1} = U_{i+1} E_i U_{i+1} - \frac{[2]}{[4]}(U_{i+1} E_i + E_i U_{i+1}) + \frac{[2]^2}{[4]^2} E_i,
\end{equation}
\begin{equation}
U_i U_{i+1} U_i - \frac{[6]-[2][3]}{[3][4]} U_i E_{i+1} U_i - \frac{[2]^2}{[4]^2} U_i = U_{i+1} U_i U_{i+1} - \frac{[6]-[2][3]}{[3][4]} U_{i+1} E_i U_{i+1} - \frac{[2]^2}{[4]^2} U_{i+1}, \label{eqn:SO3_Hecke-relations-5}
\end{equation}
Note that here $(q^2+q^{-2})U_i$ is the element $f_i = -g_i - (1-q^{-2})e_i + q^2$ of \cite{lehrer/zhang:2010, lehrer/zhang:2008} and $E_i$ is $e_i$.
The L.H.S. (or R.H.S.) of equation (\ref{eqn:SO3_Hecke-relations-5}) is equal to $([2]_q[6]_q/[3]_q[4]_q^2)F_i$, where $F_i$ is the diagram illustrated in Figure \ref{fig-F_i}.

\begin{figure}
\begin{center}
\includegraphics[width=40mm]{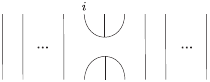} \\
\caption{$SO(3)$-TL element $F_i$} \label{fig-F_i}
\end{center}
\end{figure}

For generic $q$, the $SO(3)$-Temperley-Lieb algebra is defined \cite[Theorem 1.6]{lehrer/zhang:2008} as the quotient of the above algebra by the ideal generated by the relation $\Phi_q=0$, where $\Phi_q$ is defined in \eqref{def:Phi_q}.

For $q$ a primitive $4(2m+1)$-th root of unity, the $SO(3)_{2m}$-Temperley-Lieb algebra has generators $E_i$, $U_i$, $i \in \bbN$, with relations (\ref{eqn:SO3_Hecke-relations-1})-(\ref{eqn:SO3_Hecke-relations-5}), along with an additional generator $t$ with relation
\begin{align} \label{eqn:SO3_Hecke-relations-S}
t(E_m)&(E_{m-1}E_{m+1})(E_{m-2}E_mE_{m+2})\cdots(E_1E_3\cdots E_{2m-1})(E_2E_4\cdots E_{2m-2})\cdots(E_m)t \qquad \nonumber \\
&= [2m+1]_q f_{2m},
\end{align}
where $f_j$, $j=0,1,\ldots$ are the $SO(3)$ Jones-Wenzl projections. This relation is the algebraic version of the diagrammatic relation (\ref{eqn:diagrammatic_relations-T}).

\subsection{String Algebra construction} \label{sect:string_alg}

From now on we will restrict to the case where $\mathcal{M}$ is a $\ast$-module category for $\mathcal{C}^m$. Let $\mathcal{G}$ be the nimrep graph associated to $\mathcal{M}$, i.e. one of the $SO(3)_{2m}$ graphs illustrated in Figure \ref{fig-SO3_nimrep_graphs}, $W$ the cell-system for $\mathcal{M}$ and $F$ the functor defining the module structure of $\mathcal{M}$ as in Section \ref{sect:module_cat}. Let $A_j$ denote finite dimensional algebras in the path algebra of $\mathcal{G}$ as in Section \ref{sect:cell-systems}, and let $\mathfrak{E}_j$, $\mathfrak{U}_j$ denote the morphisms in $\mathcal{C}^m$ given by the diagrams $E_j$, $U_j$ respectively. We define $SO(3)_{2m}$-Temperley-Lieb operators $e_j := F(\mathfrak{E}_j), u_j := F(\mathfrak{U}_j) \in A_{j+1}$, so that $e_j$, $u_j$ are given by
\begin{align}
e_j &= \sum_{\sigma,\beta,\gamma} \frac{\sqrt{\phi_{r(\beta)}\phi_{r(\gamma)}}}{\phi_{s(\beta)}} \, (\sigma \cdot \beta \cdot \tilde{\beta}, \sigma \cdot \gamma \cdot \tilde{\gamma}), \label{eqn:E_j} \\
u_j &= \sum_{\sigma,\eta_i,\zeta_i,\lambda} \frac{1}{\phi_{s(\eta_1)}\phi_{r(\eta_2)}} \, W(\eta_1,\eta_2,\lambda) \overline{W(\zeta_1,\zeta_2,\lambda)} \, (\sigma \cdot \eta_1 \cdot \eta_2, \sigma \cdot \zeta_1 \cdot \zeta_2),  \label{eqn:U_j}
\end{align}
where $|\sigma|=j-1$, $|\beta|=|\gamma|=|\eta_i|=|\zeta_i|=1$.
It is a straightforward but somewhat tedious computation to show that these elements satisfy all the $SO(3)_{2m}$-Temperley-Lieb relations (\ref{eqn:SO3_Hecke-relations-1})-(\ref{eqn:SO3_Hecke-relations-5}).
The additional generator $F(\mathfrak{t}) \in A_m$ is the unique operator (up to a choice of sign) such that \eqref{eqn:SO3_Hecke-relations-S} is satisfied (see Section \ref{sect:T-cell-system}).
Since $\mathcal{M}$ is a $\ast$-module category, these operators are self-adjoint, where $(\sigma_1,\sigma_2)^{\ast} = (\sigma_2,\sigma_1)$, extended conjugate linearly to linear combinations.
Let $\text{tr}_A$ be the trace on $\bigcup_j A_j$ defined in \eqref{eqn:trace}.

\begin{Lemma} \label{Lemma-Markov_trace-SO(3)}
For $\mathcal{G}$ an $SO(3)_{2m}$ graph, $\mathrm{tr}_A$ is a Markov trace in the sense that $\mathrm{tr}_A(x e_j) = [3]^{-1} \mathrm{tr}(x)$ for any $x \in A_j$, $j \geq 1$.
\end{Lemma}

\begin{proof}
Let $x \in A_j$ be the matrix unit $(\alpha_1 \cdot \alpha_1', \alpha_2 \cdot \alpha_2')$, where $|\alpha_i|=j-1$, $|\alpha_i'|=1$ ($i=1,2$). Then
\begin{align*}
x e_j &= \sum_{\sigma, \beta, \gamma, \mu} \frac{\sqrt{\phi_{r(\beta)}\phi_{r(\gamma)}}}{\phi_{s(\beta)}} \, (\alpha_1 \cdot \alpha_1' \cdot \mu, \alpha_2 \cdot \alpha_2' \cdot \mu) \cdot (\sigma \cdot \beta \cdot \tilde{\beta}, \sigma \cdot \gamma \cdot \tilde{\gamma}) \\
&= \sum_{\sigma, \beta, \gamma, \mu} \frac{\sqrt{\phi_{r(\beta)}\phi_{r(\gamma)}}}{\phi_{s(\beta)}} \, \delta_{\alpha_2,\sigma} \delta_{\alpha_2',\beta} \delta_{\mu, \tilde{\beta}} (\alpha_1 \cdot \alpha_1' \cdot \mu, \sigma \cdot \gamma \cdot \tilde{\gamma}) \\
&= \sum_{\gamma} \frac{\sqrt{\phi_{r(\alpha_2')}\phi_{r(\gamma)}}}{\phi_{s(\alpha_2')}} \, (\alpha_1 \cdot \alpha_1' \cdot \tilde{\alpha_2'}, \alpha_2 \cdot \gamma \cdot \tilde{\gamma}),
\end{align*}
and
\begin{align*}
\mathrm{tr}_A(x e_j) &= \sum_{\gamma} \frac{\sqrt{\phi_{r(\alpha_2')}\phi_{r(\gamma)}}}{\phi_{s(\alpha_2')}} \, \mathrm{tr}_A((\alpha_1 \cdot \alpha_1' \cdot \tilde{\alpha_2'}, \alpha_2 \cdot \gamma \cdot \tilde{\gamma})) \\
&= \sum_{\gamma} \frac{\sqrt{\phi_{r(\alpha_2')}\phi_{r(\gamma)}}}{\phi_{s(\alpha_2')}} \, \delta_{\alpha_1, \alpha_2} \delta_{\alpha_1', \gamma} \delta_{\alpha_2', \gamma} [3]^{-j-1} \phi_{s(\alpha_2')} \phi_{s(\alpha_1)}^{-1} \\
&= \delta_{\alpha_1, \alpha_2} \delta_{\alpha_1', \alpha_2'} [3]^{-j-1} \phi_{r(\alpha_2')} \phi_{s(\alpha_1)}^{-1} = [3]^{-1} \mathrm{tr}_A(x).
\end{align*}
The result for arbitrary $x \in A_j$ follows by linearity of the trace.
\end{proof}

Before we present the construction of the $SO(3)$-Goodman-de la Harpe-Jones subfactor, we need a few results. Denote by $d_{\mathcal{G}}$ the depth of the graph $\mathcal{G}$, that is, $d_{\mathcal{G}} = \max_{v,v'}(d_{v,v'})$, where $d_{v,v'}$ is the length of the shortest path between vertices $v,v'$ of $\mathcal{G}$.

\begin{Lemma} \label{Lemma:aE_jb}
For any $j>d_{\mathcal{G}}$, any element in the finite dimensional algebra $A_{j+1}$ is a linear combination of elements of the form $ae_jb$ where $a,b \in A_j$.
\end{Lemma}

\begin{proof}
Let $a=(a_1 \cdot a_1', a_2 \cdot a_2')$, $b=(b_1 \cdot b_1', b_2 \cdot b_2')$ where $|a_i|=|b_i|=j-1$, $|a_i'|=|b_i'|=1$. Embedding $a$, $b$ in $A_{j+1}$ we have $a=\sum_{\tau_a} (a_1 \cdot a_1' \cdot \tau_a, a_2 \cdot a_2' \cdot \tau_a)$, $b=\sum_{\tau_b} (b_1 \cdot b_1' \cdot \tau_b, b_2 \cdot b_2' \cdot \tau_b)$, where $|\tau_i|=1$. Then
\begin{align}
a e_j b &= \sum_{\sigma,\beta,\gamma,\tau_i} \frac{\sqrt{\phi_{r(\beta)}\phi_{r(\gamma)}}}{\phi_{s(\beta)}} \, (a_1 \cdot a_1' \cdot \tau_a, a_2 \cdot a_2' \cdot \tau_a) \cdot (\sigma \cdot \beta \cdot \tilde{\beta}, \sigma \cdot \gamma \cdot \tilde{\gamma}) \cdot (b_1 \cdot b_1' \cdot \tau_b, b_2 \cdot b_2' \cdot \tau_b) \nonumber \\
&= \delta_{a_2,b_1} \, \frac{\sqrt{\phi_{r(a_2')}\phi_{r(b_1')}}}{\phi_{s(a_2')}} \, (a_1 \cdot a_1' \cdot \tilde{a_2'}, b_2 \cdot b_2' \cdot \tilde{b_1'}). \label{eqn:aE_ib}
\end{align}
Since $\phi_v \neq 0$ for all vertices $v$, providing $a_2$, $b_1$ are paths whose length is at least the depth $d_{\mathcal{G}}$ of the graph $\mathcal{G}$, i.e. $j > d_{\mathcal{G}}$, any element in $A_{j+1}$ will be a linear combination of elements of the form (\ref{eqn:aE_ib}).
\end{proof}

\subsection{$\mathcal{A}_{2m}$ subfactor} \label{sect:A2m-subfactor}

Let $M_0 \subset M_1 \subset M_2 \subset \cdots$ be the path algebra of $\mathcal{A}_{2m}$ where we now restrict to pairs of paths which begin only at the distinguished vertex $\ast_{\mathcal{A}}$ of $\mathcal{A}_{2m}$ (labelled 1 in Figure \ref{fig-SO3_nimrep_graphs}), so that $M_0 = \mathbb{C}$ and the Bratteli diagram for the inclusion $M_{l-1} \subset M_l$ is given by the bipartite unfolding of the graph $\mathcal{A}_{2m}$.

\begin{Lemma} \label{Lemma:algAk=algEU}
For the graph $\mathcal{A}_{2m}$, the finite dimensional algebra $M_{l+1}$ is generated by
$$\begin{cases}
1,e_j,u_j, \quad j=1,\ldots,l, & \text{ for } l < m-1, \\
1,e_j,u_j,F(\mathfrak{t}), \quad j=1,\ldots,l, & \text{ for } l \geq m-1
\end{cases}$$
\end{Lemma}

\begin{proof}
The statement is trivially true for $l=1$ since $M_1 \cong \bbC$. For $l=2$ we have $e_1=[3]([121],[121])$, $u_1=([122],[122])$ and $F(\mathfrak{f}_2) = 1-[3]^{-1}e_1-u_1 = ([123],[123])$.
We will show by induction that any element of $M_{l+1}$, $2 \leq l < m$, is a linear combination of elements of the form $ae_lb$, $au_lb$ and $a$, where $a,b \in M_l$. We will denote the span of all such elements by $B_l$.
Let $a,b \in M_l$ be as in the proof of Lemma \ref{Lemma:aE_jb}. From the expression for $ae_j b$ given in equation (\ref{eqn:aE_ib}) we see that for $l \leq m$ ($=d_{\mathcal{A}_{2m}}$), any matrix units $(\xi_1,\xi_2)$ in $M_{l+1}$ for which $r(\xi_i) \leq l$ are given by an element of the form $ae_lb$ for suitable $a,b \in M_l$. It remains to show that matrix units $(\xi_1,\xi_2)$ belong to $B_l$ when
$$r(\xi_i)=\begin{cases}
l+1, l+2, & l \leq m-2, \\
m, p_{\pm}, & l=m-1, \\
p_{\pm}, & l=m.
\end{cases}$$
Note that since $|\xi_i|=l+1$, for the final edge $\eta_i$ of $\xi_i$ we have $s(\eta_i) \leq l+1$ for $l \leq m-1$ (for $l \geq m =d_{\mathcal{A}_{2m}}$ $s(\eta_i)$ can be any vertex of $\mathcal{A}_{2m}$). Thus, for example, for $l \leq m-2$, the paths $\xi_i$ that we need to consider are those whose final edge $\eta_i$ satisfies $(s(\eta_i),r(\eta_i)) \in \{ (l,l+1), (l+1,l+1), (l+1,l+2) \}$.
Now
\begin{align}
a u_l b &= \sum_{\sigma,\zeta_i,\eta_i,\lambda,\tau_i} \frac{W(\zeta_1,\zeta_2,\lambda)\overline{W(\eta_1,\eta_2,\lambda)}}{\phi_{s(\zeta_1)}\phi_{r(\zeta_2)}} \, (a_1 \cdot a_1' \cdot \tau_a, a_2 \cdot a_2' \cdot \tau_a) \cdot (\sigma \cdot \zeta_1 \cdot \zeta_1, \sigma \cdot \eta_1 \cdot \eta_2) \nonumber \\
&\hspace{90mm} \cdot (b_1 \cdot b_1' \cdot \tau_b, b_2 \cdot b_2' \cdot \tau_b) \nonumber \\
&= \delta_{a_2,b_1} \sum_{\zeta,\eta} \frac{W(a_2',\zeta,\lambda)\overline{W(b_1',\eta,\lambda)}}{\phi_{s(a_2')}\phi_{r(\zeta)}} \, (a_1 \cdot a_1' \cdot \zeta, b_2 \cdot b_2' \cdot \eta), \label{eqn:aU_ib}
\end{align}
where $\lambda$ is uniquely determined by $\zeta$, $\eta$ in each summand.

For $l \leq m-2$, by an appropriate choice of $a,b \in M_l$ where $a_2=b_1$ and $s(a_2')=s(b_1')=m$ we have that (\ref{eqn:aU_ib}) yields
$$\Omega_l (a_1 \cdot a_1' \cdot \zeta_l, b_2 \cdot b_2' \cdot \eta_l) + \Omega_{l+1} (a_1 \cdot a_1' \cdot \zeta_{l+1}, b_2 \cdot b_2' \cdot \eta_{l+1})$$
where $r(\zeta_j)=r(\eta_j)=j$, $\Omega_j = \phi_{s(a_2')}^{-1}\phi_{r(\zeta_j)}^{-1}W(a_2',\zeta_j,\lambda_j)\overline{W(b_1',\eta_j,\lambda_j)}$ and $\lambda_j$ is the unique edge with $s(\lambda_j)=j$, $r(\lambda_j)=s(a_2')$. The first term above is in $B_l$, thus we can obtain that any matrix unit $(\xi_1,\xi_2) \in B_l$ where $r(\xi_i)=l+1$. For the final matrix unit $(\nu,\nu)$ where $r(\nu)=l+2$, this is easily obtained by taking the identity element in $M_{l+1}$, and subtracting all the matrix units $(\xi,\xi)$ where $r(\xi) \leq l+1$, which are all elements in $B_l$. Since $(\nu,\nu)$ is the projection orthogonal to all these other matrix units, it is in fact given by the $SO(3)_{2m}$ Jones-Wenzl projection $F(\mathfrak{f}_{l+1})$.
The case where $l=m-1$ is similar, except that the matrix units $(\nu_+,\nu_+)$, $(\nu_-,\nu_-)$, where $r(\nu_{\pm})=p_{\pm}$, are now given by the $SO(3)_{2m}$ Jones-Wenzl projections $\frac{1}{2} (F(\mathfrak{f}_m) \pm F(\mathfrak{t}))$.

Finally we consider the case $l=m$. Denote by $j_a:=r(a_1')=r(a_2')$, $j_b:=r(b_1')=r(b_2')$. We consider matrix units $(\xi_1,\xi_2) \in M_{m+1}$ where $r(\xi_i)=p_-$, the case for $r(\xi_i)=p_+$ is similar. Choosing $a,b$ such that $j_a,j_b \in \{m,p_+\}$ and $s(a_2')=s(b_1')=m$, equation (\ref{eqn:aU_ib}) yields (up to a scalar factor)
$$\sum_{s \in \{ p_{\pm} \}} W_{j,j_a,s} \overline{W_{j,j_b,s}} (a_1 \cdot a_1' \cdot \zeta_s, b_2 \cdot b_2' \cdot \eta_s) + \text{ terms in } B_m.$$
When $j_a=p_+$ or $j_b=p_+$, the sum has only one term, $W_{j,j_a,p_-} \overline{W_{j,j_b,p_-}} (a_1 \cdot a_1' \cdot \zeta_{p_-}, b_2 \cdot b_2' \cdot \eta_{p_-})$, which is therefore in $B_{m+1}$. The remaining matrix units to consider are $(\xi_1,\xi_2) \in M_{m+1}$ where the final edge $\eta_i$ in the path $\xi_i$ is the edge from $m$ to $p_+$ for both $i=1,2$. In this case $(\xi_1,\xi_2) \not\in B_{m+1}$, but is instead given by a product $(\xi_1,\nu)\cdot(\nu,\xi_2)$, where $\nu$ is a path whose final edge $\eta$ is the edge from $p_-$ to $p_+$ and thus $(\xi_1,\nu),(\nu,\xi_2) \in B_{m+1}$.
\end{proof}

Define $M^s_l := M_{l+s}$, $s \geq 0$, where $M^0_l := M_l$, and let $M^s := \bigvee_l M^s_l$.
We have periodic sequences $\{ M^s_l \subset M^{s+1}_l \}_l$ of commuting squares of period 1 in the sense of Wenzl in \cite{wenzl:1988}, where having period 1 means that, for $l$ large enough, the Bratteli diagrams for the inclusions $M^s_l \subset M^s_{l+1}$ are the same, and the Bratteli diagrams for the inclusions $M^s_l \subset M^{s+1}_l$ are the same.
Note that $M^{s-1}_l \subset M^s_l \subset M^{s+1}_l$ are basic constructions for large enough $l$ \cite[Lemma 11.8]{evans/kawahigashi:1998}, and hence $M \subset M^1 \subset M^2 \subset \cdots \;$ is the Jones tower of $M \subset M^1$.
By Lemma \ref{Lemma:algAk=algEU} $M^s_l$ is generated by $1,e_j,u_j$ ($j=1,\ldots,l+s-1$) and $F(\mathfrak{t})$ (for $l+s \geq m$).
Consider an element $(\sigma_1 \cdot \alpha_1 \cdot \alpha_1', \sigma_2 \cdot \alpha_2 \cdot \alpha_2') \in M^s_l$, where $|\sigma_i| = s+l-2$ and $|\alpha_i| = |\alpha_i'| = 1$ such that $(\sigma_1, \sigma_2) \in M^{s-1}_{l-1}$ and $(\sigma_1 \cdot \alpha_1, \sigma_2 \cdot \alpha_2) \in M^{s-1}_l$. The change of basis $(\sigma_1 \cdot \alpha_1 \cdot \alpha_1', \sigma_2 \cdot \alpha_2 \cdot \alpha_2') \mapsto \sum_{\beta_i,\beta_i'} X^{\alpha_1, \alpha_1'}_{\beta_1, \beta_1'} \overline{X^{\alpha_2, \alpha_2'}_{\beta_2, \beta_2'}} (\sigma_1 \cdot \beta_1 \cdot \beta_1', \sigma_2 \cdot \beta_2 \cdot \beta_2')$ is given by a connection $X=(X^{\alpha_1, \alpha_1'}_{\beta_1, \beta_1'})^{\alpha_1, \alpha_1'}_{\beta_1, \beta_1'}$. We may take $X$ to be given by applying the functor $F$ to the crossing \eqref{eqn:braiding-SO3}, and since the operators $e_j,u_j$ and $F(\mathfrak{t})$ are the images under $F$ of diagrams in $\mathcal{C}^m$, they do not change their form by changes of basis. Thus the connection $X$ is flat
and, by Ocneanu's compactness argument \cite{ocneanu:1991}, the higher relative commutants of the subfactor $M \subset M^1$  are given by $M' \cap M^l = M_l$.
Thus the subfactor $M \subset M^1$ has principal graph given by the bipartite unfolding of $\mathcal{A}_{2m}$. Subfactors with these principal graphs appear in \cite[Fig. 15]{izumi:1991}.

\subsection{$SO(3)$ Goodman-de la Harpe-Jones subfactors} \label{sect:GHJ}

Previous methods used to construct the Goodman-de la Harpe-Jones subfactors for the $E_7$ and $D_{\text{odd}}$ Dynkin diagrams in \cite[Lemma A.1]{bockenhauer/evans/kawahigashi:2000} and the $SU(3)$-analogue in \cite{evans/pugh:2009ii} provide braided subfactors which realise the $SO(3)$ modular invariants classified in Section \ref{sect:mod-inv-nimreps}.

The first row of the intertwining matrix $V$ for these subfactors, which is a matrix whose rows are indexed by the vertices of $\mathcal{G}$ and columns are indexed by the vertices of $\mathcal{A}_{2m}$, such that $V \Delta_{\mathcal{A}} = \Delta_{\mathcal{G}} V$, yields the algebra object $A$ for each module category. For each $SO(3)_{2m}$ graph we have the following algebra objects:
\begin{center}
\begin{tabular}{|c|c|}
\hline nimrep $\mathcal{G}$  & algebra object $A$ \\
\hline
\hline  $\mathcal{A}_{2m}$    & $f_0$         \\
\hline  $\sigma_{2m}$         & $f_0 \oplus f_1$       \\
\hline  $\mathcal{E}_8$       & $f_0 \oplus Q_+$ or $f_0 \oplus Q_-$           \\
\hline  $\mathcal{E}_8^c$     & $f_0 \oplus f_3$           \\
\hline  $\mathcal{E}_{14}$    & $f_0 \oplus f_5$          \\
\hline  $\mathcal{E}_{14}^c$  & $f_0 \oplus f_1 \oplus f_4 \oplus 2 f_5 \oplus f_6$       \\
\hline
\end{tabular}
\end{center}
For level $8$, these recover the algebra objects found in \cite{edie-michell:2017}. Each algebra object has a unique algebra object structure, apart from $f_0 \oplus f_3$, corresponing to $\mathcal{E}_8^c$ for which there are two algebra object structures, as found in \cite{edie-michell:2017}.

\section{An $SO(3)$ generalised preprojective algebra} \label{sect:SO3_preprojective}

Any module category for the Temperley-Lieb categories $TL(\delta)$ or $TL^{(4m)}$ defined in Section \ref{sect:TL-cat} yields a preprojective algebra $\Pi$ via $\Pi = \mathbb{C}\mathcal{G}/\langle \mathrm{Im} \left( F \left( \includegraphics[width=5mm]{fig-creation} \right) \right) \rangle$, where $F$ is the tensor functor defined as in \eqref{eqn:functorF} and $\langle \mathrm{Im} \left( F \left( \includegraphics[width=5mm]{fig-creation} \right) \right) \rangle \subset \mathbb{C}\mathcal{G}$ is the two-sided ideal generated by the image of the creation operators $\includegraphics[width=5mm]{fig-creation}$ in $\mathbb{C}\mathcal{G}$. It is a graded algebra, with $p^{\mathrm{th}}$ graded part $\Pi_p = (\mathbb{C}\mathcal{G})_p/\langle \mathrm{Im} \left( F \left( \includegraphics[width=5mm]{fig-creation} \right) \right) \rangle_p$, where $\langle \mathrm{Im} \left( F \left( \includegraphics[width=5mm]{fig-creation} \right) \right) \rangle_p$ is the restriction of $\langle \mathrm{Im} \left( F \left( \includegraphics[width=5mm]{fig-creation} \right) \right) \rangle$ to $(\mathbb{C}\mathcal{G})_p$ and is equal to $\sum_{i=1}^{p-1}\mathrm{Im}(e_i)$, the linear span of the images in $\mathbb{C}\mathcal{G}_p$ of the Temperley-Lieb elements $e_i$.

These preprojective algebras have an alternative description as \cite{cooper:2007}
$$\Sigma = \bigoplus_{j=0}^{\infty} F(\rho^{(j)}),$$
in the generic case, and in the non-generic case
$$\Sigma = \bigoplus_{j=0}^{4m} F(\rho^{(j)}),$$
where the Jones-Wenzl projections $\rho^{(j)}$ are the simple objects in $TL(\delta)$, $TL^{(4m)}$ respectively. Multiplication $\mu$ is defined by $\mu_{p,l} = F(\mathfrak{f}^{(p+l)}): \Sigma_p \otimes_R \Sigma_l \rightarrow \Sigma_{p+l}$.
In either case the preprojective algebra is given by applying $F$ to the direct sum of all simple objects in the tensor category. The module categories for $TL^{(4m)}$ yield precisely the finite dimensional preprojective algebras. They are in fact Frobenius algebras, that is, there is a linear function $\zeta:\Sigma \rightarrow \mathbb{C}$ such that $(x,y):=\zeta(xy)$ is a non-degenerate bilinear form (this is equivalent to the statement that $\Sigma$ is isomorphic to its dual $\Sigma^{\ast} = \mathrm{Hom}(\Sigma,\mathbb{C})$ as left (or right) $\Sigma$-modules). There is an automorphism $\beta$ of $\Sigma$, called the Nakayama automorphism of $\Sigma$ (associated to $\zeta$), such that $(x,y) = (y,\beta(x))$. Then there is an $\Sigma$-$\Sigma$ bimodule isomorphism $\Sigma^{\ast} \rightarrow {}_1 \Sigma_{\beta}$ \cite{yamagata:1996}, where the $\Sigma$-$\Sigma$ bimodule ${}_1 \Sigma_{\beta}$ is identified with $\Sigma$ as a vector space, but with the right action of $\Sigma$ twisted by the automorphism $\beta$.

In this section we will construct a generalisation of the preprojective algebra related to $SO(3)$. Let $\mathcal{M}$ be an $SO(3)_{2m}$ module category with nimrep graph $\mathcal{G}$, and denote by $F$ the corresponding module functor.
Analogous to the case of $SU(2)$ and $SU(3)$ \cite{evans/pugh:2010ii}, we define $\Pi^{SO(3)}$ to be the graded algebra with $p^{\mathrm{th}}$ graded part $\Pi^{SO(3)}_p = (\mathbb{C}\mathcal{G})_p/J_p$, where $J_p = \sum_{i=1}^{p-1}\mathrm{Im}(e_i) + \mathrm{Im}(u_i)$, the linear span of the images in $\mathbb{C}\mathcal{G}_p$ of the $SO(3)$-Temperley-Lieb elements $e_i$, $u_i$.
The algebra $\Pi^{SO(3)}$ is isomorphic to the graded algebra
$$A := \bigoplus_{j=0}^{2m} F(\rho_{j}),$$
where $F(\rho_{m+j}) \cong F(\rho_{m-j})$, $j=1,2,\ldots,m$ (c.f. Figure \ref{fig-isomorphism-fk-j_fk+j}), with $F(\rho_{m-j}) \subset (\mathbb{C}\mathcal{G})_{m-j}$ and $F(\rho_{m+j}) \subset (\mathbb{C}\mathcal{G})_{m+j}$.

The even part $TL^{(4m)}_{\mathrm{even}}$ of $TL^{(4m)}$ has simple objects $\rho^{(0)}, \rho^{(2)}, \ldots, \rho^{(4m)}$.
Since the arguments of Sections \ref{sect:mod-inv-nimreps}, \ref{sect:cell-systems} (apart from Section \ref{sect:T-cell-system}) did not use any $t$-boxes, the same arguments show that the $\ast$-module categories for $TL^{(4m)}_{\mathrm{even}}$ are completely classified by a nimrep, which is one of $\mathcal{H}^{e}$, $\mathcal{H}^{o}$, the even, odd part respectively of the Dynkin diagram $\mathcal{H} \in \{ D_{2m+2}, E_7, E_8\}$.
Now $A = \bigoplus_{j=0}^{2m} F(\varphi(\rho^{(2j)})) = \bigoplus_{j=0}^{2m} F'(\rho^{(2j)})$, where $F' = F \circ \varphi$ is a tensor functor from the even part $TL^{(4m)}_{\mathrm{even}}$ of $TL^{(4m)}$ to $\text{Fun}(\mathcal{M},\mathcal{M})$, with nimrep either $\mathcal{H}^e$ or $\mathcal{H}^o$, as in Theorem \ref{Thm:classification}.
Let $F^e$, $F^o$ be $TL^{(4m)}_{\mathrm{even}}$ module functors with nimreps $\mathcal{H}^e$, $\mathcal{H}^o$ respectively, and $\tilde{F}$ a $TL^{(4m)}$ module functor with nimrep $\mathcal{H}$. Then $\bigoplus_{j=0}^{2m} F^e(\rho^{(2j)}) \oplus \bigoplus_{j=0}^{2m} F^o(\rho^{(2j)}) \cong \bigoplus_{j=0}^{2m} \tilde{F}(\rho^{(2j)})$, the even-graded part of the preprojective algebra $\Sigma$ for $\mathcal{H}$.
Thus the Nakayama automorphism of $A$ for $\mathcal{H}^{e}$ or $\mathcal{H}^{o}$ is given by the Nakayama automorphism of $\Sigma$ for $\mathcal{H}$. The graphs $D_{2m+2}, E_7, E_8$ are precisely those for which the Nakayama automorphism of $\Sigma$ is trivial. Thus the Nakayama automorphism for $A$ is trivial for all $SO(3)_{2m}$ graphs.

\begin{Rem}
Note that the definition of preprojective algebra $\Pi$ used here is equivalent to the original one of \cite{gelfand/ponomarev:1979}, and uses the unoriented graph $\mathcal{G}$. Here $\Pi$ is the quotient of $\bbC \mathcal{G}$ by the ideal generated by the quadratic relations (one for each vertex $v$ of $\mathcal{G}$) $\sum_a a \tilde{a}$, where the sum is over all edges $a$ of $\mathcal{G}$ such that $s(a)=v$. A common definition in the literature uses instead a quiver $\mathcal{Q}$ obtained by choosing an orientation on $\mathcal{G}$, and use the quadratic relations $\sum_a \epsilon_a a \tilde{a}$, where the sum is again over all edges $a$ of $\mathcal{G}$ such that $s(a)=v$, and where $\epsilon_a$ is 1 if the edge $a$ is in $\mathcal{Q}$ and $-1$ if the edge $\tilde{a}$ is in $\mathcal{Q}$. Both definitions yield isomorphic algebras \cite{malkin/ostrik/vybornov:2006}. In the quiver definition, the Nakayama automorphism is never trivial, although the underlying permutation is always trivial for the graphs $D_{2m+2}, E_7, E_8$.
\end{Rem}

In the case of $SU(2)$ and $SU(3)$ the fusion rules yield a finite projective resolution of the corresponding (generalised) preprojective algebra. However, this does not happen for $SO(3)$:
From the fusion rules \eqref{eqn:fusion_rules-f1} for $\mathcal{C}^m$ we obtain the following sequence in $\mathcal{C}^m$ which is split exact by construction:
\begin{equation*} \label{eqn:exact_seq-f}
0 \to \rho_{j-1} \oplus \rho_j \stackrel{\psi_1^{\ast}+\psi_2^{\ast}}{\to} \rho_j \otimes \rho_1 \stackrel{\psi_3}{\to} \rho_{j+1} \to 0,
\end{equation*}
for $j=1,2,\ldots,2m-1$. For $j=0, 2m$, we have
\begin{equation*} \label{eqn:exact_seq-f-2}
0 \to \rho_0 \otimes \rho_1 \stackrel{\psi_3}{\to} \rho_{1} \to 0, \qquad\qquad 0 \to \rho_{2m-1} \stackrel{\psi_1^{\ast}}{\to} \rho_{2m} \otimes \rho_1 \to 0.
\end{equation*}
Summing over all $j=0,1,2,\ldots,2m$, we have the exact sequence
\begin{equation*}
0 \to \bigoplus_{j=0}^{2m-1} \rho_j[2] \oplus \bigoplus_{j=1}^{2m-1} \rho_j[1] \to \bigoplus_{j=0}^{2m} \rho_j \otimes \rho_1 \to \bigoplus_{j=1}^{2m} \rho_j \to 0,
\end{equation*}
where $B[m]$ denotes the space $B$ with grading shifted by $m$. From this we obtain the exact sequence
\begin{equation} \label{eqn:exact_seq-f-sum}
0 \to \bigoplus_{j=0}^{2m-1} \rho_j[2] \oplus \bigoplus_{j=1}^{2m} \rho_j[1] \to \left( \bigoplus_{j=0}^{2m} \rho_j \otimes \rho_1 \right) \oplus \rho_{2m}[1] \to \bigoplus_{j=0}^{2m} \rho_j \to \rho_0 \to 0.
\end{equation}
Similarly, by considering the exact sequence
\begin{equation*} \label{eqn:exact_seq-f-3}
0 \to \rho_{j-1} \stackrel{\psi_1^{\ast}}{\to} \rho_j \otimes \rho_1 \stackrel{\psi_2 \oplus \psi_3}{\to} \rho_j \oplus \rho_{j+1} \to 0,
\end{equation*}
we obtain the following exact sequence
\begin{equation} \label{eqn:exact_seq-f-sum-2}
0 \to \rho_{2m}[3] \to \bigoplus_{j=0}^{2m} \rho_j[3] \to \rho_0[2] \oplus \left( \bigoplus_{j=0}^{2m} \rho_j \otimes \rho_1 \right)[1] \to \bigoplus_{j=0}^{2m-1} \rho_j[2] \oplus \bigoplus_{j=1}^{2m} \rho_j[1] \to 0.
\end{equation}
Splicing the sequences \eqref{eqn:exact_seq-f-sum}, \eqref{eqn:exact_seq-f-sum-2} together we obtain the following exact sequence
\begin{align}
0 \to \rho_{2m}[3] \to \bigoplus_{j=0}^{2m} \rho_j[3] \to \rho_0[2] \oplus \left( \bigoplus_{j=0}^{2m} \rho_j \otimes \rho_1 \right)[1] \to \left( \bigoplus_{j=0}^{2m} \rho_j \otimes \rho_1 \right) \oplus \rho_{2m}[1] \nonumber \\
\to \bigoplus_{j=0}^{2m} \rho_j \to \rho_0 \to 0. \label{eqn:exact_seq-f-spliced}
\end{align}
Applying the functor $F$ to \eqref{eqn:exact_seq-f-spliced} we obtain the following finite resolution of $A$ as a left $A$-module
\begin{equation} \label{eqn:resolution-A}
0 \to R[2m+3] \stackrel{\mu_4}{\to} A[3] \stackrel{\mu_3}{\to} R[2] \oplus \left( A \otimes_R V[1] \right) \stackrel{\mu_2}{\to} \left( A \otimes_R V \right) \oplus R[2m+1] \stackrel{\mu_1}{\to} A \stackrel{\mu_0}{\to} R \to 0,
\end{equation}
where $V := (\bbC \mathcal{G})_1$ is the $A$-$A$ bimodule generated by the edges of $\mathcal{G}$.
The connecting maps are given explicitly by
\begin{align*}
\mu_0(x) &= \begin{cases} x, & \text{if } x \in R, \\ 0, & \text{otherwise.} \end{cases} \\
\mu_1((x \otimes a) \oplus 0) &= xa, \text{ for } x \in A_j, j < 2m, \qquad \mu_1((x \otimes a) \oplus v) = xa, \text{ for } x \in A_{2m} \\
\mu_2(v \oplus (s(a) \otimes a)) &= \left( \sum_{b:s(b)=v} \frac{\sqrt{\phi_{r(b)}}}{\sqrt{\phi_v}} b \otimes b' + \frac{1}{\sqrt{\phi_{s(a)} \phi_{r(a)}}} \sum_{b,b'} W(abb') x'a' \tilde{b} \otimes \tilde{b'} \right) \oplus 0, \\
\mu_2(0 \oplus (x'a' \otimes a)) &= \frac{1}{\sqrt{\phi_{s(a)} \phi_{r(a)}}} \left( \sum_{b,b'} W(abb') x'a'\tilde{b} \otimes \tilde{b'} + \sum_{c,c'} \frac{\sqrt{\phi_{r(c')}}}{\sqrt{\phi_{s(c')}}} W(a'ab) x' \tilde{c}c' \otimes \tilde{c'} \right) \\
& \quad \oplus 0, \\
\mu_3(x) &= 0 \oplus \left( \sum_b \frac{\sqrt{\phi_{r(b)}}}{\sqrt{\phi_{s(b)}}} xb \otimes \tilde{b} \right), \\
\mu_4(v) &= v,
\end{align*}
where $x, x' \in A$ with $\mathrm{deg}(x') < 2m$, $a, a'$ are edges of $\mathcal{G}$, the summations are over edges $b,b',c,c'$ of $\mathcal{G}$, and $v$ is a vertex of $\mathcal{G}$.

Let $H_A$ denote the Hilbert series of $A$, given by $H_A(t) = \sum_{p=0}^{\infty} H_{ji}^p t^p$, where the $H_{ji}^p$ are matrices which count the dimension of the subspace $\{ i x j | \; x \in A_p \}$, for $i,j$ vertices of $\mathcal{G}$.
From the exact sequence \eqref{eqn:resolution-A} we obtain the following relation for $H_A$:
$$-I t^{2m+3} + H_A(t)t^3 - \left( I t^2 + \Delta_{\mathcal{G}} H_A(t) t^2 \right) + \left( \Delta_{\mathcal{G}} H_A(t) t + I t^{2m+1} \right) - H_A(t) + I =0$$
which yields the following Hilbert series of the $SO(3)$-preprojective algebra $A$:
\begin{equation*}
H_A(t) = \frac{1-t^2+t^{2m+1}-t^{2m+3}}{1-\Delta_{\mathcal{G}} (t-t^2) - t^3} = \frac{(1+t)(1+t^{2m+1})}{1+(\Delta_{\mathcal{G}}-I)t+t^2}.
\end{equation*}

\end{document}